\documentclass[12pt]{amsart}

\usepackage[latin1]{inputenc}

\usepackage{color}

\usepackage{hyperref}

\hypersetup{colorlinks=true}

\usepackage{amsmath,amsfonts,amsthm,amssymb,mathtools} 

\usepackage[margin=2.5cm,bottom=3cm,top=3cm,footskip=1.5cm]{geometry}
 
\newcommand\N{{\mathbb N}} \newcommand\R{{\mathbb R}}
\newcommand\Z{{\mathbb Z}} \newtheorem{thm}{Theorem}

\newtheorem{rmq}{Remark}[section] \newtheorem{lemma}{Lemma} 
\newtheorem{cor}{Corollary} \newtheorem{prop}{Proposition}

\newcommand\e{{\varepsilon}}
\newcommand\E{{\mathcal{E}}}

\begin{document}

 \title{Strichartz estimates for the wave equation on a 2D model convex domain}

\author{Oana Ivanovici}
\address{Sorbonne Université, CNRS, Laboratoire Jacques-Louis Lions, LJLL, F-75005 Paris, France} \email{oana.ivanovici@sorbonne-universite.fr}

  \author{Gilles Lebeau}
  \address{Universit\'e Côte d'Azur, CNRS, Laboratoire JAD, France} \email{gilles.lebeau@univ-cotedazur.fr}

  \author{Fabrice Planchon}
  \address{ Sorbonne Université, CNRS, Institut Mathématique de Jussieu-Paris Rive Gauche, IMJ-PRG, F-75005 Paris, France}
  \email{fabrice.planchon@sorbonne-universite.fr}

      \thanks{
    \\
 O.Ivanovici and F. Planchon were supported by ERC grant ANADEL 757 996.
    }

\begin{abstract} We prove better Strichartz type estimates than expected from the (optimal) dispersion we obtained earlier on a 2d convex model domain. This follows from taking full advantage of the  space-time localization of caustics in the parametrix, despite their number increasing like the inverse square root of the distance from the source to the boundary. As a consequence, we improve known Strichartz estimates for the wave equation. Several improvements on our previous parametrix construction are obtained along the way and are of independent interest for further applications.
\end{abstract}
\maketitle

\section{Introduction and main results}

Let us consider the wave equation on a domain $\Omega$ with boundary
$\partial \Omega$ ,
\begin{equation} \label{WE} 
\left\{ \begin{array}{l} (\partial^2_t-
\Delta) u(t, x)=0, \;\; x\in \Omega \\ u|_{t=0} = u_0 \; \partial_t
u|_{t=0}=u_1,\\ Bu=0,\quad x\in \partial\Omega.
 \end{array} \right.
 \end{equation}
Here,  $\Delta$ usually stands for the Laplace-Beltrami operator on
$\Omega$. If $\partial\Omega\neq \emptyset$, the boundary condition
could be either Dirichlet ($B$ is the identity map:
$u|_{\partial\Omega}=0$) or Neumann ($B=\partial_{\nu}$ where $\nu$ is
the unit normal to the boundary.)

Solutions to the wave equation on a smooth manifold without boundaries are known to disperse. In particular, on average the wave amplitude is decaying faster than predicted by Sobolev embedding Theorem from energy estimates, and this yields so-called Strichartz estimates. In turn these estimates have been crucial in dealing with a large range of problems, both linear and nonlinear. They are also closely related to localization and decay of eigenfunctions in the compact case (through square function estimates for the wave equation). 

Let us be more specific and introduce notations: let $d$ be the spatial dimension of $\Omega$ and $\beta=d\Big(\frac 12-\frac 1r\Big)-\frac 1q$, with $(q,r)$ be an admissible pair,
\begin{equation}\label{adm} 
\frac 1q\leq \frac{(d-1)} 2 (\frac{1}{2}-\frac{1}{r}),\quad  q>2.
\end{equation}
Then, when $\Omega$ is a Riemannian manifold with empty boundary, the solution to \eqref{WE} is such that, at least for a suitable $T<+\infty$ depending only on $\Omega$, uniformly for $0<h<1$, 
\begin{equation}\label{stricrd} h^{\beta}\|\chi(hD_t)u\|_{L^q([0,T],
L^r)}\leq C h \Big( \| u_{0}\|_{H^{1}}+\| u_{1}\|_{L^2}\Big),
\end{equation} 
where $\chi\in C^{\infty}_0$ is a smooth truncation in a neighborhood of $1$, $D=-i\partial$ and $H^{1}$ is the Sobolev space associated with $\Delta$ on $\Omega$. One may often state \eqref{stricrd} differently by removing the spectral cut-off $\chi$, removing all $h$ factors and replacing $(H^{1},L^{2})$ by $(H^{\beta},H^{\beta-1})$ for the data on the righthand side, at the expense of introducing fractional Sobolev spaces on $\Omega$. When \eqref{stricrd} holds for $T=\infty$, it is said to be a global in time Strichartz estimate. For $\Omega=\mathbb{R}^d$ with flat metric, the solution  $u_{\mathbb{R}^d}(t,x)$ to
\eqref{WE} with data $(u_0=\delta_{x_{0}}, u_1=0)$ is known to be 
 \[ u_{\mathbb{R}^d}(t,x)=\frac{1}{(2\pi)^d}\int
\cos(t|\xi|)e^{i(x-x_{0})\xi}d\xi
 \]
and by stationary phase the classical dispersion estimate follows:
\begin{equation}\label{disprd}
\|\chi(hD_t)u_{\mathbb{R}^d}(t,.)\|_{L^{\infty}(\mathbb{R}^d)}\leq
C(d) h^{-d}\min\{1, (h/t)^{\frac{d-1}{2}}\}.
\end{equation}
Interpolation between  \eqref{disprd} and energy estimates, together with a duality argument, routinely provides
\eqref{stricrd}. On any boundaryless Riemannian manifold $(\Omega,g)$ one may follow the same path, replacing the exact formula by a parametrix (which may be constructed locally within a small ball, thanks to finite speed of propagation.)

On a manifold with boundary, one may no longer think that light rays are slightly distorted straight lines. There may be rays gliding along a convex part of the boundary, rays grazing a convex obstacle or combinations of both. Strichartz estimates outside a strictly convex obstacle were obtained in \cite{smso95} and turned out to be similar to the free case (see \cite{ildispext} for the more complicated case of the dispersion). Strichartz estimates with losses were obtained later on general domains, \cite{blsmso08}, using short time parametrices constructions from \cite{smso06}, which in turn were inspired by works on low regularity metrics \cite{tat02}. Losses in \cite{blsmso08} are induced by considering only time intervals that allow for no more than one reflection of a given wave packet and as such, one does not see the full effect of dispersion in the tangential directions. 

In our work \cite{Annals}, a parametrix for the wave equation inside a model of strictly convex domain was constructed that provided optimal decay estimates, uniformly with respect to the distance of the source to the boundary, over a time length of constant size. This involves dealing with an arbitrarily large number of caustics and retain control of their order. Our sharp dispersion estimate immediately yields by the usual argument Strichartz estimates with a range of pairs $(q,r)$ such that
\begin{equation}\label{adm-1/4} 
\frac 1q\leq \left(\frac{(d-1)} 2 -\frac 1 4\right)\left(\frac{1}{2}-\frac{1}{r}\right),\quad  q>2\,
\end{equation}
where, informally, the new $1/4$ factor is related to the $1/4$ loss in the dispersion estimate. On the other hand, earlier works \cite{doi}, \cite{doi2} proved that Strichartz estimates inside strictly convex domains of dimension $d\geq 2$ can hold only if $(q,r)$ are such that
\begin{equation}\label{adm-1/4bis} 
\frac 1q\leq \frac{(d-1)} 2 \left(\frac{1}{2}-\frac{1}{r}\right)-\frac 1 {6}\left(\frac{1}{4}-\frac{1}{r}\right),\quad  q>2\,.
\end{equation}
This provides a counterexample for $r>4$ and restricts the dimension to $d\leq 4$. In a recent work \cite{ILP4}, we improved the 2D counterexample, with \eqref{adm-1/4bis} replaced by a stronger requirement:
\begin{equation}\label{adm-1/10} 
\frac 1q\leq \left(\frac{1} 2-\frac 1 {10}\right)\left(\frac{1}{2}-\frac{1}{r}\right)\,,
\end{equation}
which, for $r=+\infty$, yields $q\geq 5$ (to be compared to $q\geq 24/5$ in \eqref{adm-1/4bis}.)

Our purpose now is to improve upon the positive results, in dimension $d=2$. In particular, for suitable micro-localized solutions we close the gap with the recent counterexample from \cite{ILP4}, providing a
near complete picture when $x+ |D_{x}D_{y}^{-1}|^{2}\sim h^{1/3}$, $|D_{y}|\sim h^{-1}$. Before stating our main result, we start by describing the convex model domain under consideration: the Friedlander model domain is the half-space, for $d\geq 2$,
 $\Omega_d=\{(x,y)| x>0, y\in\mathbb{R}^{d-1}\}$ with the metric $g_F$ inherited from the following Laplace operator,
 $\Delta_F=\partial^2_x+(1+x)\Delta_{\mathbb{R}^{d-1}_y}$. 

 \begin{rmq}
For the metric $g=dx^2+(1+x)^{-1}dy^2$, the Laplace-Beltrami operator is
$\triangle_{g}=(1+x)^{1/2}\partial_{x}(1+x)^{-1/2}\partial_{x}+(1+x)\Delta_{y}$
which is self adjoint with volume form $(1+x)^{-1/2}dxdy$. Our Friedlander model uses instead
the Laplace operator associated to the Dirichlet form $\int \vert
\nabla_{g}u\vert^2\,dxdy
=\int (\vert \partial_{x}u\vert^2+(1+x)\vert\partial_{y}u\vert^2)dxdy$, which is self adjoint with volume form $dxdy$. The Friedlander operator $\Delta_{F}$ allows for explicit computations and the difference $\Delta_{g}-\Delta_{F}= -(2(1+x))^{-1}\partial_{x}$ is a first order differential operator: as such, as long as we are dealing with local in time Strichartz estimates for data close to the boundary, it may be treated as a lower order perturbative term and proving local in time Strichartz estimates for $\Delta_{F}$ implies the same set of estimates for $\Delta_{g}$. Moreover, $(\Omega_d,g_F)$ is easily seen to model a strict convex domain, as a first order approximation of the unit disk $D(0,1)$ in polar coordinates $(r,\theta)$, close to the boundary $r=1$: set $r=1-x/2$, $\theta=y$.
\end{rmq}

\begin{thm} \label{thm+1} Strichartz estimates \eqref{stricrd} hold true for the wave equation \eqref{WE} on the domain $(\Omega_{2},g_{F})$, with $\Delta=\Delta_{F}$, for pairs $(q,r)$ such that
\begin{equation}\label{adm-1/6} 
\frac 1q\leq  \left(\frac{1} 2 -\gamma(2)\right)\left(\frac{1}{2}-\frac{1}{r}\right),\quad\text{ with }\quad \gamma(2)=\frac 1 {9}\,.
\end{equation}
In particular, for $r=+\infty$, we have $q\geq 5+1/7$.
\end{thm}
This result improves on known results for strictly convex domains in 2D: for $d=2$, \cite{blsmso08} obtained $\gamma(2)=1/6$ (but for any boundary), while \cite{Annals} only provide the weaker $\gamma(2)=1/4$. Note that for $d\geq 3$, \cite{Annals} had a uniform $\gamma(d)=1/4$, which already improved on \cite{blsmso08}, where $\gamma(3)=2/3$ and $\gamma(d)=(d-3)/2$ for $d\geq 4$ (but again, without restrictions on the boundary). We deliberately chose to restrict to 2D, in order to avoid another layer of technicalities. However, Theorem \ref{thm+1} generalizes to $d\geq 3$ with at least $\gamma(d)\leq 1/6$. This will be dealt with elsewhere (preliminary results appeared in \cite{ILPTunisie}). Moreover, Theorem \ref{thm+1} will extend to a generic strictly convex domain, building upon \cite{ILLP} and using the present work as a blueprint; recently \cite{ILLP} obtained $\gamma(d)> (1-1/d)/4$ for any strictly convex domain, using a weaker version of our model 2D argument. Finally, we expect that the present parametrix construction and its counterpart from \cite{ILLP} to be building blocks for sharper versions of propagation of singularities theorems in the presence of boundaries. This in turn would have important applications to control theory that, to our knowledge, have remained out of reach for a long time.

The proof of Theorem \ref{thm+1} will rely on a complete make-over of the parametrix construction of \cite{Annals}: resulting bounds on the Green function will be improved in several directions together with refinements of estimates on gallery modes from \cite{doi}, all of which are of independent interest:
\begin{itemize}
\item After introducing a localization such that $x+D_{y}^{-2}D_{x}^{2}\sim \gamma\ll 1$, one may further restrict the wave operator to this region of phase space. The worst possible case regarding dispersion will then be $x\sim \gamma$ and $|D_{x}|\lesssim \sqrt \gamma |D_{y}|$, for $\gamma\gtrsim h^{1/3}$;
\item At such fixed $\gamma$, the parametrix construction from \cite{Annals} may be extended to initial data $\delta_{(x=a,y=0)}$ with $h^{2/3-\varepsilon}<\gamma$, for any $\varepsilon>0$ and any $a>0$, improving on the previous requirement $a>h^{4/7}$;
\item the degenerate stationary phase estimates in \cite{Annals} may be refined to isolate precisely the space-time location of the worst case scenario of a swallowtail singularity. It turns out that such singularities only happen at an exceptional, discrete set of times; we then suitably average over such exceptional times, at fixed $\gamma$, before recombining the resulting Strichartz estimates. The singular points in the Green function originating at $x=a$ in the $(x,t)$ plane are $(a,4N\sqrt a\sqrt{1+a})$, for $|N|\leq 1/\sqrt a$, which creates difficulties in averaging over $a>0$, as the usual argument is blind to improvements outside a small neighborhood of this set.
\item gallery modes satisfy the usual Strichartz estimates (as already proved in \cite{doi}) but with uniform constant with respect to the order of the mode: this allows to deal with the $\gamma<h^{2/3-\varepsilon}$ region.
\end{itemize}
These improvements proved to be crucial in recent works \cite{Iv2020KG} and \cite{Iv2020Sch}: refinements over the space-time localization of the degenerate stationary phase are indeed a key point to obtain semi-classical dispersion estimates in the Schrödinger case, where the methods of \cite{Annals} fail to yield improvements over previously known results.

In the remaining of the paper, $A\lesssim B$ means that there exists a constant $C$ such that $A\leq CB$ and this constant may change from line to line but is independent of all parameters. It will be explicit when (very occasionally) needed. Similarly, $A\sim B$ means both $A\lesssim B$ and $B\lesssim A$.
\section{The half-wave propagator: spectral analysis and parametrix construction}
Let $Ai$ denote the standard Airy function, we have $Ai(-z)=A_+(z)+A_-(z)$, where
\begin{equation}
  \label{eq:Apm}
  A_\pm(z)=e^{\mp i\pi/3} Ai(e^{\mp i\pi/3} z)\,,\,\,\text{ for } \,
  z\in \mathbb{C}\,,
\end{equation}
By definition, a function $f(w)$ admits an asymptotic expansion for $w\rightarrow 0$ when there exists a (unique) sequence $(c_{n})_{n}$ such that, for any $n$, $\lim_{w\rightarrow 0} w^{-(n+1)}(f(w)-\sum_{0}^{n} c_{n} w^{n})=c_{n+1}$. We will denote $ f(w)\sim_{w} \sum_{n} c_{n} w^{n}$. Then
\begin{equation}\label{eq:ApmAE}
A_{\pm}(z)=\Psi(e^{\mp i\pi/3} z)e^{\mp\frac 23 i z^{3/2} },\quad \Psi(z)\sim_{1/z} z^{-1/4}\sum_{j=0}^{\infty} a_j z^{-3j/2}, \quad a_0=\frac{1}{4\pi^{3/2}}.
\end{equation}
The following lemma (see \cite{ILP4} for a proof) will be crucial in the analysis of reflected phases :
\begin{lemma}\label{lemL}
Define, for $\omega \in \R$, $L(\omega)=\pi+i\log \frac{A_-(\omega)}{A_+(\omega)}$, then $L$ is real analytic and strictly increasing. We also
have
\begin{equation}
  \label{eq:propL}
  L(0)=\pi/3\,,\,\,\lim_{\omega\rightarrow -\infty} L(\omega)=0\,,\,\,
  L(\omega)=\frac 4 3 \omega^{\frac 3 2}+\frac{\pi}{2}-B(\omega^{\frac 3
    2})\,,\,\,\text{ for } \,\omega\geq 1\,,
\end{equation}
with $ B(u)\sim_{1/u} \sum_{k=1}^\infty b_k u^{-k}$ and $b_k\in\R$, $b_1> 0$. Finally, one may check that
\begin{equation}
  \label{eq:propL2}
  Ai(-\omega_k)=0 \iff L(\omega_k)=2\pi k \text{ and }
  L'(\omega_k)=2\pi \int_0^\infty Ai^2(x-\omega_k) \,dx\,,
\end{equation}
where here and thereafter, $\{-\omega_k\}_{k\geq 1}$ denote the zeros of the Airy function in decreasing order.
\end{lemma}
\subsection{Spectral analysis of the Friedlander model}\label{sectspectralloc}
Let $\Omega_2$ be the half-space $\{(x,y)\in\mathbb{R}^2|, x>0,y\in\mathbb{R}\}$ and consider the operator
$\Delta_F=\partial^{2}_{x}+(1+x)\partial^{2}_{y}$ on $\Omega_2$ with Dirichlet boundary condition. After a Fourier transform in the $y$ variable, the operator $-\Delta_{F}$ becomes
$-\partial^2_x+(1+x)\theta^2$. For  $\theta\neq 0$, this is a positive self-adjoint operator 
on $L^2(\mathbb{R}_+)$, with compact resolvent. 

\begin{lemma}\label{lemorthog} (\cite[Lemma 2]{ILP4})
There exist orthonormal eigenfunctions $\{e_k(x,\theta)\}_{k\geq 0}$ with their corresponding eigenvalues $\lambda_k(\theta)=|\theta|^2+\omega_k|\theta|^{4/3}$, which form a Hilbert basis of $L^{2}(\mathbb{R}_{+})$.  These eigenfunctions have an explicit form
\begin{equation}\label{eig_k}
 e_k(x,\theta)=\frac{\sqrt{2\pi}|\theta|^{1/3}}{\sqrt{L'(\omega_k)}}
Ai\Big(|\theta|^{2/3}x-\omega_k\Big),
\end{equation}
where $L'(\omega_k)$ is given by \eqref{eq:propL2}, which yields
$\|e_k(.,\theta)\|_{L^2(\mathbb{R}_+)}=1$.
\end{lemma}
In a classical way, for $a>0$, the Dirac distribution $\delta_{x=a}$ on $\mathbb{R}_+$ may be decomposed in terms of eigenfunctions $\{e_k\}_{k\geq 1}$: $ \delta_{x=a}=\sum_{k\geq 1} e_k(x,\theta)e_k(a,\theta)$. If we consider a data at time $t=s$ such that
$u_0(x,y)=\psi(hD_y)\delta_{x=a,y=b}$, where $h\in (0,1]$ is a small parameter and
 $\psi\in C^{\infty}_0([\frac 12,\frac 3 2])$, we can write the (localized in $\theta$) Green function associated to the half-wave operator on $\Omega_{2}$ as
\begin{equation}
\label{greenfct} G^{\pm}_{h}((x,y,t),(a,b,s))=\sum_{k\geq 1}
\int_{\mathbb{R}}e^{\pm i(t-s)\sqrt{\lambda_k(\theta)}}
e^{i(y-b)\theta}
 \psi(h\theta)e_k(x,\theta)e_k(a,\theta)d\theta\,.
\end{equation}
Notice that, in addition to the cut-off $\psi(h\theta)$, which localizes the Fourier variable dual to $y$, we may add a spectral cut-off $\psi_{1}(h\sqrt{\lambda_{k}(\theta)})$ under the $\theta$ integral, where $\psi_{1}$ is also such that $\psi_{1}\in  C^{\infty}_0([\frac 12,\frac 3 2])$. Indeed, 
$$
-\Delta_{F} ( \psi(h\theta) e^{i \theta y} e_{k}(x,\theta)) =\lambda_{k}(\theta) \psi(h\theta) e^{i \theta y} e_{k}(x,\theta)\,.
$$
As observed in \cite{Annals}, the significant part of the sum over $k$ in \eqref{greenfct}  becomes then a finite sum over $k\lesssim h^{-1}$, considering the asymptotic expansion of $\omega_{k}\sim k^{2/3}$. We now go further and localize with respect to $(x-\theta^{-2}\partial_{x}^{2})$: notice $(x- \theta^{-2} \partial_{x}^{2})e_{k}(x,\theta)=(\theta^{-2}\lambda_{k}(\theta)-1) e_{k}(x,\theta)$.
As such, introducing a new parameter $\gamma\lesssim 1$,  we add a factor $\psi_{2}((\omega_{k} \theta^{-2/3})/\gamma)$ which, considering the asymptotic expansion of the Airy zeros and the $\theta$ localization, is essentially $\psi_{2}( (kh)^{2/3})/\gamma)$. If $\psi_{2}\in C^{\infty}_{0}([-1,1])$, this allows to reduce the sum over $k$ in the definition of the Green function to $k$'s such that $k\lesssim \gamma^{3/2}/h$. In the following we actually choose $\psi_{2}\in C^{\infty}_{0}([\frac 12,\frac 3 2])$ (later on, it will become clear that for $\gamma\ll a$, the corresponding part of $G^{\pm}_{h}$ is irrelevant since the factor $e_k(a,\theta)$ in this case is exponentially decreasing) and set (rescaling the $\theta$ variable for later convenience)
\begin{multline}
\label{greenfctbis} G^{\pm}_{h,\gamma}((x,y,t),(a,b,s))  =  \sum_{k\geq 1}
\frac 1 h \int_{\mathbb{R}}e^{\pm i(t-s)\sqrt{\lambda_k(\eta/h)}}
e^{i(y-b)\eta/h}  \psi(\eta)\psi_{1}(h\sqrt{\lambda_{k}(\eta/h)})\\
{}\times\psi_{2}(h^{2/3}\omega_{k}/(\eta^{2/3}\gamma)) e_k(x,\eta/h)e_k(a,\eta/h)d\eta\,.
\end{multline}
Note that from an operator point of view, abusing notations, if $G^{\pm}$ is the half-wave propagator, we are in fact considering
$$
G^{\pm}_{h,\gamma}=\psi(h D_{y})\psi_{1}(-h\sqrt{-\Delta_{F}})\psi_{2}((x+ D_{x}^{2} (D_{y})^{-2})/\gamma) G^{\pm}\,.
$$
If $\gamma$ is taken to be in $(2^{-m})_{m\in \N}$ and $\psi_{2}$ chosen accordingly (e.g. $\psi_{2}(\xi)=\phi(\xi)-\phi(2\xi)$ with $\phi\in C^{\infty}_{0}[0,3/2)$ and $\phi=1$ on $[0,1]$), we also have $G^{\pm}_{h}=\sum_{\gamma} G^{\pm}_{h,\gamma}$: this sum over $\gamma$ is finite and restricted to $h^{2/3}\lesssim \gamma\lesssim 1$; the lower bound is induced by the knowledge of $-\omega_{1}<0$ and the upper bound follows from the restriction on $k$ placed by the $\eta$ and spectral localizations. Before proceeding we remark that we have
\begin{equation}
  \label{eq:bornesup}
\sup_{x,y,t,a,b,s}  |  G^{\pm}_{h,\gamma}|\lesssim \frac{\sqrt\gamma} {h^{2}}\,.
\end{equation}
Indeed, from \eqref{greenfctbis} we write, using \eqref{eig_k}, $L'(\omega)\sim 2\omega^{1/2}$ and $|\omega_{k}|\sim k^{2/3}$
\begin{align*}
 |G^{\pm}_{h,\gamma}(\cdots) |  & \lesssim   \sum_{1\leq k\lesssim \frac{\gamma^{3/2}}h}
\frac 1 h \int_{\mathbb{R}}\psi(\eta) \frac{|\eta|^{2/3}}{h^{2/3}{\sqrt \omega_k}}\left | Ai\bigl(|\eta/h|^{2/3}x-\omega_k\bigr) Ai\bigl(|\eta/h|^{2/3}a-\omega_k\bigl)\right|\,d\eta\\
    & \lesssim  \frac 1 {h^{5/3}} \int_{\mathbb{R}}\psi(\eta)\Bigl(\sup_{z\in \R}\sum_{1\leq k\lesssim \frac{\gamma^{3/2}}h}
{k^{-\frac 1 3}} \left | Ai\bigl(z-\omega_k\bigr)\right|^{2}\Bigr)\,d\eta\lesssim \frac 1 {h^{5/3}} \left(\frac{\gamma^{3/2}} h\right)^{\frac 1 3}
\end{align*}
using the asymptotics of the Airy function to evaluate the $\sup_{z}$ (see \cite{Annals}, Lemma 3.5).\qed

Successive spectral localizations restrict $G^{\pm}$ to directions of propagation where the tangential component dominates. As already observed in \cite{Annals}, other directions do see at most one reflection (on a suitable fixed time interval) and may be dealt with using already available arguments (e.g. \cite{blsmso08}). We will therefore ignore them from now on. Moreover, we will not only need $\gamma\lesssim 1$ but $\gamma<\gamma_{0}\ll 1$; We conservatively set $\gamma_{0}=1/100$ and this may be related to how we sort out directions. Since we are after fongible local in time estimates, we may restrict ourselves to a fixed small size neighborhood in space-time without loss of generality.

We briefly recall a variant of the Poisson summation formula that will be crucial to analyze the spectral sum defining $G^{\pm}_{h,\gamma}$, see again \cite{ILP4} for a short proof.
\begin{lemma}
  In $\mathcal{D}'(\R_\omega)$, one has $ \sum_{N\in \Z} e^{-i NL(\omega)}= 2\pi \sum_{k\in \N^*} \frac 1  {L'(\omega_k)}\delta(\omega-\omega_k)$, e.g. for $\phi\in C_{0}^{\infty}$,
  \begin{equation}
    \label{eq:AiryPoissonBis}
        \sum_{N\in \Z} \int e^{-i NL(\omega)} \phi(\omega)\,d\omega = 2\pi \sum_{k\in \N^*} \frac 1
    {L'(\omega_k)} \phi(\omega_k)\,.
  \end{equation}
\end{lemma}
\subsection{A parametrix construction}

From \eqref{greenfctbis} we consider the sum in $k$, without the $\eta$ integration, with $\hbar=h/\eta$ and expanding the eigenmodes
\begin{equation}
    \label{eq:ter46}
  v_{\hbar}(t,x)=\sum_{k\geq 1} 
\frac{\psi_{1}\left({h\over \hbar}\sqrt{1+\omega_{k} \hbar^{2\over 3}}\right)\psi_{2}(\hbar^{2\over 3}\omega_{k}/\gamma)}{\hbar^{2\over 3}L'(\omega_{k})}  Ai(\hbar^{-{2 \over 3}}x-\omega_{k}) Ai(\hbar^{-{2\over 3}}a-\omega_{k})   e^{i\frac t \hbar \sqrt {1+\omega_{k}\hbar^{2 \over 3}} } \,.
\end{equation}
Using the Airy-Poisson formula \eqref{eq:AiryPoissonBis}, we can transform the sum over $k$ into a sum over $N\in \Z$ as follows
\begin{multline}
  v_{\hbar}(t,x)= \frac 1 {2\pi} \sum_{N\in \Z} \int_{\R} e^{-i NL(\omega)} \hbar^{-{2\over 3}} e^{i\frac t \hbar \sqrt{1+\omega h^{{2\over 3}}}}
\chi_{1}(\omega) \psi_{1}\left({h \over \hbar}\sqrt{1+\hbar^{{2\over 3}}\omega}\right)\psi_{2}(\hbar^{\frac 2 3}\omega/\gamma)\\
{}\times Ai (\hbar^{-\frac 2 3} x-\omega) Ai(h^{-\frac 2 3}a-\omega)  \, d\omega\,.
\end{multline}
Here, $\chi_{1}(\omega)=1$ for $\omega>2$ and $\chi_{1}(\omega)=0$ for $\omega<1$, and obviously $\chi_{1}(\omega_{k})=1$ for all $k$, as $\omega_{1}>2$. Recall that
\begin{equation}
  \label{eq:bis47}
  Ai(\hbar^{-2/3} x-\omega)=\frac 1 {2\pi \hbar^{1/3} } \int e^{\frac i \hbar (\frac{\sigma^{3}}{3}+\sigma(x-\hbar^{2/3}\omega))} \,d\sigma\,.
\end{equation}
Rescaling $\omega$ with $\alpha=\hbar^{2/3} \omega$ yields
\begin{equation}
  \label{eq:bis48}
  v_{\hbar}(t,x)= \frac 1 {(2\pi)^{3}} \sum_{N\in \Z} \int_{\R}\int_{\R^{2}} e^{\frac i \hbar  \tilde \Phi_{N}(t,x,a,\sigma,s,\alpha)}  \hbar^{-2} 
  \chi_{1}(\hbar^{-2/3}\alpha)\psi_{1}(h\hbar^{-1}\sqrt{1+\alpha}) \psi_{2}(\alpha/\gamma)  \, ds d\sigma d\alpha\,,
\end{equation}
where
\begin{equation}
  \label{eq:bis49}
    \tilde \Phi_{N}(t,x,a,\sigma,s,\alpha,\eta)=\frac{\sigma^{3}} 3+\sigma(x-\alpha)+\frac {s^{3}} 3+s(a-\alpha)-N\hbar L(\hbar^{-2/3} \alpha)+t \sqrt{1+\alpha}\,.
\end{equation}
At this point, it is worth noticing that, as $h\hbar^{-1}=\eta$ and $\eta\in [\frac 1 2,\frac 3 2]$, we may drop the $\psi_{1}$ localization in \eqref{eq:bis48} by support considerations (slightly changing any cut-off support if necessary). Therefore, 
\begin{equation}
      \label{eq:bis48bis}
  G^{+}_{h,\gamma}((t,x,y),(a,0,0))= \frac 1 {(2\pi h)^{3}} \sum_{N\in \Z} \int_{\R^{2}}\int_{\R^{2}} e^{i \frac \eta h  (y+\tilde \Phi_{N})}  \eta^{2} \psi({\eta}) 
  \chi_{1}(\hbar^{-2/3}\alpha) \psi_{2}(\alpha/\gamma) \, ds d\sigma d\alpha d\eta\,.
\end{equation}
\section{The parametrix regime in 2D}
\label{sec:parametrix-regime}

\subsection{Localizing waves for $\gamma >h^{2/3-\varepsilon}$}
\label{sec:local-waves}

In \cite{Annals} and for $a\gg h^{4/7}$, a parametrix was carefully constructed by gluing together waves that were mostly located between two consecutive reflections. With a finite number of waves overlapping each other, the supremum of the sum became the supremum of each wave on its own support. In our setting, it is convenient to replace $a$ by the localization parameter $\gamma$. Moreover, we replace the parametrix from \cite{Annals} by the (exact) sum we introduced in the previous section. 

Let $\Phi_{N,a}(t,x,y,\sigma,s,\alpha,\eta):=y+ \tilde \Phi_{N}(t,x,a,\sigma,s,\alpha,\eta)$ with $ \tilde \Phi_{N}$ defined in \eqref{eq:bis49}, 
\begin{equation}
  \label{eq:bis499}
    \Phi_{N,a}=y+\frac{\sigma^{3}} 3+\sigma(x-\alpha)+\frac {s^{3}} 3+s(a-\alpha)-N\hbar L(\hbar^{-2/3} \alpha)+t \sqrt{1+\alpha}\,.
\end{equation}
We obtain
\begin{equation}
  \label{eq:bis488bis}
    G^{+}_{h,\gamma}((t,x,y),(a,0,0))= \frac 1 {(2\pi h)^{3}} \sum_{N\in \Z} V_{N,\gamma}(t,x,y),
\end{equation}
where we have set
\begin{equation}\label{defVNgamma}
V_{N,\gamma}(t,x,y):=\int_{\R^{2}}\int_{\R^{2}} e^{i \frac \eta h  \Phi_{N,a}}  \eta^{2} \psi({\eta}) \chi_{1}(\hbar^{-2/3}\alpha) \psi_{2}(\alpha/\gamma) \, ds d\sigma  d\alpha d\eta\,.
\end{equation}

Observe that $\chi_{1}$ and $\psi_{2}$ induce $\hbar^{2/3}\lesssim \alpha \sim \gamma$, which we assume from now on.
\begin{lemma}\label{lemmeNtpetit}
  At fixed $|t|\lesssim 1$, the sum defining $G^{+}_{h,\gamma}$ is only significant for $|N|\lesssim |t|\gamma^{-1/2}$, $\gamma>a/2$ and $x<2\gamma$, e.g. $\sum_{\sqrt \gamma |N| \not\in O(t)\cup \{x>2\gamma\}\cup \{a<2\gamma\}} V_{N,\gamma}$ is $O(h^{\infty})$.
\end{lemma}
We focus on the variables $\alpha$, $s$ and $\sigma$. Using the asymptotic expansion for $L(\omega)$, we find
$$
\partial_{\alpha}\Phi_{N,a}=\frac t {2\sqrt{1+\alpha}}-\sigma-s-2N\alpha^{1/2}(1-\frac 3 4 B'(\alpha^{3/2}/\hbar))
$$
where the $B'$ term is small compared to $1$, provided $\alpha^{3/2}/\hbar$ is sufficiently large ($>2$ is already enough). On the other hand, we have $\partial_{s}\Phi_{N,a}=s^{2}+a-\alpha$ and $\partial_{\sigma}\Phi_{N,a}=\sigma^{2}+x-\alpha$.
If either $|s|\geq 3 N^{\varepsilon} \gamma^{1/2}$ or $|\sigma|\geq 3 N^{\varepsilon} \gamma^{1/2}$, non-stationary phase in one of these variables provides both enough decay to sum in $N$ and an $O(h^{\infty})$ contribution.
If $|s|,|\sigma|<3 N^{\varepsilon}\gamma^{1/2}$, then, for $|N|\geq 1$, $\Phi_{N,a}$ will not be stationary in $\alpha$ if $|t|\geq (3N+6N^{\varepsilon})\gamma^{1/2}$ and non-stationary phase in $\alpha$ provides enough decay to sum in $N$ and an $O(h^{\infty})$ contribution.
Hence, the only non trivial contribution comes from $|N|\lesssim |t|\gamma^{-1/2}$. Now, if $a-\alpha>0$, the phase $\Phi_{N,a}$ cannot be stationary in $s$. As $\alpha<2\gamma$, we get the desired result. Observe that, furthermore, by the same reasoning in $\sigma$, the only significant contribution of $G^{+}_{h,\gamma}$ is restricted to $x<2\gamma$. \qed

At fixed $t\geq C\gamma^{1/2}$, we can further bound the cardinal of those $N$ that contribute significantly among the $C|t|\gamma^{-1/2}$ which are left. Observe that our previous computation tells us that for $t=0$, only the $N=0$ term may contribute, and from Lemma \ref{lemmeNtpetit},  we can and will restrict ourselves to $|t|\geq C\gamma^{1/2}$.

We need to introduce some notations : for a given space-time location $(x,y,t)$, let $\mathcal{N}(x,y,t)$ the set of $N$ with significant contributions in \eqref{eq:bis488bis} (e.g. for which there exists at least a stationary point for the phase in all variables),
\[
\mathcal{N}(t,x,y)=\{N\in\mathbb{Z}, (\exists)(\sigma,s,\alpha,\eta) \text{ such that } \nabla_{(\sigma,s,\alpha,\eta)}\Phi_{N,a}(t,x,y,\sigma,s,\alpha,\eta)=0 \}
\]
 Call $\mathcal{N}_{1}(t,x,y)$ the set of $N$ such that $N\in \mathcal{N}(t',x',y')$ for some $(t',x',y')$ such that $|t'-t|\leq \sqrt{\gamma}$, $|x-x'|< \gamma$ and $|y'+t'\sqrt{1+\gamma}-y-t\sqrt{1+\gamma}|<\gamma^{3/2}$,
 \[
 \mathcal{N}_1(t,x,y)=\cup_{\{(t',x',y')||t'-t|\leq \sqrt{\gamma}, |x-x'|< \gamma, |y'+t'\sqrt{1+\gamma}-y-t\sqrt{1+\gamma}|<\gamma^{3/2}\}}\mathcal{N}(t',x',y').
 \]
\begin{prop}\label{propcardN} The set $\mathcal{N}_{1}(t,x,y)$ is bounded, with optimal bound
\begin{equation}
  \label{eq:113}
  \left| \mathcal{N}_{1}(t,x,y)\right| \lesssim O(1)+{|t|}{\gamma^{-1/2} (\gamma^{3}/h^{2})^{-1}}\,,
\end{equation}
Moreover, the contribution of the sum over $N\notin \mathcal{N}_1(t,x,y)$ in \eqref{eq:bis488bis} is $O(h^{\infty})$.
\end{prop}
\begin{rmq}
This result generalizes \cite[Lemma 2.17, Lemma 2.18]{Annals}: there,  $|\mathcal{N}_1(t,x,y)|$ is bounded by an absolute constant $N_{0}$ and such that $\mathcal{N}(t,x,y)\subset [1,t/(2\sqrt{a})+N_0]$. In fact, when $a\sim \gamma \gg h^{4/7}$ one can easily see that $\frac{|t|}{\gamma^{1/2} (\gamma^{3}/h^{2})}=O(1)$ for bounded $t$.
\end{rmq}
\begin{rmq}
Here our parametrix is a sum over reflected waves for all $\gamma \gg h^{2/3}$; while for $a\gg h^{4/7}$ \cite{Annals} provides a different (and less straightforward) construction. When $\gamma \lesssim h^{4/7}$, Proposition \ref{propcardN} is crucial for dispersion, providing a sharp bound on the (large) number of overlapping waves.
\end{rmq}
\begin{proof}(of Proposition \ref{propcardN})
We first re-scale variables as follows
\begin{equation}\label{txyrescale}
 t=\sqrt{\gamma}T,\quad x=\gamma X, \quad y+t\sqrt{1+\gamma}=\gamma^{3/2}Y.
\end{equation}
Notice that $y$ behaves like $-t$  when the phase $\eta\Phi_{N,a}$ is stationary in $\eta$, implying that $y+t\sqrt{1+\gamma}$ should be small near significant contributions. The relevance of the $\gamma^{3/2}$ factor in rescaling will make itself clear later. On the support of $\psi_2(\alpha/\gamma)$ we have $\alpha\sim \gamma$, therefore we also re-scale $\alpha=\gamma A$: then $A\sim 1$ on the support of $\psi_2(A)$. Since the phase $\Phi_{N,a}$ is stationary in $s,\sigma$ at $s^2+a=\alpha$, $\sigma^2+x=\alpha$, it follows that we may restrict $s$ and $\sigma$ to $|s|,|\sigma|<3\gamma^{1/2}$, otherwise non-stationary phase in either variables provides an $O(h^{\infty})$ contribution. We then also let $s=\sqrt{\gamma}S$, $\sigma=\sqrt{\gamma}\Sigma$, where $|S|,|\Sigma|<3 $ are bounded.

Recall that $\gamma<1/100$, $|N|>0$ and $|t|\gtrsim\gamma^{1/2}$. the parameter $h$ is intended to be small and we define a large parameter $\lambda_{\gamma}=\gamma^{3/2}/h$ (hence restricting $\gamma>h^{2/3-\varepsilon}$.) Set
\begin{multline}\label{PhiNagammadef}
  \Psi_{N,a,\gamma}(T,X,Y,\Sigma,S,A,\eta)=\eta\left(Y+\Sigma^3/3+\Sigma(X-A)+S^3/3+S(\frac{a}{\gamma}-A)\right.\\
    {} +T\frac{(A-1)}{\sqrt{1+\gamma A}+\sqrt{1+\gamma}}-\left.\frac 43 NA^{3/2}\right)+\frac{N}{\lambda_{\gamma}}B(\eta\lambda_{\gamma}A^{3/2}).
\end{multline}
then $\gamma^{3/2}\Psi_{N,a,\gamma}(T,X,Y,\Sigma,S,A,\eta)=\eta \Phi_{N,a}(\sqrt{\gamma}T,\gamma X,\gamma^{3/2}Y-\sqrt{\gamma}\sqrt{1+\gamma}T,\sqrt{\gamma}\Sigma,\sqrt{\gamma}S,\gamma A)$, and, in the new variables, the phase function in \eqref{defVNgamma} becomes $\lambda_{\gamma}\Psi_{N,a,\gamma}$. When $\gamma\sim a$ we write $\Psi_{N,a,a}$ and $\lambda_{\gamma}=\lambda=a^{3/2}/h$. The critical points of $\Psi_{N,a,\gamma}$ with respect to $\Sigma,S, A,\eta$ are such that
\begin{gather}\label{critSig}
  \Sigma^2+X=A,\quad S^2+a/\gamma=A\\
\label{critA}
T=2\sqrt{1+\gamma A}\Big(\Sigma+S+2N\sqrt{A}(1-\frac 34 B'(\eta\lambda_{\gamma}A^{3/2}))\Big)\\
\begin{multlined}
\label{criteta}  
Y+T\frac{(A-1)}{\sqrt{1+\gamma A}+\sqrt{1+\gamma}}+\Sigma^3/3+\Sigma(X-A)+S^3/3+S(\frac{a}{\gamma}-A)\\
{}=\frac 43 NA^{3/2}(1-\frac 34 B'(\eta\lambda_{\gamma}A^{3/2})).
\end{multlined}
\end{gather}
Introducing the term $2N\sqrt{A}(1-\frac 34 B'(\eta\lambda_{\gamma}A^{3/2}))$ from \eqref{critA} in \eqref{criteta} provides a relation between $Y$ and $T$ that doesn't involve $N$ nor $B'$ as follows:
\begin{equation}\label{critYT}
  Y+T\frac{(A-1)}{\sqrt{1+\gamma A}+\sqrt{1+\gamma}}+\frac{S^{3}+\Sigma^3} 3+\Sigma(X-A)+S(\frac{a}{\gamma}-A)
=  \frac 23 A\Big(\frac{T}{2\sqrt{1+\gamma A}}-(\Sigma+S)\Big).
\end{equation}
We first estimate the cardinal of $\mathcal{N}_1(t,x,y)$, with $t$ sufficiently large: let $j\in \{1,2\}$ and $N_j\in \mathcal{N}_1(t,x,y)$  be any two elements of $\mathcal{N}_1(t,x,y)$. Then there exist $(t_j,x_j,y_j)$ such that $N_j\in\mathcal{N}(t_j,x_j,y_j)$; writing $t_j=\sqrt{\gamma}T_j$, $x_j=\gamma X_j$, $y_j+t_j\sqrt{1+\gamma}=\gamma^{3/2}Y_j$ and rescaling $(t,x,y)$ as in \eqref{txyrescale}, we have
$|T_j-T|\leq 1$, $|X_j-X|\leq 1$, $|Y_j-Y|\leq 1$. We now prove \eqref{eq:113}. From $N_j\in\mathcal{N}(t_j,x_j,y_j)$, there exist $\Sigma_j,A_j,\eta_j,S_j$ such that \eqref{critSig}, \eqref{critA}, \eqref{criteta} holds with $T$,$X$,$Y$,$\Sigma$,$S$,$A$,$\eta$ replaced by $T_j$,$X_j$,$Y_j$,$\Sigma_j$,$S_j$,$A_j$,$\eta_j$, respectively, and we also have $S_j^2+\frac{a}{\gamma}=A_j$. We re-write \eqref{critA} as follows
\begin{equation}\label{critAj}
2N_j\sqrt{A_j}(1-\frac 34 B'(\eta_j\lambda_{\gamma}A^{3/2}_j))=\frac{T_j}{2\sqrt{1+\gamma A_j}}-(\Sigma_j+S_j).
\end{equation}
Multiplying \eqref{critAj} by $\sqrt{A_{j'}}$, where $j'\neq j$, taking the difference and dividing by $\sqrt{A_1A_2}$ yields
\begin{multline}\label{difN1N2}
  2(N_1-N_2)=\frac 32 \Big(N_1B'(\eta_1\lambda_{\gamma}A_1^{3/2})-N_2B'(\eta_2\lambda_{\gamma}A_2^{3/2})\Big)
  -\frac{\Sigma_1+S_1}{\sqrt{A_1}}+\frac{\Sigma_2+S_2}{\sqrt{A_2}}\\
{}+\frac{T_1}{2\sqrt{A_1}\sqrt{1+\gamma A_1}}-\frac{T_2}{2\sqrt{A_2}\sqrt{1+\gamma A_2}}.
\end{multline}
We need to estimate $|N_1-N_2|$. Using that $\Sigma_j,S_j<3$, $A_j\sim 1$, it follows that $\frac{\Sigma_j+S_j}{\sqrt{A_j}}=O(1)$, for $j\in \{1,2\}$. The first difference, involving $B'$, in the right hand side of \eqref{difN1N2} behaves like $(N_1+N_2)/\lambda_{\gamma}^2$:  use $B'(\eta\lambda A^{3/2})\sim -\frac{b_1}{\eta^2\lambda^2A^3}$ and $\eta,A\sim 1$. We cannot take advantage of the difference itself: each term  $N_jB'(\eta_j\lambda_{\gamma}A_j^{3/2})$ corresponds to some $\eta_j,A_j$ (close to $1$) and the difference $\frac{1}{\eta_1A_1^{3/2}}-\frac{1}{\eta_2A_2^{3/2}}$ is bounded but has no reason to be very small (the difference between $A_j$ turns out to be $O(1/T)$, but we don't have any information about the difference between $\eta_j$ which is simply bounded by a small constant on the support of $\psi$). Therefore the bound $(N_1+N_2)/\lambda_{\gamma}^2$ for the terms involving $B'$ in \eqref{difN1N2} is sharp. Since $N_j\sim T_j$, and $|T_j-T|\leq 1$, it follows that this contribution is $\sim |T|/\lambda^2_{\gamma}$. We are reduced to proving that the difference in the second line of \eqref{difN1N2} is $O(1)$. Write
\begin{multline}\label{estimdifT1T2}
\Big|\frac{T_1}{\sqrt{A_1}\sqrt{1+\gamma A_1}}-\frac{T_2}{\sqrt{A_2}\sqrt{1+\gamma A_2}}\Big|\leq \frac{|T_1-T_2|}{\sqrt{A_1}\sqrt{1+\gamma A_1}} 
+\frac{T_2}{\sqrt{A_{1}A_{2}}}\Big|\frac{\sqrt{A_2}}{\sqrt{1+\gamma A_1}}-\frac{\sqrt{A_1}}{\sqrt{1+\gamma A_2}}\Big|\\
\leq \frac{2}{\sqrt{A_1}\sqrt{1+\gamma A_1}}+\frac{T_2|A_2-A_1|(1+\gamma(A_1+A_2))}{\sqrt{A_1A_2(1+\gamma A_1)(1+\gamma A_2)}(\sqrt{A_1(1+\gamma A_1)}+\sqrt{A_2(1+\gamma A_2)})}\\
\leq C(1+T_2|A_2-A_1|),
\end{multline}
for some absolute constant $C$ and we only used $|T_2-T_1|\leq 2$ and $A_j\sim 1$. For $T$ bounded we can conclude since $|T_2-T|\leq 1$. We are therefore reduced to bound $T_2|A_2-A_1|$ when $T_2$ is sufficiently large. In order to do so, we need the $Y$ variable. We use \eqref{critYT} with $T,X,Y,\Sigma,S,A,\eta$ replaced by $T_j,X_j,Y_j,\Sigma_j,S_j,A_j,\eta_j$, $j\in\{1,2\}$ to eliminate the terms containing $N$ and $B'$ as follows:
\begin{multline}
Y_j+\Sigma^3_j/3+\Sigma_j(X_j-A_j)+S^3_j/3+S_j(\frac{a}{\gamma}-A_j)+\frac 23 A_j(\Sigma_j+S_j)\\
=T_j\Big(\frac{A_j}{3\sqrt{1+\gamma A_j}}-\frac{(A_j-1)}{\sqrt{1+\gamma A_j}+\sqrt{1+\gamma}}\Big).
\end{multline}
For $|T|$ sufficiently large we also have $|T_j|$ large and we divide by $T_j$ in order to estimate the difference $A_1-A_2$ in terms of $Y_1/T_1-Y_2/T_2$:
\begin{equation}\label{critYsurT}
\frac{Y_j}{T_j}+O(\frac{A_j^{3/2}}{T_j})=F_{\gamma}(A_j), \quad F_{\gamma}(A)=\frac{A}{3\sqrt{1+\gamma A}}-\frac{(A-1)}{\sqrt{1+\gamma A}+\sqrt{1+\gamma}}.
\end{equation}
Taking the difference of \eqref{critYsurT} with itself for $j=1,2$ gives
\[
(\frac{Y_2}{T_2}-\frac{Y_1}{T_1})+O(\frac{1}{T_1})+O(\frac{1}{T_2})=(A_2-A_1)\int_0^1\partial_A F_{\gamma}(A_1+o(A_2-A_1))do.
\]
One may check that, as $1/4<A<4$ and $\gamma\ll 1$, $F_{\gamma}$ is a decreasing function of $A$ and $-\partial_AF_{\gamma}(A)\geq 1/6-(1+A) \gamma$; it follows that we bound $A_2-A_1$ with $A_2-A_1\lesssim (\frac{Y_2}{T_2}-\frac{Y_1}{T_1})+O(\frac{1}{T_1},\frac{1}{T_2})$,

and replacing the last expression in the last line of \eqref{estimdifT1T2} yields
\begin{align}
T_2|A_2-A_1| & \lesssim T_2\Big|\frac{Y_2}{T_2}-\frac{Y_1}{T_1}\Big|+O(T_2/T_1)+O(1)\\
 & \lesssim |Y_2-Y_1|+Y_1|1-T_2/T_1|+O(T_2/T_1)+O(1)=O(1),\nonumber
\end{align}
where we have used $|Y_1-Y_2|\leq 2$, $|T_1-T_2|\leq 2$ and that $Y_1/T_1$ is bounded (which can easily be seen from \eqref{critYsurT}). This ends the proof of \eqref{eq:113}.\qed

We next proceed with proving the contribution outside of $\mathcal{N}_{1}(t,x,y)$ to be $O(h^{\infty})$. Consider first $2a\lesssim \gamma$, then critical points in $S$ are such that $S^2+a/\gamma=A$, with $A$ near $1$, therefore $S_{\pm}=\pm \sqrt{A-a/\gamma}$. Then $\Psi_{N,a,\gamma}\in\{\Psi^{\varepsilon}_{N,a,\gamma}\}$, where
\begin{multline}\label{Psivarepsdfn}
\Psi^{\varepsilon}_{N,a,\gamma}(T,X,Y,\Sigma,A,\eta):=\eta\Big(Y+\Sigma^3/3+\Sigma(X-A)+\varepsilon \frac 23(A-\frac{a}{\gamma})^{3/2}\\
+T\frac{(A-1)}{\sqrt{1+\gamma A}+\sqrt{1+\gamma}}-\frac 43 NA^{3/2}\Big)+\frac{N}{\lambda_{\gamma}}B(\eta\lambda_{\gamma}A^{3/2}).
\end{multline}
Denote $\mathcal{N}^{\varepsilon}(T,X,Y):=\{(N\in\mathbb{Z}, \exists (\Sigma,A,\eta) \text{ such that } \nabla_{(\Sigma,A,\eta)}\Psi^{\varepsilon}_{N,a,\gamma}(T,X,Y,\Sigma,A,\eta)=0\}$, then for $(t,x,y)=(\sqrt{\gamma}T,\gamma X,\gamma^{3/2}Y-\sqrt{\gamma}\sqrt{1+\gamma}T)$, we have $\mathcal{N}(t,x,y)=\cap_{\varepsilon=\pm}\mathcal{N}^{\varepsilon}(T,X,Y)$. Indeed, if $N\in \mathcal{N}(t,x,y)$, then there exists a critical point $\Sigma_c,S_c,A_c,\eta_c$ for the phase $\Psi_{N,a,\gamma}$ with $S_c^2=A_c-a/\gamma$, hence $S_c\in \{S_{\pm}\}$. It follows that $\Sigma_c,A_c,\eta_c$ is a critical point for both $\Psi^{\pm}_{N,a,\gamma}$. On the other hand, if $\Sigma_{\pm},A_{\pm},\eta_{\pm}$ is a critical point for $\Psi^{\pm}_{N,a,\gamma}$, then $(\Sigma_{\pm},\pm \sqrt{A_{\pm}-a/\gamma}, A_{\pm},\eta_{\pm})$ are critical point for $\Psi_{N,a,\gamma}$. We have to prove that
\begin{equation}\label{N1out}
\sum_{N\notin \mathcal{N}_1}W_{N,\gamma}(T,X,Y)=O(h^{\infty}),
\end{equation}
where we define $W_{N,\gamma}(T,X,Y):=V_{N,\gamma}(t,x,y)$.
For $\varepsilon\in \{\pm\}$, set
$$
\mathcal{N}^{\varepsilon}_1(T,X,Y)=\cup_{(T',X',Y')\in B_1(T,X,Y)}\mathcal{N}^{\varepsilon}(T',X',Y')\,,
$$
we have $\mathcal{N}_1=\cap_{\pm}\mathcal{N}^{\pm}_1$ and therefore, $(\mathcal{N}_1)^c=\cup_{\varepsilon\in \{\pm 1\}}(\mathcal{N}^{\varepsilon}_1)^c$, where the notation $(\mathcal{N}_1)^c$ denotes the complement of $\mathcal{N}_{1}$.
Hence \eqref{N1out} follows from proving that
\begin{equation}\label{WNdecaytoprove}
\forall \varepsilon\in\{\pm\},\,\,\,\,\,\sum_{N\notin \cup_{\pm}\mathcal{N}^{\pm}_1}W^{\varepsilon}_{N,\gamma}(T,X,Y)=O(h^{\infty})\,,
\end{equation}
where $W^{\pm }_{N,\gamma}(T,X,Y)$ have phase functions $\Psi^{\pm}_{N,a,\gamma}$ and symbols are obtained from the symbol of $W_{N,\gamma}$ multiplied by the symbol of $A_{\pm}((\eta\lambda_{\gamma})^{2/3}(A-\frac{a}{\gamma}))$. Go back to the system \eqref{critSig}, \eqref{critA}, \eqref{criteta} and define the integral curves corresponding to $\Psi^{\varepsilon}_{N,a,\gamma}$ as follows
\begin{gather}\nonumber
\tilde X_{\varepsilon}(\Sigma,A):=A-\Sigma^2,\\
\nonumber
\tilde T_{\varepsilon}(\Sigma, A,\eta):=2\sqrt{1+\gamma A}\Big(\Sigma+\varepsilon(A-\frac{a}{\gamma})^{1/2}+2N\sqrt{A}(1-\frac 34B'(\eta\lambda_{\gamma}A^{3/2}))\Big),\\
\nonumber
\tilde Y_{\varepsilon}(\Sigma,A,\eta):=-2\sqrt{1+\gamma A}\frac{(\Sigma+\varepsilon(A-\frac{a}{\gamma})^{1/2})(A-1)}{\sqrt{1+\gamma A}+\sqrt{1+\gamma}}+\frac 23\Sigma^3-\varepsilon(A-\frac{a}{\gamma})^{3/2}\\
\nonumber
\quad\quad\quad\quad\quad\quad\quad{}+4N\sqrt{A}\Big(1-\frac 34 B'(\eta\lambda_{\gamma}A^{3/2})\Big)\Big(\frac A3-\frac{(A-1)\sqrt{1+\gamma A}}{\sqrt{1+\gamma A}+\sqrt{1+\gamma}}\Big).
\end{gather}
Let $\varepsilon\in\{\pm\}$ and $N\notin \mathcal{N}^{\varepsilon}_1$. Then $N\notin \mathcal{N}^{\varepsilon}(T',X',Y')$ for all $(T',X',Y')\in B_1(T,X,Y)$, which translates into
\begin{equation}\label{criPsiNout}
\forall (T',X',Y')\in B_1(T,X,Y),\quad\nabla_{(\Sigma,A,\eta)}\Psi^{\varepsilon}_{N,a,\gamma}(T',X',Y',\Sigma,A,\eta)\neq 0\,. 
\end{equation}
Now, by design of our integral curves,
\[
\nabla_{(\Sigma,A)}\Psi^{\varepsilon}_{N,a,\gamma}(T',X',Y',\Sigma,A,\eta)=\Big(X'-\tilde X_{\varepsilon}(\Sigma,A),\frac{T'-\tilde T_{\varepsilon}(\Sigma,A,\eta)}{2\sqrt{1+\gamma A}}\Big),
\]
while
\begin{multline*}
  \partial_{\eta}\Psi^{\varepsilon}_{N,a,\gamma}(T',X',Y',\Sigma,A,\eta)=Y'-\tilde Y_{\varepsilon}(\Sigma,A,\eta)+\Sigma (X'-\tilde X_{\varepsilon}(\Sigma,A))
  \\ {}+\frac{(T'-\tilde T_{\varepsilon}(\Sigma,A,\eta))(A-1)}{\sqrt{1+\gamma A}+\sqrt{1+\gamma}}.
\end{multline*}
It follows that, for all $(T',X',Y')$ and all $(\Sigma,A,\eta)$ on the support of the symbol, we have
\begin{multline}\label{TXYbornderiv}
(|\partial_{\Sigma}\Psi^{\varepsilon}_{N,a,\gamma}|+|\partial_{A}\Psi^{\varepsilon}_{N,a,\gamma}|+|\partial_{\eta}\Psi^{\varepsilon}_{N,a,\gamma}|)(T',X',Y',\Sigma,A,\eta) \leq \\
10 (|X'-\tilde X_{\varepsilon}(\Sigma,A)|+|T'-\tilde T_{\varepsilon}(\Sigma,A,\eta)|+|Y'-\tilde Y_{\varepsilon}(\Sigma,A,\eta)|)\\
\leq 100(|\partial_{\Sigma}\Psi^{\varepsilon}_{N,a,\gamma}|+|\partial_{A}\Psi^{\varepsilon}_{N,a,\gamma}|+|\partial_{\eta}\Psi^{\varepsilon}_{N,a,\gamma}|)(T',X',Y',\Sigma,A,\eta).
\end{multline}
Using \eqref{criPsiNout} and the first inequality in \eqref{TXYbornderiv}, for all $(T',X',Y')\in B_1(T,X,Y)$ and all $(\Sigma,A,\eta)$, 
\begin{equation}\label{TXYcurves}
|X'-\tilde X_{\varepsilon}(\Sigma,A)|+|T'-\tilde T_{\e}(\Sigma,A,\eta)|+|Y'-\tilde Y_{\e}(\Sigma,A,\eta)|\neq 0
\end{equation}
( or $\Psi^{\varepsilon}_{N,a,\gamma}$ would have a critical point.) Hence, $(\tilde T_{\e}(\Sigma,A,\eta),\tilde X_{\e}(\Sigma,A),\tilde Y_{\e}(\Sigma,A,\eta))$ is not in $B_1(T,X,Y)$ (otherwise taking $(T',X',Y')=(\tilde T_{\e}(\Sigma,A,\eta),\tilde X_{\e}(\Sigma,A),\tilde Y_{\e}(\Sigma,A,\eta))$ would contradict \eqref{TXYcurves}). In other words, \eqref{TXYcurves} implies that for all $(\Sigma,A,\eta)$,
\[
|X-\tilde X_{\e}(\Sigma,A)|+|T-\tilde T_{\e}(\Sigma,A,\eta)|+|Y-\tilde Y_{\e}(\Sigma,A,\eta)|\geq 1,
\]
and using the second inequality in \eqref{TXYbornderiv} with $(T',X',Y')=(T,X,Y)$, for all $(\Sigma,A,\eta)$,
\[
(|\partial_{\Sigma}\Psi^{\varepsilon}_{N,a,\gamma}|+|\partial_{A}\Psi^{\varepsilon}_{N,a,\gamma}|+|\partial_{\eta}\Psi^{\varepsilon}_{N,a,\gamma}|)(T,X,Y,\Sigma,A,\eta)\geq \frac{1}{10}\, .
\]
Therefore non-stationary phase always applies in at least one variable among $\Sigma,A,\eta$  and each $W^{\varepsilon}_{N,\gamma}$ with $N\notin \mathcal{N}^{\varepsilon}_1$ provides a $O(\lambda_{\gamma}^{-M})$ contribution for any $M\geq 1$ (where $M$ corresponds to the number of integrations by parts.) We conclude using that the sum over $N$ in $G^+_{h,\gamma}$ restricts to  $|N|\lesssim \gamma^{-1/2}$ from Lemma \ref{lemmeNtpetit}.

Let us now deal with the remaining case: $a\sim \gamma$. We re-scale slightly differently for convenience, with $\lambda=a^{3/2}/h$ and
$t=\sqrt{a}T$, $x=aX$, $y+t\sqrt{1+a}=a^{3/2}Y$, $s=\sqrt{a}S$, $\sigma=\sqrt{a}\Sigma$, $\alpha=aA$. The phase $\Psi_{N,a,\gamma\sim a}=:\Psi_{N,a,a}$ has the same form as in \eqref{PhiNagammadef} where $\gamma$ is now replaced by $a$ and $\lambda_{\gamma}$ by $\lambda=\lambda_a$. We rewrite the integral in $S$ as an Airy function: in the new variables, 
\[
\int e^{i\eta \lambda(S^3/3+S(1-A))}dS=(\eta\lambda)^{-1/3}Ai\Big((\eta\lambda)^{2/3}(1-A)\Big).
\]
Let $\chi_0\in C^{\infty}$ be such that $\chi_{0}\equiv 1$ on $[0,\infty]$ and $\chi_{0}\equiv 0$ on $[-\infty,-2]$. Then $(\chi_0Ai)((\eta\lambda)^{2/3}(1-A))$ is a symbol of order $2/3$ supported on $(\eta\lambda)^{2/3}(A-1)\leq 2$ and $(1-\chi_0)Ai((\eta\lambda)^{2/3}(1-A)))$ is supported on  $A\geq 1$ with value $1$ on $(\eta\lambda)^{2/3}(A-1)\geq 2$. Setting again $W_{N,a}(T,X,Y):=V_{N,a}(t,x,y)$, we rewrite each integral in \eqref{eq:bis488bis} as follows $W_{N,a}(T,X,Y):=W^{0}_{N,a}(T,X,Y)+\tilde W_{N,a}(T,X,Y)$:
\begin{equation}
W^0_{N,a}(T,X,Y):=\frac{\gamma^{3/2}}{(2\pi h)^3}\int \frac{e^{i\lambda \Psi^0_{N,a}}}{(\eta\lambda)^{1/3}}(\chi_0Ai)\Big((\eta\lambda)^{2/3}(1-A)\Big) 
 \chi_1((\eta\lambda)^{2/3}A)\psi_2\Bigl(\frac \alpha \gamma A\Bigr)\,d\Sigma dA d\eta,
\end{equation}
with
\begin{equation}
  \label{eq:18}
    \Psi^{0}_{N,a}
  :=\eta\Big(Y+\Sigma^3/3+\Sigma(X-A)
+T\frac{(A-1)}{\sqrt{1+\gamma A}+\sqrt{1+a}}-\frac 43 NA^{3/2}\Big)+\frac{N}{\lambda}B(\eta\lambda A^{3/2})\,,
\end{equation}
and 
\begin{equation*}
    \tilde W_{N,a}(T,X,Y):=\frac{\gamma^{3/2}}{(2\pi h)^3}\int \frac{e^{i\lambda \Psi_{N,a,a}}}{(\eta\lambda)^{1/3}}(1-\chi_0)((\eta\lambda)^{2/3}(1-A)) 
  \chi_1((\eta\lambda)^{2/3}A)\psi_2\Bigl(\frac \alpha \gamma A\Bigr)\,d\Sigma dS dA d\eta.
\end{equation*}
To prove that $ \sum_{N\notin \mathcal{N}_1}W_{N,a}(T,X,Y)=O(h^{\infty})$ it will be enough to prove 
\begin{equation}\label{N1outtansum}
\sum_{N\notin \mathcal{N}_1}\tilde W_{N,a}(T,X,Y)=O(h^{\infty}) \text{ and } \sum_{N\notin \mathcal{N}_1} W^0_{N,a}(T,X,Y)=O(h^{\infty}).
\end{equation}
We proceed with the first sum in \eqref{N1outtansum}. By design, on the support of $(1-\chi_0)((\eta\lambda)^{2/3}(1-A))$ we have $A\geq 1$. If we set $A=1+\mu^2$, we can perform the standard stationary phase in $S$ with critical points $S=\pm \mu$ (alternatively, we may rewrite the Airy function as $A_{+}+A_{-}$ and use the associated oscillatory integrals to recover the new phases, indexed by $\pm$). From there, we proceed as we did with the case $2a\lesssim \gamma$. The only difference is that we now have a symbol of order $2/3$ in $A$ (coming from $1-\chi_{0}$), so one integration by parts with respect to $A$ provides a factor $\lambda^{-1}\lambda^{2/3}=\lambda^{-1/3}$; an integration by parts with respect to $\Sigma$ yields a factor $\lambda^{-1}$ and with respect to $\eta$ a factor $\lambda^{-1}\lambda^{2/3}$. Therefore each $\tilde W_{N,a}$ with $N\notin \mathcal{N}_1$ has at least an $O(\lambda_{\gamma}^{-M/3})$ contribution for any $M\geq 1$ (where $M$ corresponds to the number of integrations by parts) and we are done, as the sum is over $|N|\lesssim \gamma^{-1/2}$. Consider now the second sum in \eqref{N1outtansum}.
Since the phase function of $W^0_{N,a}$ does not depend on $S$, we always have $\partial_S\Psi^0_{N,a}=0$. We proceed exactly like in the case $2a\leq \gamma$, the only difference being that we may take $\varepsilon=0$.\end{proof}

\subsection{Tangential waves for $\gamma >h^{2/3(1-\varepsilon)}$}

We start with the (most difficult) case $\gamma/4\leq a\leq 2 \gamma $ and this restricts $a\geq  h^{2/3(1-\varepsilon)}$ due to the $\chi_{1}$ cut-off. We re-scale variables:
\begin{equation}
  \label{eq:83ff}
  x=aX\,, \alpha=a A\,, t=\sqrt a\sqrt{1+a} T\,, s=\sqrt a S\,, \sigma=\sqrt a \Upsilon\,, y+t\sqrt{1+a}=a^{3/2} Y\,.
\end{equation}
Define $\lambda=a^{3/2}/h$ to be our large parameter, then we write
\begin{multline}
  \label{eq:bis488terff}
  G^+_{h,\gamma}(t,x,y,a,0,0)= \frac {a^{2}} {(2\pi h)^{3}} \sum_{N\in \Z} \int_{\R^{4}} e^{i  \lambda \Psi_{N,a,a}}  \eta^{2} \psi({\eta})
  \chi_{1}(\lambda^{2/3}A) \psi_{2}\Bigl(\frac a \gamma A\Bigr) \, dS d\Upsilon  dA d\eta\,,
\end{multline}
where $\Psi_{N,a,a}$ has been introduced in \eqref{PhiNagammadef} and $\gamma\sim a$. 
From the compact support of $\psi_2$ we have $A\leq 6$; since critical points are such that $S^2=A-1$, $\Upsilon^2=A-X$,  we can restrict ourselves to $|S|,|\Upsilon|< 3$ without changing the contribution modulo $O(h^{\infty})$; we therefore insert a suitable cut-off $\chi_{2}$,  and obtain (modulo $O(h^{\infty})$) a slightly modified operator,
\begin{equation}
  \label{eq:bis488ter9999}
G^{+,\sharp}_{h,\gamma}(t,x,y,a,0,0)= \frac {a^{2}} {(2\pi h)^{3}} \sum_N \int_{\R^{4}} e^{i \lambda  \Psi_{N,a,a}}  \eta^{2}\psi({\eta}) \chi_{1}(\lambda^{\frac 2 3}A)  \chi_{2}(S,\Upsilon) \psi_{2}\Bigl(\frac a \gamma A\Bigr) \, dS d\Upsilon  dA d\eta\,,
\end{equation}
and $G^{+}_{h,\gamma}=G^{+,\sharp}_{h,\gamma}+O(h^{\infty})$ (remembering that the significant part of the sum over $N$ is for $N\lesssim 1/\sqrt a$.) Since for $X>A$ there are no real critical points with respect to $\Upsilon$, we may restrict to $X\lesssim 1$. As the Green function $G^{+}_{h,\gamma}$ is symmetric with respect to $x$ and $a$, we may even restrict ourselves to $x<a$, and we will obtain slightly better bounds in this case.

Given that $\partial_{S}\Psi_{N,a,a}=\eta(S^{2}+1-A)$, by integration by parts in $S$ we may restrict ourselves to $A>9/10$ and therefore  $G^+_{h,a}(t,x,y,a,0,0)=\sum_N W^+_{N,a}(T,X,Y)+O(h^{\infty})$ with
\begin{equation}
  \label{eq:bis488ter999999}
  W_{N,a}(T,X,Y): = \frac {a^{2}} {(2\pi h)^{3}}  \int_{\R^{2}}\int_{\R^{2}} e^{i \lambda \Psi_{N,a,a}}  \eta^{2} \psi({\eta})
  \chi_{2}(S,\Upsilon) \psi_{3}(A) \, dS d\Upsilon  dA d\eta\,,
\end{equation}
where $\psi_{3}$ has support in $[9/10,6]$ (and actually includes dependence on an harmless factor $a/\gamma$, which we are hiding since it will be irrelevant from now on.)
\begin{rmq}\label{remNpetitlambda}
When $N\lesssim \lambda$, the phase factor $e^{iNB(\eta \lambda A^{3/2})}$ does not oscillate: we can move it into the symbol. Indeed, we have $|NB(\eta\lambda A^{3/2})|\sim |\frac{b_1 N}{\eta\lambda A^{3/2}}|\sim |N/\lambda|\lesssim1 $. The remaining phase is linear in $\eta$, stationary phase in $\eta$ does not produce decay but localizes the wave front in physical space. When $N\geq \lambda$, the term $NB(\eta\lambda A^{3/2})$ produces oscillations, but it will allow to perform stationary phase in both $A$ and $\eta$, also providing decay with respect to $\eta$ (useful especially for  $N\geq \lambda^2$.)
\end{rmq}
\begin{prop}
  \label{Prop2}
Let $1\leq |N|\lesssim\lambda$ and let $W_{N,a}(T,X,Y)$ be defined in \eqref{eq:bis488ter999999}. Then the stationary phase theorem applies in $A$ and yields
\begin{multline}
  \label{eq:bis488ter999999machin}
    W_{N,a}(T,X,Y)= \frac {a^{2}} {h^{3}(N\lambda)^{\frac 1 2}}  \int_{\R}\int_{\R^{2}} e^{i \lambda\eta (Y+\phi_{N,a}(T,X,\Upsilon,S))}  \eta^{2} \psi({\eta})\\
{}\times   \chi_{3}(S,\Upsilon,a,1/N,h,\eta) \, dS d\Upsilon d\eta+O(h^{\infty})\,,
\end{multline}
where $\chi_{3}$ has compact support in $(S,\Upsilon)$ and harmless dependency on the parameters $a$, $h$, $1/N$, $\eta$. The critical point $A_c$ solves
\begin{equation}
  \label{eq:89ff<}
  \frac{T\sqrt{1+a}}{2N\sqrt{1+a A}}-\frac{\Upsilon+S} N-2A^{1/2}=0\,.
  \end{equation}
and the phase function $ \phi_{N,a}$ is given by
\begin{equation}\label{critphaNpetit}
\phi_{N,a}(T,X,\Upsilon,S)=\frac{\Upsilon^3}{3}+X\Upsilon+\frac{S^3}{3}+S-(S+\Upsilon)A_c+T\frac{(A_c-1)}{1+\sqrt{\frac{1+aA_c}{1+a}}}-\frac 43 NA_c^{3/2}.
\end{equation}
\end{prop}
\begin{proof}
Using Remark \ref{remNpetitlambda}, we immediately move  the factor $e^{iNB(\eta\lambda A^{3/2})}$ into the symbol. Therefore the phase of $W_{N,a}$ becomes $\Psi_{N,a,a}(T,X,Y,\Upsilon,S,A,\eta)-\frac N \lambda B(\eta \lambda A^{3/2})$ which is linear in $\eta$ and stationary in $A$ where its derivative in $A$ vanishes, which is nothing but \eqref{eq:89ff<}. It immediately follows from \eqref{eq:89ff<} that  $T/(4N)$ is bounded when the phase is stationary in $A$, and even that $T/(4N)\sim \sqrt{A_{c}}$. The second derivative of the phase with respect to $A$ equals $-\frac{N}{A^{1/2}}(1+\frac{aTA^{1/2}\sqrt{1+a}}{4N(1+aA)^{3/2}})$; as $aT/N=O(a)$, the second derivative does not vanish and has size comparable to $|N|$; from this we may apply stationary phase to get the desired expansion.
\end{proof}
Setting  $K:=\frac{T}{4N}$ and $w:=\frac{S+\Upsilon}{2N}$, \eqref{eq:89ff<} for the critical point $A_c$ becomes 
\begin{equation}\label{eq:Acsol}
A^{1/2}=K\sqrt{\frac{1+a}{1+aA}}-w.
\end{equation}
\begin{lemma}\label{lemAcsol}
The critical point $A_c(K,w,a)$ satisfying \eqref{eq:Acsol} is $A^{1/2}_c(\cdot)=K_{\infty}(a)-w(1-a\E(\cdot))$, where $K_{\infty}(a): =K\sqrt{\frac{2(1+a)}{1+\sqrt{1+4K^2a(1+a)}}}$ and where $\E$ is a smooth function, uniformly bounded and defined as follows :
\[
\E(K,w,a)=\int_0^1\frac{A_c^{1/2}(K,\theta w,a)\frac{(A_c^{1/2}(K,\theta w,a)+\theta w)^3}{K^2(1+a)} d \theta}{1+aA_c^{1/2}(K,\theta w,a)\frac{(A_c^{1/2}(K,\theta w,a)+\theta w)^3}{K^2(1+a)}}.
\]
Moreover $K\rightarrow K_{\infty}(a)$ is a smooth bijection near $1$ and $K_{\infty}(a)-1=(K-1)(1+O(a))$. 
\end{lemma}
\begin{proof}
At $w=0$, $A_c^{1/2}$ is the unique solution to $Z=K\sqrt{\frac{1+a}{1+aZ^2}}$, therefore we immediately get $A_c^{1/2}(K,0,a)=K_{\infty}(a)$. Next, $A_{c}^{1/2}(K,w,0)=K-w$ and $\partial_{a}(A^{1/2}-K\sqrt{1+a}/\sqrt{1+aA}-w)=K/(2(\sqrt{1+a}(1+aA)^{3/2})\sim K/2$, so that we can apply the implicit function theorem to get a unique solution $A_{c}^{1/2}(K,w,a)$. We write $A_c^{1/2}(K,w,a)=A_c^{1/2}(K,0,a)+w\int_0^1\partial_w (A_c^{1/2})(K,\theta w,a) d\theta$, and it remains to compute $\partial_w (A_c^{1/2})$. Taking the derivative of \eqref{eq:Acsol} with respect to $w$ yields
\begin{equation}
  \label{deriveeAc}
  \partial_w(A_c^{1/2})\Big(1+aA_c^{1/2}\frac{K\sqrt{1+a}}{\sqrt{1+aA_c}^3}\Big)=-1
\end{equation}
and using again \eqref{eq:Acsol} to replace $\frac{1}{\sqrt{1+aA_c}}=\frac{A_c^{1/2}+w}{K\sqrt{1+a}}$ provides the desired $\E(K,w,a)$. Finally,
\begin{multline*}
  K_{\infty}(a)-1=(K-1)\frac{\sqrt{2(1+a)}}{\sqrt{1+\sqrt{(1+2a)^{2}+(K^{2}-1)4a(a+1)}}}\\
  {}+\frac{\sqrt{2(1+a)}}{\sqrt{1+(1+2a)\sqrt{1+(K^{2}-1)O(a)}}}-1
\end{multline*}
and the second line is indeed $(K^{2}-1)O(a)$.
\end{proof}
\begin{prop}
  \label{prop3}
Let  $|N|\geq \lambda$ and $W_{N,a}(T,X,Y)$ be defined in \eqref{eq:bis488ter999999}, then the stationary phase applies in both $A$ and $\eta$ and yields
\begin{equation}
  \label{eq:bis488ter999999machin>}
    W_{N,a}(T,X,Y)= \frac {a^{2}} {h^{3} N} \int_{\R^{2}} e^{i \lambda \Psi_{N,a,a}(T,X,Y,\Upsilon,S,A_c,\eta_c)}  \chi_{3}(S,\Upsilon,a,1/N,h) \, dS d\Upsilon +O(h^{\infty})\,,
\end{equation}
where $\chi_{3}$ has compact support in $(S,\Upsilon)$ and harmless dependency on the parameters $a,h,1/N$. 
\end{prop}
\begin{proof}
The phase $\Psi_{N,a,a}$ from \eqref{PhiNagammadef} is stationary in $A,\eta$ when $\partial_A \Psi_{N,a,a}=0$, $\partial_{\eta} \Psi_{N,a,a}=0$, where
\begin{align}
    \label{eq:86bis}
  \partial_{A}\Psi_{N,a,a}& =\eta\Big(-\Upsilon-S+\frac T 2 \frac{\sqrt{1+a}}{\sqrt{1+aA}}-2N A^{1/2}(1-\frac 3 4  B'(\eta \lambda A^{3/2}))\Big)\\
  \label{eq:86e}
  \partial_{\eta}\Psi_{N,a,a} & =Y+\frac{\Upsilon^{3}} 3+\Upsilon(X-A)+\frac {S^{3}} 3+S(1 -A)\\
 & \;\;\;\;\;\;\;\;\;\;\;{}+T\frac{\sqrt{1+a} (A-1)}{\sqrt{1+a}+\sqrt{1+aA}}-\frac 4 3 N A^{3/2}(1-\frac 3 4  B'(\eta \lambda A^{3/2}))\nonumber
\end{align}
where 
$B'(\eta\lambda A^{3/2})\sim O(1/\lambda^2)$.
The second order derivatives are given by 
\begin{gather}\nonumber
\partial^2_{A}\Psi_{N,a,a}\sim -\eta\frac{N}{A^{1/2}}\,,\quad\quad\quad \partial^2_{\eta}\Psi_{N,a,a}=N\lambda A^3B''(\eta\lambda A^{3/2})\sim \frac{N}{\lambda^2},\\
\nonumber
\partial^2_{\eta,A}\Psi_{N,a,a}=\eta^{-1}\partial_A\Psi_{N,a,a}+\frac 32\eta \lambda NA^{2}B''(\eta\lambda A^{3/2})).
\end{gather}
As $B''\sim O(1/\lambda^{3})$, $\partial^{2}_{\eta,A}\Psi_{N,a,a}\sim N/\lambda^{2}$ ; at the critical points, we have
$\text{det Hess }\Psi_{N,a,a}\sim \frac{N^2}{\lambda^2}$, for  $N\geq \lambda$, and we may apply stationary phase in $(A,\eta)$.
\end{proof}
We now study critical points in $(A,\eta)$ for $|N|>\lambda$. Applying the implicit function theorem to the system $\partial_A\Psi_{N,a,a}=0,\partial_{\eta}\Psi_{N,a,a}=0$ around $(S=0,\Upsilon=0)$ yields at most a pair of critical points $(A_c,\eta_c)$ belonging to the support of the symbol and depending on all variables. In the following we will only need the derivatives of $A_c$ with respect to the two remaining variables $S,\Upsilon$. For that, we take the derivatives of \eqref{eq:86bis}, \eqref{eq:86e} with respect to $S$, $\Upsilon$. This gives
\[
\left(
\begin{array}{cc}
\partial_S A_c , &  \partial_S \eta_c    \\
\end{array}
\right)
\left(
\begin{array}{cc}
 \partial^2_{A}\Psi_{N,a,a} &  \partial^2_{\eta,A}\Psi_{N,a,a}    \\
\partial^2_{\eta,A}\Psi_{N,a,a}  &      \partial^2_{\eta} \Psi_{N,a,a}
\end{array}
\right)=
\left(
\begin{array}{c}
   -1  \\
   S^2+1-A_c\\   
\end{array}
\right),
\]
\[
\left(
\begin{array}{cc}
\partial_{\Upsilon} A_c , &  \partial_{\Upsilon} \eta_c    \\
\end{array}
\right)
\left(
\begin{array}{cc}
 \partial^2_{A}\Psi_{N,a,a} &  \partial^2_{\eta,A}\Psi_{N,a,a}    \\
\partial^2_{\eta,A}\Psi_{N,a,a}  &      \partial^2_{\eta} \Psi_{N,a,a}\end{array}
\right)=
\left(
\begin{array}{c}
   -1  \\
   \Upsilon^2+X-A_c\\   
\end{array}
\right).
\]
\
From this system we obtain the following Lemma :
\begin{lemma}
The critical point $A_c$ is such that, with $A_{c}(S=\Upsilon=0)=A_{c|0}$,
\begin{gather*}
A_c=A_{c|0}+(S,\Upsilon)\cdot \int_0^1\nabla_{(S,\Upsilon)}(A_c(T,X,Y,\theta\Upsilon,\theta S,a))d\theta\,,\\
\nabla_{(S,\Upsilon)}A_c=\frac{N\lambda A_c^3B''(\eta_c\lambda A_c^{3/2})}{\text{det Hess }\Psi_{N,a,a}({A_c,\eta_c})}\Big(-1+\frac {3\eta_c(S^2+1-A_c)}{2A_c},-1+\frac {3\eta_c(\Upsilon^2+X-A_c)}{2 A_c}\Big).
\end{gather*}
Moreover, $A_{c|0}=1+4(\frac{3Y}{T}-1)(1+O(a))$ belongs to a neighborhood of $1$ and at $S=\Upsilon=0$,
\begin{gather}\label{TNAc0}
T=4N\sqrt{A_{c|0}}(1-\frac 34B'(\eta_{c|0}\lambda A^{3/2}_{c|0})\sqrt{\frac{1+aA_{c|0}}{1+a}}\\
Y=4N\sqrt{A_{c|0}}(1-\frac 34B'(\eta_{c|0}\lambda A_c^{3/2}|_0)\Big(\frac{A_{c|0}}{3}-\frac{(A_c|_{0}-1)\sqrt{1+aA_{c|0}}}{\sqrt{1+a}+\sqrt{1+aA_{c|0}}}\Big)\\
A^{1/2}_{c|0}=\frac{T}{4N}(1+O(a))(1+O(\lambda^{-2}))\,.
\end{gather}
\end{lemma}
We are left with an integral over $(S,\Upsilon)$. In the remaining part of this section we prove the following dispersive estimates that will be crucial in order to obtain better Strichartz estimates than those implied by the usual duality argument and dispersion (recall $t=\sqrt a T$):
\begin{prop}\label{propdisptangN=0}
   For $|T|\leq \frac 52$, we have
  \begin{equation}
    \label{eq:decW0a}
    |W_{0,a}(T,X,Y)|+|G^{+,\sharp}_{h,\gamma}(t,x,y,a,0,0)|\lesssim  \frac{\sqrt \gamma}{h^2}\inf \left( 1, \left(\frac{h}{\gamma t}\right)^{1/2}\right)\,.
  \end{equation}
\end{prop}
\begin{prop}
\label{propdispNpetitpres}
For $1\leq |N|<\lambda^{1/3}$ and $|T-4N|\lesssim 1/N$, we have
  \begin{equation}
    \label{eq:2hh}
       \left| W_{N,a}(T,X,Y)\right| \lesssim  \frac {h^{1/3}} {h^{2}((N/\lambda^{1/3})^{1/4}+|N(T-4N)|^{1/6})}\,.
  \end{equation}
\end{prop}

\begin{prop}
\label{propdispNpetitloin}
For $1\leq |N|<\lambda^{1/3}$ and $|T-4N|\gtrsim 1/N$, we have
  \begin{equation}
    \label{eq:2ff}
       \left| W_{N,a}(T,X,Y)\right| \lesssim   \frac{h^{1/3}}{{h^{2}}(1+|N(T-4N)|^{1/4})}\,.
  \end{equation}
\end{prop}
\begin{rmq}
  If $X<1$, we can replace $|N(T-4N)|^{1/4}$ by $N^{1/2}|T-4N|^{1/2}$ in \eqref{eq:2ff}.
\end{rmq}
\begin{prop}
\label{propdispNgrand}
For $N\gtrsim \lambda^{1/3}$ (hence $a\lesssim h^{1/3}$), we have:
\begin{enumerate}
\item when $\lambda^{1/3}\lesssim N\lesssim \lambda$,
  \begin{equation}
    \label{eq:1ff}
       \left| W_{N,a}(T,X,Y)\right|\lesssim    \frac {h^{1/3}} {h^{2}((N/\lambda^{1/3})^{1/2} +(N|T-4N|)^{1/4})}\, ;
  \end{equation}
    \item when $\lambda\lesssim N\lesssim 1/\sqrt a$, (hence $a\lesssim h^{1/2}$),
    \begin{equation}
    \label{eq:1ff>}
       \left| W_{N,a}(T,X,Y)\right|\lesssim   \frac {h^{1/3}\lambda^{2/3}} {h^{2}N }\,.
  \end{equation}
  \end{enumerate}
\end{prop}
\begin{rmq}
  If $X<1$ we can replace $|N(T-4N)|^{1/4}$ by $\lambda^{1/6}|T-4N|^{1/2}$ in \eqref{eq:1ff}.
\end{rmq}

\begin{cor}\label{cordispNgrand}
For $|T|\geq \lambda^2$, (hence $a\lesssim h^{4/7}$), we have
\begin{equation}
|G^{+,\sharp}_{h,\gamma}(t,x,y,a,0,0)|\lesssim \frac{h^{1/3}}{h^{2}\lambda^{4/3}}\,.
\end{equation}
\end{cor}
The corollary follows at once from the last bound in Proposition \ref{propdispNgrand} and Proposition \ref{propcardN}:
\begin{multline}
|G^{+,\sharp}_{h,\gamma}(t,x,y,a,0,0)|\lesssim \sum_{N\in \mathcal{N}_{1}(x,y,t)} |W_{N,a}(T,X,Y)|\lesssim \frac{h^{1/3}\lambda^{2/3}}{h^{2}N} \frac{|T|}{\lambda^{2}}
\lesssim \frac{h^{1/3}}{h^{2}  \lambda^{4/3}}\,.
\end{multline}
We will prove the three propositions in reverse order.
\subsubsection{Proof of Proposition \ref{propdispNgrand}}
We start with $\lambda^{1/3}\lesssim N\lesssim \lambda$: we have $W_{N,a}(T,X,Y)$ from Proposition \ref{prop3}, \eqref{eq:bis488ter999999machin}, and $A_c=A_c(\frac{T}{4N}, \frac{S+\Upsilon}{2N},a)$ from Lemma \ref{lemAcsol}. For $\phi_{N,a}$ given in \eqref{critphaNpetit}, we will prove the following, uniformly in $\eta\in\text{supp}(\psi)$
\begin{equation}
  \label{eq:90}
   \left| \int_{\R^{2}} e^{i \lambda \eta \phi_{N,a}(T,X,\Upsilon,S)}  \chi_{3}(S,\Upsilon,a,1/N,h,\eta) \, dS d\Upsilon \right| \lesssim \frac{\lambda^{-2/3}}{1+\lambda^{1/6}|K_{\infty}^2(a)-1|^{1/4}}\,.
\end{equation}
Re-scale variables with $\Upsilon=\lambda^{-1/3}\tilde p$ and $S=\lambda^{-1/3}\tilde q$ and define $P=\lambda^{2/3}(K_{\infty}(a)^{2}-X)$  and $Q=\lambda^{2/3}(K_{\infty}(a)^{2}-1)$. We actually prove, uniformly in $(P,Q)$:
\begin{equation}
  \label{eq:91}
   \left| \int_{\R^{2}} e^{i\eta \tilde\Phi_{N,a,\lambda}}  \chi_{3}(\lambda^{-1/3} \tilde q,\lambda^{-1/3}\tilde p ,a,1/N,h,\eta) \, d\tilde p d\tilde q \right| \lesssim \frac{1}{1+|P|^{1/4}+|Q|^{1/4}}\,,
\end{equation}
where the rescaled phase $\tilde \Phi_{N,a,\lambda}(T,X,\tilde q,\tilde p)=\lambda\phi_{N,a}(T,X,\lambda^{-1/3}\tilde p,\lambda^{-1/3} \tilde q)$,    
is such that
\begin{align}
  \partial_{\tilde q}\tilde \Phi_{N,a,\lambda} &= 
\lambda^{2/3}\Big(S^2+1-A_c(K,\frac{S+\Upsilon}{2N},a)\Big)|_{(S,\Upsilon)=(\lambda^{-1/3}\tilde q,\lambda^{-1/3}\tilde p)},\\
 \partial_{\tilde p}\tilde \Phi_{N,a,\lambda} 
&= \lambda^{2/3}\Big(\Upsilon^2+X-A_c(K,\frac{S+\Upsilon}{2N},a)\Big)|_{(S,\Upsilon)=(\lambda^{-1/3}\tilde q,\lambda^{-1/3}\tilde p)}.
\end{align}
Using Lemma \ref{lemAcsol}, in the new variables
\[
A_c(K,\lambda^{-1/3}\frac{(\tilde q+\tilde p)}{2N}, a)=\Big(K_{\infty}(a)-\lambda^{-1/3}\frac{(\tilde q+\tilde p)}{2N}(1-a\E)\Big)^2.
\]
With these notations, the first order derivatives of $\tilde \Phi_{N,a,\lambda}$ read as 
\[
\partial_{\tilde q}\tilde \Phi_{N,a,\lambda}=\tilde q^2-Q+\frac{\lambda^{1/3}}{N}K_{\infty}(a)(\tilde p+\tilde q)(1-a\E)-\frac{(\tilde p+\tilde q)^2}{4N^2}(1-a\E)^2,
\]
\[
\partial_{\tilde p}\tilde \Phi_{N,a,\lambda}=\tilde p^2-P+\frac{\lambda^{1/3}}{N}K_{\infty}(a)(\tilde p+\tilde q)(1-a\E)-\frac{(\tilde p+\tilde q)^2}{4N^2}(1-a\E)^2,
\]
where we recall that $\E$ is a smooth, uniformly bounded function. As $\lambda^{1/3}\lesssim N$, if $Q,P$ are bounded, then \eqref{eq:91} obviously holds for bounded $(\tilde p,\tilde q)$  and by integration by parts if $|(\tilde p,\tilde q)|$ is large. So we can assume that $|(Q,P)|\geq r_{0}$ with $r_{0}\gg1$. Set $(Q,P)=r \exp(i\theta)=r(\sin\theta, \cos\theta)$ and re-scale again $(\tilde p,\tilde q)=r^{1/2} (p,q)$: we aim at proving \eqref{eq:91} in that range, which is now
\begin{equation}
  \label{eq:93}
    \left| \int_{\R^{2}} e^{i r^{3/2}  \eta\Phi_{N,a,\lambda}}  \chi_{3}(\lambda^{-1/3}r^{1/2}  q,\lambda^{-1/3}r^{1/2} p ,a,1/N,h,\eta) \, d p d q \right| \lesssim \frac{1}{r^{5/4}}\,.
\end{equation}
Now $r_{0}<r\lesssim \lambda^{2/3}$ (indeed, $r\sim |(P,Q)|\lesssim \lambda ^{2/3}$), $r^{3/2}$ is our large parameter,  we have $\Phi_{N,a,\lambda}(T,X,q,p)=r^{-3/2}\tilde \Phi_{N,a,\lambda}(T,X,r^{1/2}q,r^{1/2}p)$ and
\begin{align}\label{derivphipq}
  \partial_{ p} \Phi_{N,a,\lambda} & = p^{2}-\cos\theta+ \frac{\lambda^{1/3} K_{\infty}(a)( p+ q)}{ N r^{1/2}}(1-a\E)-\frac{( p+ q)^{2}}{4 N^{2}}(1-a\E)^2, \\
 \partial_{ q} \Phi_{N,a,\lambda} & = q^{2}-\sin\theta+ \frac{\lambda^{1/3} K_{\infty}(a)( p+ q)}{ N r^{1/2}}(1-a\E)-\frac{( p+ q)^{2}}{4 N^{2}} (1-a\E)^2,
\end{align}
where, abusing notations, $\E$ is now  $\E(K,r^{1/2}\frac{\lambda^{-1/3}}{2N}(p+q),a)$. On the support of the symbol $\chi_3(\lambda^{-1/3}r^{1/2}p,\lambda^{-1/3}r^{1/2}q,a,1/N,h,\eta)$ we have $|( p, q)|\lesssim \lambda^{1/3} r^{-1/2}< \lambda^{1/3} r_0^{-1/2}$, and therefore, for $\lambda^{1/3}\lesssim N$, the last term in both derivatives is $O(r_{0}^{-1})$, while the next to last term is $r_{0}^{-1/2}O(p+q)$. Hence, when $|(p,q)|>M$ with $M$ sufficiently large, the corresponding part of the integral is $O(r^{-\infty})$ by integration by parts. So we are left with restricting our integral to a compact region in $(p,q)$, renaming the symbol $\chi_{4}$. We remark that everything is symmetrical with respect to $P$ and $Q$ and we deal with $P\geq Q$, that is to say, $\cos \theta \geq \sin \theta$ and therefore $\theta\in (-\frac{3\pi}{4},\frac{\pi}{4})$. We proceed depending upon the size of $Q=r\sin\theta$. If $\sin \theta<-C/r^{1/2}$ for a sufficiently large $C>0$, then $\partial_{q} \Phi_{N,a,\lambda}>c/(2r^{1/2})$ for some $C>c>0$ and the phase is non stationary.
Indeed, in this case we have
\[
\partial_{ q}\Phi_{N,a,\lambda}\geq  q^{2}+\frac{C}{2 r^{1/2}}+\frac{\lambda^{1/3}}{N} K_{\infty}(a)\frac{(p+ q)}{r^{1/2}}(1-a\E)-\frac{(p+q)^{2}}{4 N^{2}}(1-a\E)^2\,;
\]
using boundedness of $(p, q)$,  $|r^{1/2}(p, q)|\lesssim \lambda^{1/3}$ (from the support of $\chi_{3}$), and $1\ll\lambda^{1/3}\lesssim N$, we then have
\[
\frac{\lambda^{1/3}}{r^{1/2}N} (p+q)\Big[K_{\infty}(a)-\frac{r^{1/2}(p+ q)}{N\lambda^{1/3}}(1-a\E)\Big]\leq \frac{C}{4r^{1/2}}.
\]
It follows that $\partial_{ q}\Phi_{N,a,\lambda}>C/(4r^{1/2})$. Next, let $\sin \theta >-C/r^{1/2}$ and assume $P>0$ (otherwise apply non-stationary phase), which in turn implies $P>r_{0}/2$. Indeed, $\cos \theta \geq \sin \theta>-C/r^{1/2}$ implies that $\theta\in (-\frac{C}{\sqrt{r_0}},\frac{\pi}{4})$ and therefore in this regime we have $\cos\theta\geq \frac{\sqrt{2}}{2}$. Finally, consider the case $|\sin \theta |<C/r^{1/2}$ for $C>0$ like before.
Non degenerate stationary phase applies in $p$, at two (almost) opposite values of $ p$, such that $| p_{\pm}|\sim |\pm\sqrt{\cos\theta}|\geq 1/4$, which can be written as follows
\begin{multline}
  \label{eq:95}
      r \int_{\R^{2}} e^{i r^{3/2} \eta\Phi_{N,a,\lambda}} \chi_{4}(p, q ,a,1/N,h,\eta) \, d p d q \\ 
      = \frac r {r^{3/4}} \left( \int_{\R} e^{i r^{3/2} \eta\Phi_{+,N,a,\lambda}}  \chi^{+}( q ,a,1/N,h,\eta) \,  d q\right. 
{}+\left. \int_{\R} e^{i r^{3/2} \Phi_{-,N,a,\lambda}}  \chi^{-}(q ,a,1/N,h,\eta) \, d q\right)\,.
\end{multline}
Indeed, the phase is stationary in $p$ when
\[
p^2=\cos\theta-\frac{\lambda^{1/3}K_{\infty}(a)}{Nr^{1/2}}(p+q)(1-a\E)+\frac{(p+q)^2}{4N^2}(1-a\E)^2,
\]
and from $\cos\theta\geq \frac{\sqrt{2}}{2}$ and $\frac{1}{r}\leq \frac{1}{r_0}\ll 1$, $\partial_p\Phi_{N,a,\lambda}=0$  has exactly two (separate) solutions, that we denote $p_{\pm}=\pm\sqrt{\cos\theta}+O(r^{-1/2})$. Using \eqref{derivphipq}, at these critical points,
\[
\partial^2_{p}\Phi_{N,a,\lambda}|_{
p_{\pm}}=2p+\frac{\lambda^{1/3}K_{\infty}(a)}{Nr^{1/2}}(1+O(a))+O(N^{-2})|_{p_{\pm}},
\]
where we used boundedness of $(p,q)$,  $\partial_{p}\E=O(\frac{r^{1/2}\lambda^{-1/3}}{N})$ to deduce smallness of all the terms except the first one. Then $\lambda^{1/3}\lesssim N$, $r^{-1/2}\ll 1$, boundedness of $K_{\infty}(a)$ (close to $1$) together imply that for $p\in\{p_{\pm}\}$, we have $\partial^2_{p}\Phi_{N,a,\lambda}|_{p_{\pm}}= 2p_{\pm}+O(r^{-1/2})$ and $|p_{\pm}|\geq \frac 14-O(r^{-1/2})$; therefore stationary phase applies. The critical values, denoted $\Phi_{\pm,N,a,\lambda}$, are such that
\begin{align}\label{derivPhi3q}
\partial_q \Phi_{\pm,N,a,\lambda}(q,.) & =\partial_q\Phi_{N,a,\lambda}(q,p_{\pm},.)\nonumber \\
 & = (q^{2}-\sin\theta+ \frac{\lambda^{1/3} K_{\infty}(a)( p+ q)}{ N r^{1/2}}(1-a\E)-\frac{( p+ q)^{2}}{4 N^{2}} (1-a\E)^2)({p=p_{\pm}}).
\end{align}
As $|\sin \theta |<C/r^{1/2}$, the remaining phases $\Phi_{-,N,a,\lambda}$ and $\Phi_{+,N,a,\lambda}$ may be stationary but degenerate. However, taking two derivatives in \eqref{derivPhi3q}, one easily checks that $|  \partial^{3}_{q} \Phi_{\pm,N,a,\lambda}|\geq 2-O(r_{0}^{-1/2})$. Hence we get, by Van der Corput lemma,
\begin{equation}
  \label{eq:97}
\left|  \int_{\R} e^{i r^{3/2}\eta\Phi_{\pm,N,a,\lambda}}  \chi^{\pm}(q ,a,1/N,h,\eta) \,  d q\right| \lesssim (r^{3/2})^{-1/3}\,,
\end{equation}
which in turn implies the remaining part of \eqref{eq:93} for our current range of $(P,Q)$,
\begin{equation}
  \label{eq:98}
  \left|   \int_{\R^{2}} e^{i r^{3/2} \eta \Phi_{N,a,\lambda}}  \chi_{4}(p, q, a,1/N,h,\eta) \, d p d q\right|\lesssim r^{-5/4}.
\end{equation}
Notice moreover that $|Q|=|r\sin\theta|\leq C r^{1/2}$, hence from $r^{2}=P^{2}+Q^{2}$, we have $P\sim r$ and $ 1/ {r^{1/4}}\lesssim 1/{(1+|Q|^{1/2})}$; under the restriction $P>Q$, e.g. $X<1$, a better estimate holds,
\begin{equation}
  \label{eq:1001}
     \left| \int_{\R^{2}} e^{i \lambda \eta\phi_{N,a}}  \chi(s,\sigma,a,1/N,h,\eta) \, ds d\sigma \right| \lesssim \frac{\lambda^{-2/3}}{(1+|Q|^{1/2})}\,.
\end{equation}
In the last case $\sin \theta>C/r^{1/2}$, which means $P\geq Q\geq Cr^{1/2}$, stationary phase  holds in $(p,q)$ : the determinant of the Hessian matrix is at least $C\sqrt{\cos \theta}\sqrt{\sin \theta}$ and we get the following bound for the integral in \eqref{eq:98}
$$
\frac C {(\sqrt{\cos \theta}\sqrt{\sin \theta})^{1/2} r^{3/2}}  \lesssim\frac 1 r \frac 1  {(r \sqrt{\cos \theta}\sqrt{\sin \theta})^{1/2}}\lesssim \frac 1 r \frac{1}{|PQ|^{1/4}},
$$
so in this case we get (compare to \eqref{eq:90})
\begin{equation}
  \label{eq:100}
     \left| \int_{\R^{2}} e^{i \lambda \phi_{N,a}}  \chi(\lambda^{-1/3} \tilde q,\lambda^{-1/3} \tilde p,a,1/N,h,\eta) \, d\tilde p d\tilde q \right| \lesssim \frac 1 {\lambda^{2/3}|PQ|^{1/4}}\leq \frac 1 {\lambda^{2/3}r^{1/4}}\,.
   \end{equation}
   Of course with $P\geq Q$ the $r^{1/4}$ factor may be improved to $|Q|^{1/2}$. However, we need to consider the symmetrical case, in which case we ultimately retain the $r^{1/4}$ factor, which always yields $|Q|^{1/4}$ irrespective of the relative positions of $P$ and $Q$. Hence, in all cases, with $|Q|\leq r$, recalling that $Q=\lambda^{2/3}(K^2_{\infty}(a)-1)=\lambda^{2/3}(K_{\infty}(a)+1)(K-1)(1+O(a))$ and $K=T/(4N)$,
   \[
|W_{N,a}(T,X,Y)|\lesssim \frac{h^{1/3}}{h^2}\frac{\lambda^{5/6}}{N^{1/2}} \frac 1 {\lambda^{2/3}(1+r^{1/4})}\lesssim \frac{h^{1/3}}{h^2}\frac{1}{(N/\lambda^{1/3})^{1/2}+N^{1/4}|T-4N|^{1/4}}\,,
\]
and in the particular case $X<1$,
\[
  |W_{N,a}(T,X,Y)|\lesssim \frac{h^{1/3}}{h^2}\frac{\lambda^{5/6}}{N^{1/2}} \frac 1 {\lambda^{2/3}(1+|Q|^{1/2})}\lesssim  \frac{h^{1/3}}{h^2}\frac{1}{(N/\lambda^{1/3})^{1/2}+\lambda^{1/6}|T-4N|^{1/2}}\,.
\]
When $|N|\geq \lambda$ the proof proceeds similarly : just replace $(K_{\infty}(a)^2-1)$ by $A_{c|0}-1$ in \eqref{eq:90}. When $|N|\leq\lambda^2$ we use \eqref{TNAc0} to replace $A_{c|0}-1$ in \eqref{eq:90} by $\frac{T}{4N}(1+O(\lambda^{-2}))-1$. When $|N|\geq \lambda^2$, we cannot take advantage of \eqref{TNAc0} anymore since $|T-4N(1+O(\lambda^{-2}))|=|T-4N+O(N/\lambda^2)|$ and the last term can be large. We may use $A_{c|0}-1=4(\frac{3Y}{T}-1)(1+O(a))$ whose infimum is always $0$; notice that in this case in the first order derivatives of $\tilde\Phi_{N,a,\lambda}$ we can keep only the first two terms $(\tilde q^2-Q,\tilde p^2-P)$, $Q=\lambda^{2/3}(A_{c|0}-1)$, $P=\lambda^{2/3}(A_{c|0}-X)$, since the part of $A_c$ that depends on $\tilde p,\tilde q$ is too small to oscillate and we can bring it in the symbol. In fact, discarding the terms depending on $\tilde p,\tilde q$ in $A_c$ gives essentially a product of two Airy functions whose worst decay is $(1+|P|)^{-1/4}(1+|Q|)^{-1/4}$. This concludes the proof of Proposition \ref{propdispNgrand}.\qed

\subsubsection{Proof of Propositions \ref{propdispNpetitloin} and \ref{propdispNpetitpres}}
As $1\leq N<\lambda^{1/3}\ll \lambda$, we move $\exp(iNB(\lambda A^{3/2}))$ (from the phase $\Psi_{N,a,a}$ of $W_N$) into the symbol; the critical point $A_c$ is given in Lemma \ref{lemAcsol}, the critical value for the phase is $\eta\phi_{N,a}$ and $\phi_{N,a}$ (defined in \eqref{critphaNpetit}) does not depend on $\eta$. 
The following bound is proved in \cite{Annals} (for a phase that was constructed differently), uniformly for $\eta\in \text{supp}(\psi)$,
\begin{equation}
  \label{eq:90bis}
   \left| \int_{\R^{2}} e^{i \lambda \eta\phi_{N,a}(T,X,S,\Upsilon)}  \chi_{3}(S,\Upsilon,a,1/N,h,\eta) \, dS d\Upsilon \right| \lesssim N^{1/4}\lambda^{-3/4}\,.
 \end{equation}
 Informally, the decay should be understood as resulting from a non degenerate stationary phase in one variable, followed by an application of Van der Corput lemma (with a non vanishing fourth derivative in the remaining variable). This accounts for the $1/2+1/4=3/4$ exponent on the large parameter. We now obtain better bounds, either because in the remaining variable the phase has non vanishing derivative of order three (generating $1/2+1/3=5/6$ decay) or two. In doing so, we uncover the geometry of the curves on which the phase may degenerate.
 
Set $\Lambda=\lambda/N^{3}$ to be the new (large) parameter. Re-scale again variables with  $S={q}/N$ and $\Upsilon={p}/N$ and set $\Lambda \tilde\Phi_{N,a}(T,X,{p},{q})=\lambda \phi_{N,a}(T,X,{p}/N,{q}/N)$. On the support of $\chi$ we then have $|({q},{p})|\lesssim N$. We are reduced to proving
\begin{equation}
  \label{eq:91hh}
   \left| \int_{\R^{2}} e^{i \Lambda \eta\tilde\Phi_{N,a}}  \chi({q}/N,{p}/N ,a,1/N,h,\eta) \, d{p} d{q} \right| \lesssim \Lambda^{-3/4}\,.
\end{equation} 
We have $
\nabla_{({q},{p})}\tilde\Phi_{N,a}  = \Big({q}^2+N^2(1-A_c), {p}^2+N^2(X-A_c)\Big)$, where, using Lemma \ref{lemAcsol}, 
\begin{equation}\label{eq:alpha_c22}
A_c\left(K,{\frac{({q}+{p})}{2N^2}},a\right)=\left(K_{\infty}(a)-\frac{({q}+{p})}{2N^2}(1-a\E(K,w,a))\right)^2_{\bigl|w=\frac{({q}+{p})}{2N^2}}.
\end{equation}
We define ${P}=(K^{2}_{\infty}(a)-X)N^{2}$ and ${Q}=(K^{2}_{\infty}(a)-1)N^{2}$. With these notations, 
\begin{align*}
  \partial_{{q}}\tilde \Phi_{N,a} & ={q}^2-{Q}+K_{\infty}(a)({q}+{p})(1-a\E)-\frac{1}{4N^2}({q} +{p})^2(1-a\E)^2\,\\
\partial_{{p}}\tilde\Phi_{N,a} & ={p}^2-{P}+K_{\infty}(a)({q}+{p})(1-a\E)-\frac{1}{4N^2}({q} +{p})^2(1-a\E)^2\,.
\end{align*}
\begin{rmq}
For $N$ sufficiently small even the terms with $\frac 1N$ may provide important contributions. At this stage and given that all variables were properly rescaled with respect to $a$, the reader may, at first, set $a=0$ (and even $N=1$ !) to make all subsequent computations more straightforward while capturing the correct asymptotics.
\end{rmq}
We start with $|({P},{Q})|\geq r_{0}$ for some large, fixed $r_{0}$, in which case we can follow the same approach as in the previous case. Set again ${P}=r\cos \theta$ and ${Q}=r\sin \theta$. If $|({p},{q})|<r_{0}/2$, then the corresponding integral is non stationary and we get decay by integration by parts. We change variables $({p},{q})=r^{1/2}( p', q')$ with $r_0\leq r\lesssim N^2$ and aim at proving
\begin{equation}
  \label{eq:101}
     \left| r  \int_{\R^{2}} e^{i r^{3/2} \Lambda \eta\Phi_{N,a}}  \chi(r^{1/2} q'/N,r^{1/2}  p'/N ,a,1/N,h,\eta) \, d p' d q' \right| \lesssim r^{-1/4 }\Lambda^{-5/6}\,,
\end{equation}
where $\chi$ is compactly supported and $\Phi_{N,a}(T,X, p', q'):=r^{-3/2}\tilde\Phi_{N,a}(T,X,r^{1/2} p',r^{1/2} q')$. Compute 
\begin{align*}
  \partial_{ p'} \Phi_{N,a} & =p'^{2}-\cos\theta+ \frac{ K_{\infty}(a)}{r^{1/2}}(q'+p')(1-a\E)-\frac{(q'+p')^2(1-a\E)^{2}}{4 N^{2}}, \\
 \partial_{ q'} \Phi_{N,a} & = q'^{2}-\sin\theta+ \frac{ K_{\infty}(a)}{r^{1/2}}(q'+p')(1-a\E)-\frac{(q'+p')^2(1-a\E)^{2}}{4 N^{2}}\,.
 \end{align*}
As in the previous case,  $\cos\theta$ and $\sin\theta$ play symmetrical parts: hence we set ${P}\geq  {Q}$, $\cos\theta\geq \sin \theta$. If $|( p', q')|\geq M$ for some large $M\geq 1$, then, for critical points, $ p'^2_c\geq  q'^2_c$ and if $M$ is sufficiently large non-stationary phase applies in $p'$. Therefore we are reduced, again, to bounded $|(p',q')|$. We deal with three cases, depending upon ${Q}=r\sin\theta $: if $\sin \theta<-\frac{C}{\sqrt{r}}$ for some sufficiently large constant $C>0$, then
\[
\partial_{ q'} \Phi_N \geq  q'^{2}+\frac{C}{r^{1/2}}+ \frac{ K_{\infty}(a)}{r^{1/2}}(q'+p')(1-a\E)-\frac{(q'+ p')^2(1-a\E)^{2}}{4 N^{2}}\,.
\]
As $|( p',  q')|$ is bounded, $\E$ is bounded, $N$ is sufficiently large (from $N>\sqrt{r}\geq \sqrt{r_0}$) and  $\frac{1}{\sqrt{r}}\geq \frac 1N$, it follows that non-stationary phase applies: the sum of the last three terms in the previous inequality is greater than $C/(2r^{1/2})$ for $C$ large enough. If $|\sin\theta|\leq \frac{C}{\sqrt{r}}$ then, again, $\theta\in (-\frac{C}{\sqrt{r_0}},\frac{\pi}{4})$ and $\cos\theta\geq \frac{\sqrt{2}}{2}$. We have $|{Q}| =|r\sin\theta|\leq C \sqrt{r}$; if $|{Q}|<C$, then $1+|{Q}|\lesssim r^{1/2}$, while $|{P}|\sim r$. As in the previous case the stationary phase applies in $p'$ with non-degenerate critical points $p'_{\pm}$ and yields a factor $(r^{3/2}\Lambda)^{-1/2}$; the critical values $ \Phi_{\pm,N,a}$ of the phase function at these critical points are such that $|\partial^3_{q'} \Phi_{\pm,N,a}|\geq 2-O(\frac{1}{\sqrt{r_0}})$ and therefore the integral with respect to $ q'$ is bounded by $(r^{3/2}\Lambda)^{-1/3}$ by Van der Corput lemma. We obtain \eqref{eq:101}. If we are interested solely in $X<1$, we may bound $r^{-1/4}\lesssim (1+|{Q}|^{1/2})^{-1}$. Finally, if $\sin\theta > \frac{C}{\sqrt{r}}$, then ${Q}=r\sin\theta>C\sqrt{r}$ and therefore $N^2|K^2_{\infty}(a)-1|>Cr^{1/2}$. We directly perform stationary phase with large parameter $r^{3/2}\Lambda$ as the determinant of the Hessian matrix at the critical point is at least $C\sqrt{\cos\theta\sin \theta}$: this yields the following bound for the left hand side term in \eqref{eq:101}, where the last inequality holds only for $X<1$
\[
\frac{c r}{(\sqrt{\sin\theta}\sqrt{\cos\theta})^{1/2}r^{3/2}\Lambda}=\frac{1}{\Lambda}\frac{1}{({P}{Q})^{1/4}}\leq \frac{1}{\Lambda}\frac{1}{{Q}^{1/2}}\,.
\]
Considering we may exchange $P$ and $Q$, we have just proved Proposition \ref{propdispNpetitloin}: for $N<\lambda^{1/3}$ and $|T-4N|\gtrsim 1/N$,
\begin{align*}
|W_{N,a}(T,X,Y)|&=\frac{h^{1/3}\lambda^{4/3}}{h^2\sqrt{N}\sqrt{\lambda}N^2}\Big | r  \int_{\R^{2}} e^{i r^{3/2} \Lambda\eta \Phi_{N,a}}  \chi(r^{1/2} p'/N,r^{1/2}  q'/N ,a,1/N,h,\eta) \, d p' d q'\Big |\\
&\lesssim \frac{h^{1/3}\lambda^{5/6}}{h^2N^{5/2}}r^{-1/4}\Big(\frac{\lambda}{N^3}\Big)^{-5/6}\lesssim \frac{h^{1/3}}{h^2}\frac{1}{(1+|{Q}|^{1/4})}\\
& \lesssim\frac{h^{1/3}}{h^2}\frac{1}{(1+N^{1/4}|T-4N|^{1/4})}\,.
\end{align*}
In the special case $X<1$, we may replace $N^{1/4}|T-4N|^{1/4}$ by $N^{1/2}|T-4N|^{1/2}$.

We now move to the most delicate case $|({P},{Q})|\leq r_{0}$. For $|({p},{q})|$ large, the phase is non stationary and integrations by parts provide $O(\Lambda^{-\infty})$ decay. So we may replace $\chi$ by a cutoff, that we still call $\chi$, that is compactly supported in $|({p},{q})|<R$.
We will improve the estimates from \cite{Annals} that were sharp only at ${p}={q}=0$ on this limiting case and prove
\begin{equation}\label{toexpect}
   \left| I:= \int_{\mathbb{R}^{2}} e^{i \Lambda \eta \tilde\Phi_{N,a}}  \chi({q}/N,{p}/N ,a,1/N,h,\eta) \, d{p} d{q} \right| \lesssim \frac{\Lambda^{-5/6}}{\Lambda^{-1/12}+|{Q}|^{1/6}}\,.
\end{equation}
In \cite{Annals}, we proved a general lemma covering the most degenerate cases. Here, as explained earlier, we just proceed by identifying one variable where the usual stationary phase may be performed, and then evaluate the remaining 1D oscillating integral using Van der Corput lemma with different decay rates depending on the lower bounds on derivatives of order at most $4$. Replacing $A_c$ using \eqref{eq:alpha_c22} in $\partial_{{q}}\tilde \Phi_{N,a}  ={q}^2+N^2(1-A_{c\bigl|w=\frac{({p}+{q})}{2N^2}})$ and $\partial_{{p}}\tilde \Phi_{N,a} ={p}^2+N^2(X-A_{c\bigl|w=\frac{({p}+{q})}{2N^2}})$,
\begin{align*}
  \partial_{{q}}\tilde \Phi_{N,a} & ={q}^2-{Q}+K_{\infty}(a)({q}+{p})(1-a\E)-\frac{1}{4N^2}({q} +{p})^2(1-a\E)^2,\\
\partial_{{p}}\tilde\Phi_{N,a} & ={p}^2-{P}+K_{\infty}(a)({q}+{p})(1-a\E)-\frac{1}{4N^2}({q} +{p})^2(1-a\E)^2.
\end{align*}
Define $H_N$ such that $-{p}{P}-{q}{Q}+H_N({q},{p},X,T)=\tilde \Phi_{N,a}({q},{p},X,T)$, then
\begin{align}
  \label{eq:H_Nderivq}
\partial_{{q}}H_N  & ={q}^2+K_{\infty}(a)({q}+{p})(1-a\E)-\frac{1}{4N^2}({q} +{p})^2(1-a\E)^2\,,\\
\label{eq:H_Nderivp}
\partial_{{p}}H_N  & ={p}^2+K_{\infty}(a)({q}+{p})(1-a\E)-\frac{1}{4N^2}({q} +{p})^2(1-a\E)^2\,.
\end{align}
Moreover, the second order derivatives of $H_N$ follow directly from those of $\tilde\Phi_{N,a}$ :
\begin{align}\label{secondderH_NG_N}
 \partial^2_{{q}}H_N & =\partial^2_{qq}\tilde\Phi_{N,a} =2{q}-2N^2A_c^{1/2}\partial_{w}(A_c^{1/2})\frac{\partial w}{\partial {q}}=
2q-A_c^{1/2}\partial_{w}(A_c^{1/2})_{\bigl|w=\frac{({p}+{q})}{2N^2}},\\
\partial^2_{{p}}H_N &=\partial^2_{pp}\tilde\Phi_{N,a} =2{p}-2N^2A_c^{1/2}\partial_{w}(A_c^{1/2})\frac{\partial w}{\partial {p}}=
2p-A_c^{1/2}\partial_{w}(A_c^{1/2})_{\bigl|w=\frac{({p}+{q})}{2N^2}},\\
 \partial^2_{{q},{p}}H_N &=\partial^2_{qp}\tilde\Phi_{N,a}  =2N^2A_c^{1/2}\partial_{w}(A_c^{1/2})\frac{\partial w}{\partial {q}}=
-A_c^{1/2}\partial_{w}(A_c^{1/2})_{\bigl|w=\frac{({p}+{q})}{2N^2}},
\end{align}
and we recall from \eqref{deriveeAc} that the derivative of $A_c$ with respect to $w$ is given by 
\[
\partial_{w}(A_c^{1/2})(K,w,a)=-\frac{1}{1+aA_c^{1/2}\frac{(A_c^{1/2}+w)^3}{K^2(1+a)}}.
\]
The determinant of the Hessian matrix of $H_N$ reads as follows
\begin{align}
  \label{eq:103}
  \nonumber
 \text{det Hess }H_N & =4{p}{q}-2 ({p}+{q})A_c^{1/2}\partial_{w}(A_c^{1/2})_{\bigl|w=\frac{({p}+{q})}{2N^2}}\,\\
  &=4{p}{q}+2 ({p}+{q})\frac{(K_{\infty}(a)-\frac{({q}+{p})}{2N^2}(1-a\E))}{1+aA_c^{1/2}\frac{(A_c^{1/2}+w)^3}{K^2(1+a)}} {}_{\bigl|w=\frac{({p}+{q})}{2N^2}}\,.
\end{align}
When $\text{ det Hess} H_N$ is away from $0$, the usual stationary phase applies, so we expect the worst contributions to occur in a neighborhood of $\mathcal{C}_N=\{({q},{p}),\text{det Hess }H_N=0\}$. Informally, for $|N|\neq 0$, the equation defining $\mathcal{C}_{N}$ will be close to either a parabola ($|N|=1)$ or an hyperbola ($|N|\geq 2)$:
\begin{equation*}
\mathcal{C}_{\pm 1\bigl|a=0}=\Big\{ (p-q)^2=2K(p+q)\Big\}\,\,\text{ and, for}\,|N|\geq 2\,,\,\,\mathcal{C}_{N\bigl|a=0}=\Big\{ 4p q +2 K(p+q)=\frac{(p+q)^{2}}{N^{2}}\Big\}\,.
\end{equation*}
These curves suggest to rotate variables: let $\xi_{1}=({p}+{q})/2$ and $\xi_{2}=({p}-{q})/2$. Then ${p}=\xi_{1}+\xi_{2}$ and ${q}=\xi_{1}-\xi_{2}$, and setting $h_N(\xi_1,\xi_2):=-\tilde \Phi_{N,a}({p},{q})$, from the above definition of $H_{N}$ we get
\begin{align*}
    h_N(\xi_1,\xi_2)&=(\xi_1+\xi_2){P}+(\xi_1-\xi_2){Q}-H_N(\xi_1-\xi_2,\xi_1+\xi_2) \\
    & =\xi_{1}M_{1}+\xi_{2} M_{2}-H_N(\xi_1-\xi_2,\xi_1+\xi_2),
\end{align*}
where we set $M_{1}={P}+{Q}$ and $M_{2}={P}-{Q}$. 
Using $\Big|\frac{\partial({p},{q})} {\partial(\xi_1,\xi_2)}\Big|=2$, 
\begin{align}
  \label{eq:106}
 \frac 1 2 \text{ det Hess}_{(\xi_1,\xi_2)} h_N & =  
\Big(4{p}{q}+2({p}+{q})\frac{(K_{\infty}(a)-\frac{({q}+{p})}{2N^2}(1-a\E))}{1+aA_c^{1/2}\frac{(A_c^{1/2}+w)^3}{K^2(1+a)}}\Big)_{\bigl|w=\frac{({p}+{q})}{2N^2}, {p}=\xi_1+\xi_2, {q}=\xi_1-\xi_2}\\
\nonumber   & = 4\Big[\xi_{1}^{2}\Big(1-\frac{1}{N^2}(1-a\underline \E)\Big)-\xi_{2}^{2}+K_{\infty}(a)\xi_{1}\Big(\frac{1-a\underline \E}{1-a\E}\Big)\Big]_{\bigl|w=\frac{\xi_1}{N^2}},
\end{align}
where we used \eqref{eq:103} and set $1-a\underline \E(K,w,a):=({1-a\E(K,w,a)})({1+aA_c^{1/2}\frac{(A_c^{1/2}+w)^3}{K^2(1+a)}})^{-1}$,  for $w=\frac{\xi_1}{N^2}$. Outside a small neighborhood of the set $\{\text{det Hess }_{(\xi_1,\xi_2)}h_N= 0\}$, the stationary phase applies in both $(\xi_1,\xi_2)$; in fact in \cite{Annals} we focused on degenerate critical points from this set. Here, we let $|N|\geq 1$ and we will consider all cases irrespective of the Hessian. It should nevertheless be clear from the proof that the most degenerate cases are in a small neighborhood of $\text{det Hess}_{(\xi_1,\xi_2)}h_N=0$.
We compute first $\partial^2_{\xi_2}h_N=-4\xi_1$: using that $\frac{\partial {p}}{\partial\xi_2}=-\frac{\partial {q}}{\partial\xi_2}=-1$, 
\begin{align}
\label{firstderivH_Nxi2}
\partial_{\xi_2}(H_N(\xi_1-\xi_2,\xi_1+\xi_2)) & =(\partial_{{p}}H_N-\partial_{{q}}H_N)_{\bigl|{p}=\xi_1+\xi_2, {q}=\xi_1-\xi_2}\\
\partial^2_{\xi_2}(H_N(\xi_1-\xi_2,\xi_1+\xi_2)) & =\Big(\partial^2_{{p}}H_N-2\partial^2_{{p}, {q}}H_N+\partial^2_{{q}}H_N\Big)_{\bigl|{p}=\xi_1+\xi_2, {q}=\xi_1-\xi_2}\,.
\end{align}
Using now \eqref{secondderH_NG_N} and replacing $({p},{q})$ by $(\xi_1+\xi_2, \xi_1-\xi_2)$ yields
\begin{equation}\label{secderivxi1xi1}
\partial^2_{\xi_2}h_N(\xi_1,\xi_2)=-\partial^2_{\xi_2}(H_N(\xi_1-\xi_2,\xi_1+\xi_2))=-4\xi_1.
\end{equation}
\paragraph{\bf Case $|\xi_1|\geq c>0$}
From \eqref{secderivxi1xi1},  stationary phase in $\xi_2$ applies, with large parameter $\Lambda$. Using \eqref{firstderivH_Nxi2}, the critical point is such that $M_{2}=\partial_{{p}}H_N({q},{p})-\partial_{{q}}H_N({q},{p})$. Using \eqref{eq:H_Nderivq} and \eqref{eq:H_Nderivp} and replacing ${p}, {q}$ by $(\xi_1+\xi_2),(\xi_1-\xi_2)$, we get $M_{2}=4\xi_1\xi_2$, so that $\xi_{2,c}=\frac{M_{2}}{4\xi_1}$. Next, compute higher order derivatives for the critical value of the phase $h_N(\xi_1,\xi_2)$ at $\xi_2=\xi_{2,c}$:
\begin{align}\label{critptnondegxi1}
\partial_{\xi_1}(h_N(\xi_1,\xi_{2,c}))& =\partial_{\xi_1}h_N(\xi_1,\xi_2)|_{\xi_2=\xi_{2,c}}+\frac{\partial\xi_{2,c}}{\partial\xi_1}\partial_{\xi_2}h_N(\xi_1,\xi_2)|_{\xi_2=\xi_{2,c}}
=\partial_{\xi_1}h_N(\xi_1,\xi_2)|_{\xi_2=\xi_{2,c}}\\
\nonumber
&=M_{1}-(\partial_{{p}}H_N+\partial_{{q}}H_N)_{\bigl|{q}=\xi_1-\xi_2,{p}=\xi_1+\xi_2}\\
\nonumber
&=M_{1}-2\xi_{1}^{2}\Big(1-\frac{1}{N^2}(1-a\E)^2\Big)-2\xi_{2,c}^{2}-4K_{\infty}(a)\xi_{1}(1-a\E)_{\bigl|w=\frac{\xi_{1}}{N^2}}\,.
\end{align}
The second order derivative of $h_N(\xi_1,\xi_{2,c})$ with respect to $\xi_1$ equals
\begin{align}\label{seccritptnondegxi1}
\partial^2_{\xi_1}(h_N(\xi_1,\xi_{2,c})) & =-4\Big[\xi_1\Big(1-\frac{1}{N^2}(1-a\E)^2-\frac{a}{N^2}(1-a\E)\xi_1\partial_{\xi_{1}}\E\Big)-\frac{M_{2}^2}{16\xi_1^3}
\\ & \quad\quad\quad\quad\quad{}+K_{\infty}(a)\Big((1-a\E)-a\xi_1\partial_{\xi_{1}}\E\Big)\Big] \nonumber\\
& =-4\Big[\xi_1\Big(1-\frac{1}{N^2}(1-a\tilde \E)\Big)-\frac{M_{2}^2}{16\xi_1^3}+K_{\infty}(a)\frac{1-a\tilde \E}{1-a\E}\Big]\,,\nonumber
\end{align}
where $\tilde \E$ is defined as
\begin{equation}\label{deftildef}
1-a\tilde \E(K,\xi_{1}/N^{2},a):=(1-a\E(K,\xi_{1}/N^{2},a)(1-a\E(K,\xi_{1}/N^{2},a) -a\xi_{1}\partial_{\xi_{1}}(E(K,\xi_{1}/N^{2},a))\,.
\end{equation}
\begin{lemma}
For $|N|\geq 1$, the critical point $\xi_1$ of $h_N(\xi_1,\xi_{2,c})$ can be degenerate of order at most $2$.
\end{lemma}
We need to prove that, when $\partial_{\xi_1}(h_N(\xi_1,\xi_{2,c}))=\partial^2_{\xi_1}(h_N(\xi_1,\xi_{2,c}))=0$, the third order derivative doesn't vanish. Taking the derivative of \eqref{seccritptnondegxi1} with respect to $\xi_1$ yields
\begin{equation}\label{thrdcritptnondegxi1}
\partial^3_{\xi_1}(h_N(\xi_1,\xi_{2,c}))
=-4\Big[\Big(1-\frac{1}{N^2}(1-a\partial_{\xi_1}(\xi_1\tilde \E))\Big)+\frac{3M_{2}^2}{16\xi_1^4}+\frac{aK_{\infty}(a)\breve{\E}}{(1-a\E)^2}\Big],
\end{equation}
where we set $\breve{\E}:=(1-a\tilde \E)\partial_{\xi_1}\E-(1-a \E)\partial_{\xi_1}\tilde \E$. For all $|N|\geq 2$, the bracket in the right hand side term of \eqref{thrdcritptnondegxi1} is strictly positive and we may apply Van der Corput lemma with non-vanishing third order derivative. We are left with $|N|=1$ where the third order derivative may vanish when $M_{2}$ is very small : but in this case, the second derivative given in \eqref{seccritptnondegxi1} doesn't vanish and we can apply the usual stationary phase. Indeed, recall that  $(\xi_1,\xi_2)$ is  bounded (from $({p},{q})$ bounded). Therefore, when $N^2=1$, the coefficient of $\xi_1$ in \eqref{seccritptnondegxi1} is $O(a)$ and for $M_{2}$ such that $|M_{2}|\leq 4c^2$ we have $|\frac{M_{2}^2}{16\xi_1^3}|\leq c$ for every $|\xi_1|\geq c$. We conclude using that $K_{\infty}(a)\sim 1$. As such, the contribution of the set $|\xi_{1}|>c$  \eqref{toexpect} is at most $C(c)\times 1/\Lambda^{1/2} \times 1/\Lambda^{1/3}$, where the first factor is the non degenerate stationary phase in $\xi_{2}$ and the second one is coming from Van der Corput in $\xi_{1}$, irrespective of the value of $N$. We stress that we made no use of the value of the Hessian here.\qed
\paragraph{\bf Case $|\xi_1|\leq c$}
Let now $\xi_1$ belong to a small neighborhood of $0$, $|\xi_1|\leq c<1/2$, then stationary phase (with non-degenerate critical point) applies in $\xi_1$: using \eqref{eq:H_Nderivq} and \eqref{eq:H_Nderivp}, compute (see \eqref{critptnondegxi1})
\begin{align}
\partial_{\xi_1}h_N(\xi_1,\xi_2) 
\label{eq:stath_N} & = M_{1}-2\xi_{1}^{2}\Big(1-\frac{1}{N^2}(1-a\E)^2\Big)-2\xi_{2}^{2}-4K_{\infty}(a)\xi_{1}(1-a\E)_{\bigl|w=\frac{\xi_1}{N^2}}\\
\label{eq:secderxi1hN}
\partial^2_{\xi_1}h_N(\xi_1,\xi_2) & =-4\Big(\xi_1(1-\frac{1}{N^2}(1-a\tilde \E))+K_{\infty}(a)\frac{1-a\tilde \E}{1-a\E}\Big),
\end{align}
where $\tilde \E$ was defined in \eqref{deftildef}.
As  $a\ll 1$ and $\tilde \E, \E$ are uniformly bounded, $\partial^2_{\xi_{1}} h_{N}$  cannot vanish provided $\xi_1$ is sufficiently small.
We denote $\xi_{1,c}$ the solution to $\partial_{\xi_1}h_N(\xi_1,\xi_2)=0$, then, using \eqref{eq:stath_N}, 
\begin{equation}\label{xi1c}
 \xi_{1,c}^{2}\Big(1-\frac{1}{N^2}(1-a\E)^2\Big)+2K_{\infty}(a)\xi_{1,c}(1-a\E)\Big|_{w=\frac{\xi_{1,c}}{N^2}}=\frac  {M_{1}}2-\xi_{2}^{2}.
 \end{equation}
\begin{rmq}\label{rmqboundsxi1cxi2}
Since $|M_{1}|=|{P}+{Q}|\leq 2 r_0$ with $r_0$ fixed, it follows from \eqref{xi1c} that
 both $\xi_{1,c}$ and $\xi_2$ remain bounded at stationary points and $|M_{1}/2-\xi_2^2|\leq 4c$.
\end{rmq}
Recall that $\E$ is a function of $(K,\frac{\xi_{1}}{N^2},a)$ (and, in particular, independent of $\xi_2$), uniformly bounded, which implies that the same holds within parenthesis in \eqref{xi1c}. Consider $a=0$, then \eqref{xi1c} has an unique, explicit solution that reads as follows (we chose the solution that is closer to zero, as both $\xi_{1}$ and $M_{1}/2-\xi_{2}^{2}$ are small)
\[
\xi_{1,c|{a=0}}=(M_{1}/2-\xi_2^2)F_0(M_{1}/2-\xi_2^2,1/N^2,K),
\]
\[
F_0(M_{1}/2-\xi_2^2,1/N^2,K)=\frac{1}{K+\sqrt{K^2+(M_{1}/2-\xi_2^2)(1-1/N^2)}},\quad \forall N\neq 0.
\]
\begin{lemma}
The solution to $\partial_{\xi_1}h_N=0$ reads as follows
\begin{equation}\label{eq:formxi1cxi2}
\xi_{1,c}=(M_{1}/2-\xi_2^2)F(M_{1}/2-\xi_2^2,1/N^2,K,a),
\end{equation}
where $F$ is a smooth function in all the variables such that $|\partial^k_{\xi_2}F|\leq C_k$, for all $k\geq 0$, where $C_k$ are positive constants. Moreover, $F(M_{1}/2-\xi_2^2,1/N^2,K,0)=F_0(M_{1}/2-\xi_2^2,1/N^2,K)$.
\end{lemma}
\begin{proof}
From $\E$ being a function of $ (K,\xi_{1}/N^2,a)$, we let
\[
H(\xi_{1},1/N^2,K,a,z):=\xi_{1}^{2}\Big(1-\frac{1}{N^2}(1-a\E)^2\Big)+2K_{\infty}(a)\xi_{1}(1-a\E)-z
\]
and therefore \eqref{xi1c} translates into $H(\xi_{1,c},1/N^2,K,a,M_{1}/2-\xi_2^2)=0$. As both $\xi_{1}$ and $z$ are small, we may set $\xi_{1}=z F(z)$, and apply the implicit function theorem to $\tilde H(F,z,a)=z F^{2}(1-(1-a\E^{2}(zF))/N^{2})+2K_{\infty}(a)F (1-a\E(zF))-1$: notice that $\tilde H(1/(2K),0,0)=0$ and $\partial_{F}\tilde H(1/(2K),0,0)=2K$. Let $z=M_{1}/2-\xi_2^2$. We get a function $F(z,1/N^2,K,a)$ such that
\[
\xi_{1,c}=zF(z,\cdot), \quad \partial_{z} (\xi_{1,c})=\frac{1}{\partial_{z}H(F(z,\cdot),\cdot)+1}.
\]
In doing so, we may possibly reduce the size of the constant $c$ without loss of generality. Knowing the explicit formula for $a=0$ yields $F(\cdot,0)=F_{0}$.
\end{proof}
Let $\tilde h_N(\xi_2):=h_N(\xi_{1,c},\xi_2)$, with $\xi_{1,c}$ given in \eqref{eq:formxi1cxi2}: as
\[
\partial_{\xi_2}\tilde h_N(\xi_2)=\partial_{\xi_2}\xi_{1,c}\partial_{\xi_1}h_N(\xi_{1,c},\xi_2)+\partial_{\xi_2}h_N(\xi_1,\xi_2)|_{\xi_1=\xi_{1,c}}=\partial_{\xi_2}h_N(\xi_1,\xi_2)|_{\xi_1=\xi_{1,c}},
\]
it follows from \eqref{firstderivH_Nxi2} that
$\partial_{\xi_2}\tilde h_N=0$ if and only if $M_{2}-(\partial_{{p}}H_N-\partial_{{q}}H_N)({q},{p})|_{(\xi_1,\xi_2)}=0$. Using \eqref{eq:H_Nderivq} and \eqref{eq:H_Nderivp}, we have  $\partial_{\xi_2}\tilde h_N=0$ if and only if $M_{2}-4\xi_{1,c}\xi_2=0$. We need the second derivative of $\tilde h_N$ that we will compute using $\partial_{\xi_2}\tilde h_N(\xi_2)=M_{2}-4\xi_{1,c}\xi_2$. Since 
\[
\partial^2_{\xi_2}\tilde h_N =-4(\partial_{\xi_2}\xi_{1,c}\xi_2+\xi_{1,c}),
\]
we first compute the derivatives of $\xi_{1,c}$ with respect to $\xi_2$. Using \eqref{eq:formxi1cxi2}, we have, for $z=M_{1}/2-\xi_2^2$,
\[ 
\partial_{\xi_2}\xi_{1,c}=\frac{dz}{d\xi_2}\partial_z(zF(z,\cdot))=-2\xi_2\tilde F(M_{1}/2-\xi_2^2,\cdot),
\]
where we have set $\tilde F(z,\cdot)=F(z,\cdot)+z\partial_z F(z,\cdot)$.
The second derivative of $\tilde h_N$ is then
\begin{align}
\partial^2_{\xi_2}\tilde h_N & =-4(\partial_{\xi_2}\xi_{1,c}\xi_2+\xi_{1,c})\nonumber\\
 \label{eq:d2hN}                                 &=-4\Big((M_{1}/2-\xi_2^2)F(M_{1}/2-\xi_2^2,\cdot)-2\xi_2^2\tilde F(M_{1}/2-\xi_2^2,\cdot)\Big)\\
                                    &=-4\Big((M_{1}/2-\xi_2^2)(F+2\tilde F)(M_{1}/2-\xi_2^2,\cdot)-M_{1}\tilde F(M_{1}/2-\xi_2^2,\cdot)\Big)\nonumber \,.
\end{align}
As $|z|<4c$ and $a\ll 1$, if $|M_{1}|>15c$ we have $|\partial^{2}_{\xi_{2}}\tilde h_{N}|\geq c/K$ and the usual stationary phase in $\xi_{2}$ will hold (notice how this is consistent with being away from a small neighborhood of $\mathcal{C}_{N}$, see \eqref{eq:106}). Since $z$ is small, critical points are degenerate if and only if 
\begin{equation}\label{eqgedxi2}
M_{1} F(z,\cdot)-3 z F(z,\cdot)+ z(M_{1}-2z)\partial_{z}F(z,\cdot)=0\,.
\end{equation}
As both $M_{1}$ and $z$ are now of size at most $15c$, the third term is $O(c^{2})$ and we set $z=M_{1}Z(M_{1},a)$: we let $H(Z,M_{1},a)=F(M_{1}Z,\cdot)-3ZF(M_{1}Z,\cdot)+Z(M_{1}-2M_{1}Z)(\partial_{z} F)(M_{1}Z,\cdot)$, so that $H(1/3,0,0)=0$ and $\partial_{Z}H(1/3,0,0)\sim F(0,\cdot,0)\neq 0$, so that $z=M_{1}Z(M_{1},a)$, with $Z(0,0)=1/3$. Moreover, we may compute exactly $Z(M_{1},0)=Z_{0}(M_{1},K,1/N^{2})$. Therefore, we have
\begin{equation}\label{implicitxi2}
(M_{1}/2-\xi_2^2)=M_{1} Z(M_{1},1/N^2,K,a),\quad Z(0,1/N^{2},K,0)=1/3.
\end{equation}
Therefore, degenerate critical points $\xi_{2,\pm}$ (solution to $\partial^2_{\xi_2}\tilde h_N=0$) only exist if $M_{1}\geq 0$ (otherwise, non degenerate stationary phase applies and provides better decay $1/\sqrt{(-M_{1}\Lambda)}$, but this case will be subsumed in a later application of Van der Corput). These $\xi_{2,\pm}$ are functions of $\sqrt{M_{1}}$ that coincide at $M_{1}=0$ (at which they are both equal to $0$.) In fact, from\eqref{implicitxi2} we get
\begin{equation}\label{implicitAxi2}
\xi_{2,\pm}=\pm\frac{\sqrt{M_{1}}}{\sqrt{6}}(1+O(M_{1},a))\,.
\end{equation}
We now compute the second and third derivatives of $\xi_{1,c}$ with respect to $\xi_2$ :
\begin{align*}
\partial^2_{\xi_2}\xi_{1,c} & =\frac{d^2z}{d\xi_2^2}\partial_z (zF(z))+(\frac{dz}{d\xi_2})^2\partial^2_{z}(zF(z))|_{z=(M_{1}/2-2\xi_2^2)}\\
 & =-2\tilde F\big({M_{1}}/2-2\xi_2^2\big)+4\xi_2^2\partial_{z} \tilde F\big({M_{1}}/22-2\xi_2^2\big),
\end{align*}
and therefore
\begin{multline}\label{thirdxi2}
  \partial^3_{\xi_2}\tilde h_N(\xi_2)=-4(2\partial_{\xi_2}\xi_{1,c}+\xi_2\partial^2_{\xi_2}\xi_{1,c})\\
  =24\xi_2\Big(\tilde F(M_{1}/2-2\xi_2^2,\cdot)+\frac 23 \xi_2^2\partial_{z}\tilde F(M_{1}/2-2\xi_2^2,\cdot)\Big)
  =\frac{12 \xi_2}{K}(1+O(\xi_2^2,a)).
\end{multline}
At $M_{1}=0$ the order of degeneracy is higher as $\xi_{2,\pm}|_{M_{1}=0}=0$ and $\partial^3_{\xi_2}\tilde h_N(\xi_{2,\pm})|_{M_{1}=0}=0$. However, the fourth derivative doesn't cancel at $\xi_{2,\pm}$ as it stays close to $12/K$ for small $a$ (recall the $O(\cdot)$ of \eqref{thirdxi2} is a smooth function of its arguments, hence we do not need to compute further derivatives of $\xi_{1,c}$ to get the leading term). Going back to $\partial_{\xi_{2}}\tilde h_{N}=0$, we deduce that $M_{2}=O(c^{3/2})$. When degenerate critical points exist,
from $M_{2}=4\xi_{1,c}|_{\xi_{2,\pm}}\xi_{2,\pm}$, using \eqref{eq:formxi1cxi2} and \eqref{implicitAxi2}, we have
\begin{equation}
  \label{eq:17}
  M_{2}=4\frac{\xi_{2,\pm}^3}{K}(1+O(M_{1},a))=\pm\frac{\sqrt{2}}{3\sqrt{3}K}M_{1}^{3/2}(1+O(M_{1},a))\,.
\end{equation}
Hence, at degenerate critical points, we must have \eqref{eq:17}, which is at leading order the equation of a cusp. We may now conclude in the small neighborhood $B_{c}=\{ |\xi_{1}|+|\xi_{2}|^{2}+|M_{1}|+|M_{2}|^{2/3}\}\lesssim c$ (as outside the usual stationary phase in both variables applies) by using Van der Corput lemma on the remaining oscillatory integral in $\xi_{2}$, with phase $h_{N}(\xi_{2})$: $\partial^{4}_{\xi_{2}} h_{N}$ is bounded from below, which yields an upper bound $\Lambda^{-1/4}$ (uniformly in all parameters). When moreover $M_{1}\neq 0$, the third derivative of the phase is bounded from below by $|\xi_{2}|/K$: either $|\xi_{2}^{2}-M_{1}/6|<|M_{1}|/12$ and then  $|\partial^{3}_{\xi_{2}} h_{N}|$ is bounded from below by $\sqrt{|M_{1}|}/(12K)$, or $|\xi_{2}^{2}-M_{1}/6|>(|M_{1}|/12)$ and, going back to \eqref{eq:d2hN}, $\partial^{2}_{\xi_{2}} h_{N}$ is bounded from below by $|M_{1}|/(12K)$:

this yields an upper bound $(\sqrt{|M_{1}|} \Lambda)^{-1/3}+({|M_{1}|} \Lambda)^{-1/2}$. Both terms are better than $\Lambda^{-1/4}$ provided that $|M_{1}|>\Lambda^{-1/2}$, and then $(\sqrt{|M_{1}|} \Lambda)^{-1/3}\geq ({|M_{1}|} \Lambda)^{-1/2}$. Finally, we obtained
\begin{equation}
  \label{eq:112}
  |I|\lesssim \frac 1 {\Lambda^{1/2}}\inf\left(\frac 1 {\Lambda^{1/4}},  \frac{1}{|M_{1}|^{1/6} \Lambda ^{1/3}}\right)\,.
\end{equation}
From $M_{1}={P}+{Q}$ and $M_{2}={P}-{Q}\sim \pm |M_{1}|^{3/2}\ll |M_{1}|$, ${P}\sim {Q}$ and therefore $|M_{1}|\sim 2 {|Q|}$, which is our desired bound, as the non degenerate stationary phase in $\xi_{1}$ provided the factor $\Lambda^{-1/2}$. \qed
\begin{rmq}
  The same rotation of the $(S,\Upsilon)$ variables may also be applied when $|N|>\lambda^{1/3}$ (and $|N|<\lambda^{1/3}$, $(P,Q)$ large). While longer than the direct argument we used, it does provide more insight on the geometry of the wave front. However, the worst possible non degenerate stationary phase scenario, around  $N\sim \lambda^{2/3}$ provides the same $\lambda^{2/3}$ decay as the most degenerate case : singular points cannot be singled out in decay estimates, even if we know precisely where they are. This set of $N$'s turns out to be the worst case scenario for Strichartz estimates later.
\end{rmq}

\subsubsection{Proof of Proposition \ref{propdisptangN=0}}
We first prove that for $|T|\leq \frac 52$ only $W_{N,a}(T,X,Y)$ with $N\in \{0,\pm 1\}$ provide a non-trivial contribution. 
The phase $\Psi_{N,a,a}$ is stationary for $A$  such that
\begin{equation}\label{criteqAN=0}
\frac{T\sqrt{1+a}}{2\sqrt{1+aA}}=2\sqrt{A}N+\Upsilon+S,\quad N\in\mathbb{Z},
\end{equation}
and stationary for $\Upsilon,S$ such that $S^2=A-1$ and $\Upsilon^2=A-X$, with $X\geq 0$; therefore we must have ${|\Upsilon+S|^{2}}<{4A}$.
Let $|T|\leq \frac 52$ and consider $|N|\geq 2$: if $\Psi_{N,a,a}(T,\cdot)$ is stationary in $(A,\Upsilon,S)$ then 
\[
\frac{5\sqrt{1+a}}{8\sqrt{1+aA}}\geq\Big|\frac{T\sqrt{1+a}}{2\sqrt{1+aA}}\Big|\geq 2\sqrt{A}|N|-|\Upsilon+S|\geq 2\sqrt{A}\Big(2-\frac{|\Upsilon+S|}{2\sqrt{A}}\Big)\geq 2\sqrt{A},
\]
which yields $\sqrt{A}\leq \frac{5}{16}$, but for such small values of $A$ the phase would be non-stationary in $S$ (and we are out of the support of $\psi_{3}(A)$.) Therefore, for such $T$'s, the sum over $N$ in $G^{+,\sharp}_{h,\gamma}(T,\cdot)$ reduces to $N\in \{0,\pm 1\}$. Let $N=0$: we will apply the stationary phase in $(A,\Upsilon,S)$ as long as $|\Upsilon+S|\geq c$ for some small constant $c>0$. Compute
\begin{equation}
  \label{eq:11}
  \det\text{Hess}_{A,\Upsilon,S}\Psi_{0,a,a}(T,\cdot)|_{\nabla \Psi_{0,a,a}=0}=2|\Upsilon+S|\Big[1+4a\frac{S\Upsilon(\Upsilon+S)^2}{(1+a)T^2}\Big],
\end{equation}
and the critical points satisfy \eqref{criteqAN=0} with $N=0$, which yields $4a\frac{S\Upsilon(\Upsilon+S)^2}{(1+a)T^2}\leq aA\ll 1$. Therefore \eqref{eq:11} only vanishes for $\Upsilon+S=0$. Write $1=\chi_0(\Upsilon+S)+(1-\chi_0)(\Upsilon+S)$, where $\chi_0\in C^{\infty}_{0}(-2c,2c)$ and $\chi_{0}=1$ on $[-c,c]$ for some small $c$ and write $W_{0,a}:=W_{0,a,\chi_0}+W_{0,a,1-\chi_0}$ accordingly, by splitting the symbol. Then in $W_{0,a,1-\chi_0}$ the usual stationary phase applies with large parameter $\lambda$ and yields
\[
|W_{0,a,1-\chi_0}(T,X,Y)|\lesssim \frac{a^2}{h^3}\lambda^{-3/2}\sim\frac{1}{h^2}\frac{h^{1/2}}{a^{1/4}}\leq \frac{1}{h^2}\Big(\frac ht\Big)^{1/2}.
\]
On the support of $W_{0,a,\chi_0}$ we have to sort out more cases: using \eqref{criteqAN=0}, $\Psi_{0,a,a}$ is non-stationary in $A$ on the support of $\chi_0(S+\Upsilon)$ if $|T|\geq 4c$ or if $|X-1|\geq 4c$. Set $\xi_1=\Upsilon+S$, $\xi_2=S-\Upsilon$, and
\begin{align}
g(\xi_1,\xi_2,A)  := & \Psi_{0,a,a}(T,X,Y,(\xi_1-\xi_2)/2,(\xi_1+\xi_2)/2, A,1)\nonumber\\
  = & \begin{multlined}[t] \frac{(\xi_1-\xi_2)^3}{3}+X\frac{(\xi_1-\xi_2)}2+\frac{(\xi_1+\xi_2)^3}{3}+\frac{(\xi_1+\xi_2)}2\\ {}-\xi_1A +T\frac{\sqrt{1+a}(A-1)}{\sqrt{1+aA}+\sqrt{1+a}}\,,
    \end{multlined}
\end{align}
with $\xi_1\leq 2c$, $|T|\leq 4c$, $|X-1|\leq 4c$. Then the stationary phase applies in $(\xi_1,A)$: we have
\begin{equation}\label{xi1AcritN=0}
\partial_{\xi_1}g=2\xi_1^2+2\xi_2^2+(1+X-2A)/2,\quad \partial_A g=-\xi_1+\frac{T\sqrt{1+a}}{2\sqrt{1+aA}},
\end{equation}
then $\partial^2_{\xi_1,A}g=-1$, while $|(\partial^2_{\xi_1}g)(\partial^2_{A}g)|\sim |aT\xi_1|=O(a)$ and by stationary phase we get a decay factor $\lambda^{-1}$ and are left with the integration in $\xi_2$. If we denote $\xi_{1,c}, A_c$ the critical points satisfying $\nabla_{\xi_1,A}g=0$, we have
\[
\partial_{\xi_2}g(\xi_{1,c},\xi_2,A_c)=4\xi_{1,c}\xi_2+(1-X)/2,\quad \partial^2_{\xi_2}g(\xi_{1,c},\xi_2,A_c)=4\xi_{1,c}+4\xi_2\partial_{\xi_2}\xi_{1,c}.
\]
Using \eqref{xi1AcritN=0} we compute
\[
\xi_{1,c}^2\Big(1+2a(\xi_{1,c}^2+\xi_2^2+(1+X)/4)\Big)=\frac{T^2}{4}(1+a),
\]
therefore $\xi_{1,c}=\frac T2(1+O(a)$, $\partial_{\xi_2}\xi_{1,c}=-2a\xi_2\xi_{1,c}(1+O(a))$ and $\partial^2_{\xi_2}g(\xi_{1,c},\xi_2,A_c,\cdot)=4\xi_{1,c}(1+O(a\xi_2^2))$. If $\lambda |T|\geq M$, for some large $M>1$ then stationary phase applies in $\xi_2$ and yields an additional factor $(\lambda |T|)^{-1/2}$. We eventually find for some large $M>1$,
\[
|W_{0,a,\chi_0}(T,X,Y)|\lesssim \frac{a^2}{h^3}
\left\{ \begin{array}{l}
\lambda^{-1}(\lambda |T|)^{-1/2},\quad \text{ if }  M/\lambda\leq |T| (\leq 5c), \\ 
\lambda^{-1}, \quad \text{ if } |T|\leq M/\lambda.
 \end{array} \right.
\]
For $|T|\in [M/\lambda, 5c]$, we have $\frac{a^2}{h^3}\lambda^{-1}(\lambda |T|)^{-1/2}=\frac{1}{h^2}(\frac ht)^{1/2}$. For $|T|=\frac{|t|}{\sqrt{a}}\leq M/\lambda=M h/a^{3/2}$ we have $\frac{a^2}{h^3}\lambda^{-1}=\frac{\sqrt{a}}{h^2}\leq \frac{M}{h^2}(\frac ht)^{1/2}$ since $t>Mh/a$. This yields $|W_{0,a,\chi_0}(T,X,Y)|\lesssim \frac{1}{h^2}(\frac ht)^{1/2}$. Let now $N=1$ and $|T|\leq \frac 52$ : using \eqref{criteqAN=0} with $N=1$ yields
\begin{equation}\label{ineqTN=1}
\frac{5\sqrt{1+a}}{8\sqrt{1+aA}}\geq \frac{|T|\sqrt{1+a}}{2\sqrt{1+aA}}\geq 2\sqrt{A}-|\Upsilon+S|,
\end{equation}
and therefore $|\Upsilon+S|\geq \frac 38$ when the phase is stationary in $A$. Compute
\[
\det\text{Hess}_{A,\Upsilon,S}\Psi_{1,a,a}(T,\cdot)=\Big|4S\Upsilon\partial^2_{A}\Psi_{1,a,a}(T,\cdot)-2(\Upsilon+S)\Big|,
\]
where $\partial^2_{A}\Psi_{1,a,a}=-\frac{1}{\sqrt{A}}+\partial^2_{A}\Psi_{0,a,a}=-\frac{1}{\sqrt{A}}+O(a)$. If $|S|\leq \frac{1}{16}$ or $|\Upsilon|\leq \frac{1}{16}$, then using $|\Upsilon+S|\geq \frac 38$ we apply stationary phase in $(A,\Upsilon,S)$ since $2|\Upsilon+S|-4\frac{|S\Upsilon|}{\sqrt{A}}\geq \frac 34-\frac{|\Upsilon|}{2\sqrt{A}} \geq \frac 14$. If both $|S|,|\Upsilon|\geq\frac{1}{16}$ and $|2\frac{S\Upsilon}{\sqrt{A}}+\Upsilon+S|\geq c$ for some $c>0$, stationary phase applies. Consider the case $|2\frac{S\Upsilon}{\sqrt{A}}+\Upsilon+S|\leq 2c$ for some small $c$ and $|S|,|\Upsilon|\geq\frac{1}{16}$: then $S,\Upsilon<0$ and the critical points are $S_c=-\sqrt{A-1}$, $\Upsilon_c=-\sqrt{A-X}$. $A$ should be solution to
\[
2\frac{\sqrt{A-X}\sqrt{A-1}}{\sqrt{A}}=\sqrt{A-1}+\sqrt{A-X}\quad \Leftrightarrow\quad \frac{1}{\sqrt{1-1/A}}+\frac{1}{\sqrt{1-X/A}}=2
\]
and since $\frac{1}{\sqrt{1-X/A}}\geq 1$ and $\frac{1}{\sqrt{1-1/A}}= 1+1/(2A)+O(1/(2A)^2)> 1+1/A\geq 1+1/9$, there is no solution. Therefore for critical $A,S,\Upsilon$, 
$|\det\text{Hess}_{A,\Upsilon,S}\Psi_{1,a,a}(T,\cdot)|\geq c>0$. Stationary phase then applies, providing a factor $\lambda^{-3/2}$ and $|W_{1,a}(T,\cdot)|\lesssim \frac{1}{h^2}(\frac ht)^{1/2}$.    
Notice that if $|T|$ takes larger values, \eqref{ineqTN=1} doesn't help anymore to lower bound $|\Upsilon+S|$ : when $|T|\sim 4$ we start to see the first swallowtail in the wavefront set. This completes the proof of Proposition \ref{propdisptangN=0}.\qed
\subsection{Transverse waves for $\gamma >h^{2/3-\varepsilon}$}
We now consider $a<\gamma/4$: re-scale variables,
\begin{equation}
  \label{eq:83}
  x=\gamma X\,,\quad \alpha=\gamma A\,,\quad t=\sqrt \gamma T\,,\quad s=\sqrt \gamma S\,,\quad \sigma=\sqrt \gamma \Upsilon\,,\quad y+t\sqrt{1+\gamma}=\gamma^{3/2} Y\,.
\end{equation}
Let $\lambda_{\gamma}=\gamma^{3/2}/h$ be our large parameter, with $\Psi_{N,a,\gamma}$ introduced in \eqref{PhiNagammadef}, we have
\begin{equation}
  \label{eq:bis488ter9999bis}
G^+_{h,\gamma}(t,x,y,a,0,0)= \frac {\gamma^{2}} {(2\pi h)^{3}} \sum_N \int_{\R^{2}}\int_{\R^{2}} e^{i \lambda  \Psi_{N,a,\gamma}}  \eta^{2} \psi({\eta})  \chi_{1}(\lambda^{2/3}A)  \psi_{2}(A) \, dS d\Upsilon  dA d\eta\,,
\end{equation}
where $\mathrm{supp} \,\psi_{2}\subset [\frac 12,\frac 3 2]$. As critical points in $S,\Upsilon$ are such that $S^2=A-\frac{a}{\gamma}$, $\Upsilon^2=A-X$, using that $A\leq 3/2$ we restrict ourselves to $|S|,|\Upsilon|< 3/2$ without changing the contribution modulo $O(h^{\infty})$. As for $X>A$ there are no real critical points with respect to $\Upsilon$, we may restrict to $X\lesssim 1$ (hence $x\lesssim \gamma$. Actually from symmetry of the Green function $G^{+}_{h,\gamma}$, we could even restrict to $x<a$, e.g. $X<a/\gamma<1/4$ but we will not need it).

Therefore  $G^+_{h,\gamma}(t,x,y,a,0,0)=G^{+,\flat}_{h,\gamma}+O(h^{\infty})$ with $G^{+,\flat}_{h,\gamma}=\sum_N W_{N,\gamma}(T,X,Y)$ and
\begin{equation}
  \label{eq:bis488ter999999bis}
    W_{N,\gamma}(T,X,Y): = \frac {\gamma^{2}} {(2\pi h)^{3}}  \int_{\R^{2}}\int_{\R^{2}} e^{i \lambda \Psi_{N,a,\gamma}}  \eta^{2} \psi({\eta})
 \chi_{2}(S)\chi_{2}(\Upsilon) \psi_{2}(A) \, dS d\Upsilon  dA d\eta\,.
\end{equation}
Here we dropped the $\chi_{1}$ cut-off that is irrelevant since we consider $\lambda_{\gamma}$ to be large. We will prove the following propositions (recall $t=\sqrt \gamma T$, therefore $t\gamma/h=T\lambda_{\gamma}$):
\begin{prop}
  \label{transverseTpp1}
  For $|T|\leq 1$, we have
  \begin{equation}
    \label{eq:decW0}
    |W_{0,\gamma}(T,X,Y)|+|G^{+,\flat}_{h,\gamma}(,x,y,a,0,0)|\lesssim  \frac{\sqrt \gamma}{h^2}\inf \left( 1, \left(\frac{h}{\gamma t}\right)^{1/2}\right)\,.
  \end{equation}
\end{prop}
\begin{prop}
    \label{transverseTpp9}
  For $1\leq |T|\leq 9$, we have
  \[
\sum_{|N|\leq 2} |W_{N,\gamma}(T,X,Y)|\lesssim  \frac{\sqrt \gamma}{h^2}\inf \left(1, \left(\frac{h}{\gamma t}\right)^{1/3}\right)\sim  \frac{\gamma^{1/6}}{h^2}\left(\frac{h}{ t}\right)^{1/3}\,.
\]
\end{prop}
\begin{prop}
    \label{transverseTpg9}
  For $9\leq |T|\lesssim \lambda_{\gamma}^{2}$, we have
  \[
 |W_{N,\gamma}(T,X,Y)|\lesssim  \frac{\sqrt \gamma}{h^2} \frac{h^{1/3}}{\sqrt \gamma N^{1/2}}=  \frac{h^{1/3}\gamma^{1/4}}{{h^{2}}\sqrt t}\,,
\]
\begin{equation}
  \label{eq:7}
  |G^{+,\flat}_{h,\gamma}(t,x,y,a,0,0)|\lesssim  \frac{h^{1/3}\gamma^{1/4}}{{h^{2}}\sqrt t}\,.
\end{equation}
\end{prop}
\begin{prop}
  \label{WgrandTN2}
  For $|T|\gtrsim \lambda_{\gamma}^{2}$, we have
\begin{equation}
  \label{eq:6}
  |W_{N,\gamma}(T,X,Y)|\lesssim  \frac{\gamma^2}{h^{3}\lambda_{\gamma}^{5/6} N}\,,
\end{equation}
\begin{equation}
  \label{eq:8}
  |G^{+,\flat}_{h,\gamma}(t,x,y,a,0,0)|\lesssim \frac {\sqrt \gamma} {h^{2}} \left(\frac{h}{{\gamma}^{3/2}}\right)^{11/6}= \frac {h^{1/3}}{h^{2}\lambda_{\gamma}^{3/2}}\,.
\end{equation}
\end{prop}
We start with $W_{N,\gamma}$. As $a/\gamma <1/4$ and $A>1/2$, $|S_{\pm}|=\sqrt{A-\frac{a}{\gamma}}>1/2$ and the usual stationary phase applies in this variable; the critical values of the phase at $S_{\pm}$ are $\{ \Psi^{\pm}_{N,a,\gamma}(T,X,Y,\Upsilon,A,\eta)\}$, where $\Psi^{\pm}_{N,a,\gamma}(T,X,Y,\Upsilon,A,\eta)$ have been defined in \eqref{Psivarepsdfn} and accordingly we set $W_{N,\gamma}=W^{+}_{N,\gamma}+W^{-}_{N,\gamma}$ . The phases $\Psi^{\pm}_{N,a,\gamma}$ are stationary in $A$ when 
\begin{equation}\label{critAtranv}
\frac{T}{2\sqrt{1+\gamma A}}-\Upsilon\pm\sqrt{A-\frac{a}{\gamma }}=2N\sqrt{A}\Big(1-\frac 34B'(\eta\lambda_{\gamma}A^{3/2})\Big).
\end{equation}
The second derivative with respect to $A$ is
\begin{equation}\label{secderphiN}
\partial^2_{A}\Psi^{\pm}_{N,a,,\gamma}=-\frac{\gamma T}{4\sqrt{1+\gamma A}^3}\pm\frac{1}{2\sqrt{A-a/\gamma}}-\frac{N}{\sqrt{A}}\Big(1-\frac 34B'(z)-\frac 94zB''(z)\Big)_{\bigl|z=\eta\lambda_{\gamma}A^{3/2}}\,.
\end{equation}
If $N=0$, the middle term dominates the first one (as $\gamma\ll 1$). If $|N|\geq 1$, the last term overcomes the next to last: this is obvious unless $|N|=1$, in which case $\sqrt A\sim 2\sqrt{A-a/\gamma}$ would imply $3A\sim a/\gamma<1/4$ which is excluded by the support condition on $A$. Therefore $|\partial^2_{A}\Psi^{\pm}_{N,a,\gamma}|\gtrsim 1$.

The stationary phase in $A$ applies so
we are left to deal with the remaining variables $\Upsilon, \eta$. When $|T|$ is not too large, we will only integrate in $\Upsilon$ and then simply discard the integral in $\eta$ which is bounded as it has compact support; when $|T|> \lambda_{\gamma}^2$ the stationary phase in $\eta$ turns out to be particularly useful.
The critical value of the phase functions $\Psi^{\pm}_{N,a,\gamma}(T,X,Y,\Upsilon,A_c,\eta)$ satisfies
\begin{equation}\label{critUps2}
\partial_{\Upsilon}\Psi^{\pm}_{N,a,\gamma}(T,X,Y,\Upsilon,A_c,\eta)=\Upsilon^2+X-A,
\end{equation}
\[
\partial^2_{\Upsilon}\Psi^{\pm}_{N,a,\gamma{|}A_c}=(\partial^2_{\Upsilon}\Psi^{\pm}_{N,a,\gamma}+2\partial_{\Upsilon}A_c\partial^2_{\Upsilon,A}\Psi^{\pm}_{N,a,\gamma}+(\partial_{\Upsilon}A_c)^2\partial^2_{A}\Psi^{\pm}_{N,a,\gamma})_{|A_c}.
\]
From \eqref{critAtranv} (equation on $A_c=A_c(\Upsilon,T,N,\eta)$), we have $(\partial_{\Upsilon}A_c)(\partial^2_{A}\Psi^{\pm}_{N,a,\gamma}|_{A_c})=1$. Then,
\begin{equation}\label{critsecUps}
\partial^2_{\Upsilon}\Psi^{\pm}_{N,a,\gamma}|_{A_c}=2\Upsilon+\partial_{\Upsilon}A_c(1+2\partial^2_{\Upsilon,A}\Psi^{\pm}_{N,a,\gamma})=2\Upsilon-\partial_{\Upsilon}A_c.
\end{equation}

\begin{lemma}
For a given $|N|\geq 2$ and a given point $(T,X,Y)$ with $T\sim 4N$, the phase $\Psi^{\pm}_{N,a,\gamma|A_{c}}$ has at most one degenerate critical point of order two with respect to $\Upsilon$.

\end{lemma}
\begin{proof}\label{lemN>2}
  Let $|N|\geq 2$. Using \eqref{secderphiN}, $\partial^2_{\Upsilon}\Psi^{\pm}_{N,a,\gamma,|A_c}=0$ becomes $2\Upsilon=\partial_{\Upsilon}A_c$:

  \[
\Upsilon=\frac{-\sqrt{A_c}}{(2N\mp1)(1-\frac 34 B'(z)-\frac 94 zB''(z))\mp \frac 34(B'(z)+3zB''(z))\pm \frac{a/(\gamma\sqrt{A_c-a/\gamma})}{(\sqrt{A_c}+\sqrt{A_c-a/\gamma})}+\frac{\gamma TA_c^{1/2}}{2(1+\gamma A_c)^{3/2}}},
\]
with $z=\eta\lambda_{\gamma}A_c^{3/2}$. From \eqref{critAtranv} we also get
\begin{equation}
\sqrt{A_c}=\frac{\frac{T}{2\sqrt{1+\gamma A_c}}-\Upsilon\mp \frac{a/\gamma}{(\sqrt{A_c}+\sqrt{A_c-a/\gamma})}-\frac 34\sqrt{A_c}B'(\eta\lambda_{\gamma}A_c^{3/2})}{(2N\mp 1)(1-\frac 34 B'(\eta\lambda_{\gamma}A_c^{3/2}))}
\end{equation}
hence a solution to $\partial^2_{\Upsilon}\Psi^{\pm}_{N,a,\gamma,\pm|A_c}=0$ is such that
\[
\Upsilon=\frac{\Upsilon-\frac{T}{2}(1+O(\gamma))+O(a/\gamma)+O(\lambda_{\gamma}^{-2})}{(2N\mp 1)^2(1+O(\lambda_{\gamma}^{-2})+O(\gamma))}\,.
\]
Notice that for every given $T, N$ we obtain an unique solution $\Upsilon_c(T,N,\gamma)$. Asking $\Upsilon_c$ to be a critical point gives also $\Upsilon_c^2=A_c-X$, which provides the relation between $T$ and $X$ at those points where the phase is degenerate of order at least two. To check the order of degeneracy we consider higher order derivatives. We have $\partial^3_{\Upsilon}\Psi^{\pm}_{N,a,\gamma|{A_c}}=2-\partial^2_{\Upsilon}A_c=2+(\partial_{\Upsilon}A_c)^3\partial^3_{A}\Psi^{\pm}_{N,a,\gamma|A_c}$. Compute
\begin{equation}
  \label{thirdderiv}
\partial^3_{A}\Psi^{\pm}_{N,a,,\gamma}|_{A_c}=\frac{3\gamma^2 T}{8\sqrt{1+\gamma A}^5}\mp\frac{1}{4\sqrt{A-a/\gamma}^3}+\frac{N}{2\sqrt{A}^3}(1+O(\lambda_{\gamma}^{-2})),
\end{equation}
hence $\partial^3_{A}\Psi^{\pm}_{N,a,\gamma}|_{A_c}\sim \frac{(N\mp 1/2)}{2A_c^{3/2}}$, while $\partial_{\Upsilon}A_c\sim (\partial^2_{A}\Psi^{\pm}_{N,a,\gamma}|_{A_c})^{-1}\sim -\frac{\sqrt{A_c}}{(N\mp 1/2)}$ and $(\partial_{\Upsilon}A_c)^3\sim -\frac{A_c^{3/2}}{(N\mp 1/2)^3}$ which means that for $|N|\geq 2$ the critical point $\Upsilon_c$ is degenerate of order exactly two and the absolute value of the third derivative is bounded from below by a constant. 
\end{proof}
\begin{lemma}\label{lemN=0}
For $N=0$ and $T\neq 0$,  we have $\partial^2_{\Upsilon}\Psi^{\pm}_{0,a,\gamma}|_{A_c}= T(1+O(\gamma))$.
\end{lemma}
\begin{proof}
Let $N=0$ and $A_c=A_c(\Upsilon,T,N,\eta)$ be the solution to \eqref{critAtranv}; taking the derivative with respect to $\Upsilon$ yields
\[
\partial_{\Upsilon}A_c=\pm \frac{2\sqrt{A_c-a/\gamma}}{1\mp\frac{\gamma T\sqrt{A_c-a/\gamma}}{2(1+\gamma A_c)^{3/2}}}=\frac{2(\Upsilon-\frac{T}{2\sqrt{1+\gamma A_c}})}{1+\frac{\gamma T}{2(1+\gamma A_c)^{3/2}}(\frac{T}{2\sqrt{1+\gamma A_c}}-\Upsilon)},
\]
and we obtain, after replacing $\Upsilon$ with its expression from \eqref{critAtranv}
\begin{equation}
\partial^2_{\Upsilon}\Psi^{\pm}_{0,a,\gamma}|_{A_c} =2\Upsilon-\frac{2(\Upsilon-\frac{T}{2\sqrt{1+\gamma A_c}})}{1+\frac{\gamma T}{2(1+\gamma A_c)^{3/2}}(\frac{T}{2\sqrt{1+\gamma A_c}}-\Upsilon)} =\frac{T}{\sqrt{1+\gamma A_c}}\Big(1\mp\frac{\gamma \sqrt{A_c-a/\gamma}}{2(1+\gamma A_c)}\Big)\,,
\end{equation}
which ends the proof.
\end{proof}
\begin{lemma}\label{lemNpm1} For $N=\pm 1$ the following holds:
\begin{enumerate}
\item each phase function $\Psi^{\mp}_{\pm 1,a,\gamma}(T,X,Y,\Upsilon,A_c,\eta)$ has at most one degenerate critical point  $\Upsilon_c$ of order exactly two. Moreover, $|\Upsilon_c|\leq A_c$ and $|T|\gtrsim \sqrt{A_{c}}$.
\item for $T\neq 0$, we have $\partial^2_{\Upsilon}\Psi^{\pm}_{\pm,a,\gamma}|_{A_c}= T(1+O(\gamma))$.
\end{enumerate}
\end{lemma}
\begin{proof}
Let $N=1$ : for $\Psi^-_{1,a,\gamma}$ the proof follows the one of Lemma \ref{lemN>2}, except for the lower bound on $T$ that is a direct consequence of \eqref{critAtranv}, with $\pm=-$. For $\Psi^+_{1,a,\gamma}$ we proceed as in the proof of Lemma \ref{lemN=0}; notice that in that case, \eqref{critAtranv}, with $\pm=+$ forces $|T|\lesssim 1$.
\end{proof}
\begin{lemma}\label{lemT>lambdasquare}
Let $\lambda_{\gamma}^2\lesssim |N|$ and $\Psi^{\pm}_{N,a,\gamma}$ given in \eqref{Psivarepsdfn}. Then stationary phase applies for $\Psi^{\pm}_{N,a,\gamma}$ in $(A,\eta)$ and, moreover, its critical value $\Psi^{\pm}_{N,a,\gamma}(T,X,Y,\Upsilon,A_c,\eta_c)$ has at most one critical point of order exactly two in $\Upsilon$.
\end{lemma}
\begin{proof}
The derivatives with respect to $A,\eta$ of $\Psi^{\pm}_{N,a,\gamma}$ are given by
\begin{gather}\nonumber
\partial_A(\Psi^{\pm}_{N,a,\gamma})=\eta\Big(\frac{T}{2\sqrt{1+\gamma A}}+\Upsilon^3/3+\Upsilon(X-A)\pm\sqrt{A-a/\gamma}-2N(1-\frac 34 B'(\eta\lambda_{\gamma}A^{3/2}))\Big),\\
\nonumber
\partial_{\eta}(\Psi^{\pm}_{N,a,\gamma})=Y+\Upsilon^3/3+\Upsilon(X-A)\pm \frac 23(A-\frac a \gamma)^{3/2}+T \frac{(\sqrt{1+\gamma A}-\sqrt{1+\gamma})}\gamma\\
\nonumber
-\frac 43NA^{3/2}(1-\frac 34B'(\eta\lambda_{\gamma}A^{3/2})).
\end{gather}
The second order derivatives are 
\begin{gather}\nonumber
\partial^2_{A}(\Psi^{\pm}_{N,a,\gamma})=\text{ RHS term in \eqref{secderphiN}}\,,\nonumber\quad\quad 
\partial^2_{\eta}(\Psi^{\pm}_{N,a,\gamma})=N\lambda_{\gamma}A^3B''(\eta\lambda_{\gamma}A^{3/2})\sim \frac{N}{\lambda_{\gamma}^2},\\
\nonumber
\partial^2_{\eta,A}(\Psi^{\pm}_{N,a,\gamma})=\eta^{-1}\partial_A\Psi^{\pm}_{N,a,\gamma}+\frac 32\eta\lambda_{\gamma}N A^2B''(\eta\lambda_{\gamma}A^{3/2}),
\end{gather}
and when $\nabla_{A,\eta}\Psi^{\pm}_{N,a,\gamma}=0$, the determinant of the Hessian matrix of $\Psi^{\pm}_{N,a,\gamma}$ is
\[
\text{det Hess }\Psi^{\pm}_{N,a,\gamma|{\nabla_{A,\eta}\Psi^{\pm}_{N,a,\gamma}=0}}\sim \frac{N^2}{\lambda_{\gamma}^2},\quad N\geq \lambda_{\gamma}^2\,.
\]
It remains to deal with the integral in $\Upsilon$. We proceed exactly like in the proof of Lemma \ref{lemN>2}: the critical point $A_c$ is given by the same formula while the critical point $\eta_c$ doesn't interfere here, as
\begin{gather*}
  \partial_{\Upsilon}\Big(\Psi^{\pm}_{N,a,\gamma}(T,X,Y,\Upsilon,A_c,\eta_c)\Big)=\eta_c(\Upsilon^2+X-A_c)\,\\
\partial^2_{\Upsilon}\Big(\Psi^{\pm}_{N,a,\gamma}(T,X,Y,\Upsilon,A_c,\eta_c)\Big)=\partial_{\Upsilon}\eta_c(\Upsilon^2+X-A_c)+\eta_c(2\Upsilon-\partial_{\Upsilon}A_c)
\end{gather*}
and at the stationary point $\Upsilon^2+X-A_c=0$, we get the same formula as in \eqref{critsecUps} (with a factor $\eta_c$ near $1$). In the same way, the third derivative is
\[
\partial^3_{\Upsilon}\Big(\Psi^{\pm}_{N,a,\gamma}
(T,X,Y,\Upsilon,A_c,\eta_c)\Big)=\partial^2_{\Upsilon}\eta_c(\Upsilon^2+X-A_c) 
{}+2\partial_{\Upsilon}\eta_c(2\Upsilon-\partial_{\Upsilon}A_c)
+\eta_c(2-\partial^2_{\Upsilon}A_c),
\]
hence when the first and second derivative vanish, the third one behaves exactly like \eqref{thirdderiv}.
\end{proof}
We now consider several cases, depending on $T$. If $|T|\lesssim 1$, we can easily see that $\Psi^{\pm}_{N,a,\gamma}$ can be stationary only for $N\in \{0,\pm 1\}$, since for $|N|\geq 2$ \eqref{critAtranv} has no solution for $A$: $(2N-1)\sqrt A(1-C/\lambda^{2}_{\gamma})\leq 1/2+3/2$ yields for, $|N|=2$, $A<4/9(1+\tilde C/\lambda^{2})$, which is outside the support of $A$ for $\lambda_{\gamma}$ large enough.

Let first $N=0$. Replacing $N=0$ in \eqref{secderphiN} gives $|\partial^2_{A}\Psi^{\pm}_{0,a,\gamma}|\sim 1/2$, hence stationary phase applies in $A$ with large parameter $\lambda_{\gamma}$.
From Lemma \ref{lemN=0} it follows that $\partial^2_{\Upsilon}\Psi^{\pm}_{0,a,\gamma}|_{A_c}\sim T$, where $T=t/\sqrt{\gamma}$ and here $|T|\lesssim 1$: to apply the stationary phase in $\Upsilon$ we need $(\gamma^{3/2}/h)\times (t/\sqrt{\gamma})\gg 1$:
therefore, when $h/\gamma \ll t(\leq \sqrt{\gamma})$, the stationary phase applies and gives
\[
|W_{0,\gamma}(T,X,Y)|\lesssim \frac{\gamma^2}{h^3}\lambda_{\gamma}^{-1}\times \Big(\frac{t\gamma}{h}\Big)^{-1/2} = \frac{\sqrt \gamma}{h^2}\Big(\frac{t\gamma}{h}\Big)^{-1/2}=\frac{1}{h^2}\Big(\frac ht\Big)^{1/2}.
\]
When $t\gamma/h\leq M^2$ for some constant $M>0$, the integration in $A$ doesn't help anymore and 
\[
|W_{0,\gamma}(T,X,Y)|\lesssim \frac{\gamma^2}{h^3}\lambda_{\gamma}^{-1}=\frac{\sqrt{\gamma}}{h^2}\leq \frac{M}{h^2}\Big(\frac ht\Big)^{1/2}.
\]
We have just obtained the first part of \eqref{eq:decW0} in Proposition \ref{transverseTpp1}.
\begin{rmq}
  Notice that this last bound is the same as on a domain without boundary. More precisely, it matches the boundary-less case for a frequency localized Dirac data, with $\eta\sim 1/h$ and $\xi\sim \sqrt\gamma/h$, where dispersion takes over Sobolev embedding for $t>h/\gamma$.
\end{rmq}
Let $N=\pm 1$, then, according to Lemma \ref{lemNpm1}, $\partial^{2}_{\Upsilon}\Psi^{\pm}_{\pm 1,a,\gamma}\sim T$; we conclude as for $N=0$ by splitting according to whether $t\gamma/h\lesssim M^2$ or $t\gamma/h>M^2$. On the other hand, the phases $\Psi^{\mp}_{\pm 1,a,\gamma}$ may have critical points degenerate of order (exactly) two, but they only contribute if $|T|\geq 1$ from Lemma \ref{lemNpm1}. As such, for small $|T|$, the significant contribution to $G^{+}_{h,\gamma}$ is $W_{0,\gamma}+W^{+}_{1,\gamma}+W^{-}_{-1,\gamma}$ and we obtain the second term in \eqref{eq:decW0}, completing the proof of Proposition \ref{transverseTpp1}. \qed

Now, let $1\leq |T|\leq 9$; equation \eqref{critAtranv} has solutions for $N\in \{0,\pm 1,\pm 2\}$ and stationary phase applies in $A$ with large parameter $\lambda_{\gamma}$. When $N=0$, the usual stationary phase applies in $\Upsilon$ since $|T|\geq 1$; when $N=\pm 1$, stationary phase would apply for $W^{\pm}_{\pm 1,\gamma}$ but for $|T|\gtrsim 1$ they are non stationary and do not contribute significantly; according to Lemma \ref{lemNpm1} the phase functions $\Psi^{\mp}_{\pm 1,a,\gamma}$ have degenerate critical points of order exactly two, which provides a bound of the integral in $\Upsilon$ in $W^{\mp}_{\pm 1,\gamma}$ of the form $\lambda_{\gamma}^{-1/3}$. When $N=\pm 2$, the phase functions have degenerate critical points of order two. In this regime we obtain, for $t\in [\sqrt{\gamma},9\sqrt{\gamma}]$
\[
\sum_{|N|\leq 2}|W_{N,\gamma}(T,X,Y)|\leq \frac{\gamma^2}{h^3}\lambda_{\gamma}^{-1-1/3}\sim \frac{1}{h^2}h^{1/3}\sim \frac{1}{h^2}\gamma^{1/6}\Big(\frac ht\Big)^{1/3}\,,
\]
and this completes the proof of Proposition \ref{transverseTpp9}.\qed

Then, when $9\leq |T|\lesssim 4\lambda_{\gamma}^2 $, equation \eqref{critAtranv} has a solution only for $|N|\geq 2$ and the third order derivative with respect to $A$ is bounded from below, therefore we get
$$
|W_{N,\gamma}(T,X,Y)|\lesssim \frac{\gamma^{2}}{h^{3}} \frac{1}{\lambda_{\gamma}^{1/2}}\frac 1 {\sqrt{N\lambda_{\gamma}}} \frac{1}{\lambda_{\gamma}^{1/3}}\,,
$$
where the first factor $\lambda_{\gamma}^{-1/2}$ comes from the stationary phase in $S$, the factor $(N\lambda_{\gamma})^{-1/2}$ from the stationary phase with respect to $A$ and $\lambda_{\gamma}^{-1/3}$ from the degenerate critical point of order two in $\Upsilon$. Using that $t/\sqrt{\gamma}\sim 4N$ we find
\begin{gather*}
  |W_{N,\gamma}(T,X,Y)|\lesssim \frac{\gamma^2}{h^{3}} \frac{h}{\gamma^{3/2}} \frac 1 {\sqrt N} \frac{h^{1/3}}{\gamma^{1/2}}\lesssim \frac 1 {h^{2}} \frac{h^{1/3}}{\sqrt N} \sim \frac 1 {h^{2}} \frac {\gamma^{1/4} h^{1/3}}{t^{1/2}}\,,\\
|G^{+,\flat}_{h,\gamma}(t,x,y,a,0,0)|\lesssim  \sum_{N\in\mathcal{N}_1(x,y,t)}|W_{N,\gamma}(T,X,Y)| \lesssim \frac{1}{h^{2}} \frac {\gamma^{1/4}h^{1/3}}{t^{1/2}},
\end{gather*}
since $ |\mathcal{N}_1|\lesssim O(1)+T/\lambda_{\gamma}^2=O(1)$ as here we consider only $9\lesssim T\lesssim 4\lambda_{\gamma}^2$. Finally, let $4\lambda_{\gamma}^2\lesssim |T|$ hence $\lambda_{\gamma}^2\lesssim |N|$. Since $|T|\lesssim 1/\sqrt{\gamma}$, this regime corresponds to $\gamma\lesssim h^{4/7}$ that hasn't been dealt with in \cite{Annals}. Here, stationary phase applies in $(A,\eta)$. 
Using Lemma \ref{lemT>lambdasquare}, we get
$$
|W_{N}(T,X,Y)|\lesssim \frac{\gamma^{2}}{h^{3}}\lambda_{\gamma}^{-1/2}\lambda_{\gamma}^{-1}(N^2/\lambda_{\gamma}^2)^{-1/2}\lambda_{\gamma}^{-1/3}\lesssim \frac{\gamma^2}{h^{3}}\frac 1 { N} \frac{1}{\lambda_{\gamma}^{5/6}},
$$
where the first factor $\lambda_{\gamma}^{-1/2}$ comes from the integration in $S$, $\lambda_{\gamma}^{-1}(N^2/\lambda_{\gamma}^2)^{-1/2}$ from the stationary phase in $A,\eta$ and $\lambda_{\gamma}^{-1/3}$ from the integral in $\Upsilon$.
For $\lambda_{\gamma}^{2}\lesssim |N|$ 
$$
|G^{+,\flat}_{h,\gamma}(t,x,y,a,0,0)|\lesssim \sum_{N\in \mathcal{N}_{1}(x,y,t)} |W_{N}(T,X,Y)|\lesssim \frac{\gamma^2}{h^{3}} \frac 1 { T} \frac{1}{\lambda_{\gamma}^{5/6}} \frac{T}{\lambda_{\gamma}^{2}}\lesssim \frac 1 {h^{2}} \frac{h^{1/3}}{\lambda_{\gamma}^{\frac {3}{2}}}\,,
$$
and this achieves the proof of Proposition \ref{WgrandTN2}.\qed
\section{Strichartz estimates}
\label{sec:strichartz-estimates}
We intend to prove Theorem \ref{thm+1}. We may reduce ourselves to half-wave frequency localized operators $G^{\pm}_{h}$, and then, to a sum of operators $G^{\pm}_{h,\gamma}$. Assuming we get a bound with a constant $\gamma^{\varepsilon}$, $\varepsilon>0$, we will have the same bound (with fixed constant) on $G^{\pm}_{h}$. We proceed as usual by duality, hereby reducing ourselves to an inhomogeneous estimate on
\begin{equation}
  \label{eq:114}
  u^{\pm}_{h,\gamma}(t,x,y)=\int G^{\pm}_{h,\gamma}(x,y,t,a,b,s) f(a,b,s)\,da db ds\,.
\end{equation}

\subsection{The parametrix regime: $\gamma\gtrsim h^{2/3-\varepsilon}$}

We search for the smallest $q$ such that, for $|t|\lesssim 1$,
\begin{equation}
  \label{eq:115}
  \|u^{\pm}_{h,\gamma}\|_{L^{q}_{t}L^{\infty}_{x,y}} \lesssim \frac {h^{2/q}}{h^{2}} C_{q}(\gamma) \|f\|_{L^{q'}_{s} L^{1}_{a,b}}\,,
\end{equation}
with $C_{q}(\gamma)\lesssim 1$. It will turn out to be convenient to prove
\begin{equation}
  \label{eq:116}
  \|u^{\pm}_{h,\gamma}\|_{L^{p}_{t}L^{\infty}_{x,y}} \lesssim \frac {h^{1/p}}{h^{2}} C_{q}(\gamma) \|f\|_{L^{1}_{s,a,b}}\,,
\end{equation}
with $p=q/2$. As the adjoint of $G^{\pm}_{h,\gamma}$ is $G^{\mp}_{h,\gamma}$, we will recover the previous estimate by duality and interpolation: duality yields
\begin{equation}
  \label{eq:117}
    \|u^{\mp}\|_{L^{\infty}_{t,x,y}} \lesssim \frac {h^{1/p}}{h^{2}} C_{2p}(\gamma) \|f\|_{L^{p'}_{s}L^{1}_{a,b}}\,,
\end{equation}
and interpolation midway between \eqref{eq:116} and \eqref{eq:117} provides \eqref{eq:115}. We may now replace $f\in L^{1}_{s,a,b}$ in \eqref{eq:116} by a Dirac at $(s,a,b)$: then we have $ v^{\pm}(x,y,t)=G^{\pm}_{h,\gamma}(x,y,t,a,b,s)$, and we are left to prove that, for $s\in (-1,1)$, we have $\sup_{a,b,s}  \|v^{\pm}\|_{L^{p}_{t}L^{\infty}_{x,y}} \lesssim \frac {h^{1/p}}{h^{2}} C_{2p}(\gamma)$.
\begin{rmq}
  The whole point of replacing \eqref{eq:115} by \eqref{eq:116} is now clear : we have the $\sup_{a}$ after time integration and we will take advantage of our refined bounds around a discrete time sequence.
\end{rmq}
Notice that $s$ is actually irrelevant, and so is $b$: set $s=b=0$, we would like to prove
\begin{equation}
  \label{eq:120}
  \sup_{a}  \int |G^{\pm }_{h,\gamma}(t,x,y,a)|^{p}_{L^{\infty}_{x,y}}\, dt \lesssim \left (\frac {h^{1/p}}{h^{2}} C_{2p}(h,\gamma)\right)^{p}\,.
\end{equation}
We recall that, at fixed $\gamma$, we know that $0<a\lesssim \gamma$ and $G^{+}_{h,\gamma}=G^{+,\sharp}_{h,\gamma}+G^{+,\flat}_{h,\gamma}+O(h^{\infty})$. Now, if $|t|<\sqrt \gamma$, the operator $G^\pm_{h,\gamma}(t)$ only sees at most one reflection, and as such it satisfies the free case dispersion estimate, as proved in \cite{blsmso08} with a generic boundary. In our particular setting we do have Propositions \ref{propdisptangN=0} and \ref{transverseTpp1} that provide the correct decay estimate, and we get
\begin{equation}
    \label{eq:1201}
  \sup_{a}  \int_{|t|\lesssim \sqrt \gamma } |G^{\pm }_{h,\gamma}(t,x,y,a)|^{p}_{L^{\infty}_{x,y}}\, dt \lesssim \int_{0}^{h/\gamma} \left(\frac {\sqrt \gamma}{h^{2}} \right)^{p}\,dt +\int_{h/\gamma}^{\sqrt \gamma} \left(\frac {h^{1/2}}{h^{2}t^{1/2}} \right)^{p}\,dt
\end{equation}
which is bounded by $h^{1/2}/h^{2}$ for $p=2$ (except for an irrelevant $\log$ that may be removed by computing the weak $L^{2}$ norm). Next, we record the bound from \eqref{eq:1201} but for any $2<p$:
\begin{equation}
    \label{eq:1202}
  \sup_{a}  \left(\int_{|t|\lesssim \sqrt \gamma } |G^{\pm }_{h,\gamma}(t,x,y,a)|^{p}_{L^{\infty}_{x,y}}\, dt\right)^{\frac 1p} \lesssim \frac {h^{\frac 1 p}}{h^{2}} \gamma^{\frac 1 2-\frac 1 p}\,.
\end{equation}
We then split the operator defined by \eqref{eq:114}: $u^{\pm}_{h,\gamma}=u^{\pm,0}_{h,\gamma}+w^{\pm}_{h,\gamma}$, where the kernel of $u^{\pm,0}_{h,\gamma}$ is restricted to $|t-s|\lesssim \gamma$. Having the usual dispersion bound for this truncated kernel, we get
\begin{equation}
  \label{eq:10bis}
  \|u^{\pm,0}_{h,\gamma}\|_{L^{4}_{t}L^{\infty}} \lesssim \frac{h^{1/2}}{h^{2}} \|f\|_{L^{4/3}_{t}L^{1}}\,.
\end{equation}
Hence on such a short time scale we recover the classical $L^{4/3}L^{1}\rightarrow L^{4}L^{\infty}$ bound on the 2D wave equation in $\R^{2}$. Going back to an homogeneous estimate and using conservation of energy, we may pile up $1/\sqrt\gamma$ estimates to obtain an homogeneous estimate on a longer time interval, at the expense of a large constant $1/(\sqrt\gamma)^{1/4}$, and then convert that estimate to an inhomogeneous estimate again, and it will hold for both $u^{\pm}_{h,\gamma}$ and $w^{\pm}_{h,\gamma}$ (using \eqref{eq:10bis} for $w^{\pm}_{h,\gamma}=u^{\pm}_{h,\gamma}-u^{\pm,0}_{h,\gamma }$): 
\begin{equation}
  \label{eq:10}
\|u^{\pm}_{h,\gamma}\|_{L^{4}_{t}L^{\infty}}+  \|w^{\pm}_{h,\gamma}\|_{L^{4}_{t}L^{\infty}} \lesssim \frac{h^{1/2}}{h^{2}} \gamma^{-\frac 1 4}\|f\|_{L^{4/3}_{t}L^{1}}\,.
\end{equation}
\begin{rmq}
Recall we have (from \eqref{eq:bornesup}), $\|u^{\pm}_{h,\gamma}\|_{L^{\infty}_{t,x}}\lesssim h^{-2}\sqrt\gamma  \|f\|_{L^{1}_{t,x}}$. Interpolating this with \eqref{eq:10} (for $u^{\pm}_{h,\gamma}$) yields an $L^{6}L^{\infty}$ bound, which is nothing but the bound from \cite{blsmso08}.
\end{rmq}

We now proceed with larger times, and start with $\gamma>h^{1/3}$, where one may easily
check that the condition $N<\lambda^{1/3}$ is always satisfied. Moreover, in this regime, there are no overlaps between waves. We  first consider the tangential part, $G^{+,\sharp}_{h,\gamma}=\sum_{N} W_{N,a}$. From Propositions \ref{propdispNpetitpres} and \ref{propdispNpetitloin}, set $N_{max}\sim 1/\sqrt \gamma \lesssim \lambda^{1/3}$ (the maximum number of reflections we may observe on our given time interval of size comparable to one), and write 
\begin{align*}
  \int_{\sqrt\gamma}^{1} \sup_{X,Y}|\sum_{N} W_{N,a}(t/\sqrt a,X,Y)|^{p}\,dt  \lesssim &   \sqrt \gamma\sum_{1\leq N\lesssim \frac 1 {\sqrt \gamma}} \int_{4N}^{4N+4} \sup_{X,Y}| W_{N,a}(T,X,Y)|^{p}\,dT \\
    \lesssim  & \frac{h^{\frac p 3} \sqrt \gamma}{h^{2p}} \sum_{1\leq N\lesssim \frac {1} {\sqrt \gamma}}\left( \int_{4N+1/N}^{4N+4} \frac{1}{N^{p/4}|(T-4N)|^{p/4}} \,dT \right.\\
  & {}+ \left.\int_{4N}^{4N+1/N}   \frac{1}{((N/\lambda^{1/3})^{\frac 14}+|N(T-4N)|^{1/6})^{p}} \,dT\right)\,.
\end{align*}
Changing variables, we compute the right-hand side, for $3\leq p <6$: 
\begin{align*}
\int_{\sqrt \gamma}^{1}(\cdots)\,dt  \lesssim  & \frac{h^{\frac p 3} \sqrt \gamma}{h^{2p}} \sum_{N\lesssim \frac {1} {\sqrt \gamma}}\left( \int_{1/N}^{4} \frac{1}{N^{p/4}| \theta|^{p/4}} \,d\theta  
  + \int_{0}^{1/N}   \frac{1}{(N^{\frac 14}\lambda^{-\frac 1 {12}}+|N\theta|^{1/6})^{p}} \,d\theta \right)\\
    \lesssim  & \frac{h^{\frac p 3} \sqrt \gamma}{h^{2p}} \sum_{N\lesssim \frac {1} {\sqrt \gamma}}\left( \frac 1 {N}  \int_{1}^{4N} \frac{1}{| z|^{p/4}} \,dz 
+ \int_{0}^{{\sqrt N} \over {\sqrt\lambda}}   \frac{\lambda^{\frac p {12}}}{N^{\frac p4}} \,d\theta   + \frac 1 N \int_{N^{\frac 3 2}/\lambda^{\frac 1 2}}^{1}   \frac{1}{|z|^{p/6}} \,dz\right)\\
  \lesssim & \frac{h^{\frac p 3} \sqrt \gamma}{h^{2p}} \sum_{N\lesssim \frac {1} {\sqrt \gamma}} \left(
 \sup(N^{-\frac p 4},\frac{\log N} N)+  \frac{\lambda^{\frac {p-6} {12}}}{N^{\frac {p-2}4}}  + \frac 1 N\right)\\
 \lesssim & \frac{h^{\frac p 3} \sqrt \gamma}{h^{2p}} \left(  \frac{1}{(\lambda\gamma^{3/2})^{\frac{6-p} {12}}}+\sup({{\sqrt \gamma}^{\frac p 4-1}},|\log \gamma |^{2})\right)
\end{align*}
where $1\lesssim \lambda \gamma^{3/2}=\gamma^{3}/h$ in our regime.
We now do the same computation but for the transverse part, $G^{+,\flat}_{h,\gamma}=\sum_{N} W_{N,\gamma}$, using Propositions \ref{transverseTpp9} and \ref{transverseTpg9}, for $T>1$:
\begin{align*}
  \int_{\sqrt\gamma}^{1} \sup_{X,Y}|\sum_{N} W_{N,\gamma}(t/\sqrt a,X,Y)|^{p}\,dt  \lesssim &   \sqrt \gamma\sum_{N\lesssim \frac 1 {\sqrt \gamma}} \int_{4N}^{4N+4} \sup_{X,Y}| W_{N,\gamma}(T,X,Y)|^{p}\,dT \\
    \lesssim  & \frac{h^{\frac p 3} \sqrt \gamma}{h^{2p}} \left( \sum_{N\lesssim 9} \int_{4N}^{4N+4} \frac{1}{N^{p/3}} \,dT + \sum_{9\lesssim N\lesssim \frac {1} {\sqrt \gamma}} \int_{4N}^{4N+4} \frac{1}{N^{p/2}} \,dT \right)
\end{align*}
and, for $2< p$ the sum is finite.  Therefore, summing 
both tangential and transverse estimates, for all $a\lesssim \gamma$ and choosing $p=4$ yields
\begin{equation}
  \label{eq:122}
  \left(\int_{\sqrt\gamma}^{1} \sup_{X,Y}|G^{+}_{h,\gamma}(\cdots)|^{4}\,dt\right)^{\frac 1 4}  \lesssim \frac{h^{\frac 1 3}}{h^{2}}  \gamma^{1/8}|\log \gamma|^{1/2}\,.
\end{equation}
Using \eqref{eq:122} we get an estimate for $w^{\pm}_{h,\gamma}$: first from $L^{1}_{s,a,b}$ to $L^{4}_{t}L^{\infty}_{x,y}$ and then by duality and interpolation,
\begin{equation}
  \label{eq:123}
  \|w^{\pm}_{h,\gamma}\|_{L^{8}_{t}L^{\infty}} \lesssim \frac{h^{1/4}}{h^{2}} h^{1/12}\gamma^{1/8}|\log \gamma|^{1/2}\|f\|_{L^{8/7}_{t}L^{1}}\,.
\end{equation}
We notice that $h^{1/12}\gamma^{1/8}\lesssim \gamma^{3/8}$. We may now interpolate between \eqref{eq:123} and \eqref{eq:10}: pick an interpolating exponent $\theta=3/5$, we have $1/5=\theta/4+(1-\theta)/8$ and $3 (1-\theta)/8- \theta/4=0$, we get
\begin{equation}
  \label{eq:12bisbisder}
  \|w^{\pm}_{h,\gamma}\|_{L^{5}_{t}L^{\infty}} \leq \frac {h^{2/5}} {h^{2}} (\log \gamma)^{1/5} \|f_{\gamma}\|_{L^{5/4}_{t}L^{1}}\,.
\end{equation}
The same bound holds for $u^{\pm,0}_{h,\gamma}$, with a much better $\gamma^{1/5}$ replacing the $\log$ factor. If in the interpolation step, we pick $\theta<3/5$, there is an additional positive power of $\gamma$ left, making the $\log$ irrelevant, and we can even sum over $\gamma \gtrsim h^{1/3}$: denote this sum by $G^{\pm}_{h,\gtrsim h^{1/3}}$, we proved
\begin{prop}
  \label{prop:p5}
  The half-wave operator $G^{\pm}_{h,\gtrsim h^{1/3}}$ is such that, for any $q>5$,
  \begin{equation}
    \| G^{\pm}_{h,\gtrsim h^{1/3}} u_{0}\|_{L^{q}_{t}L^{\infty}_{x,y}} \lesssim
    h^{1/q-1} \|u_{0}\|_{2}\,.
  \end{equation}
\end{prop}
\begin{rmq}
Considering that the counterexample in \cite{ILP4} precludes $q<5$, Proposition \ref{prop:p5} is optimal up to the endpoint $q=5$. In fact, one may refine our argument to get an optimal $L^{q}L^{\infty}$ estimate for $G^{\pm}_{h,\gamma}$, depending on $h^{1/3}\lesssim \gamma\lesssim 1$, and connecting $q=5$ and $q=4$.
\end{rmq}
We now proceed with the lower regime $h^{2/3}<\gamma<h^{1/3}$. Here we essentially get (at most) three regimes: $|t|<\sqrt \gamma \lambda^{1/3}$, $\sqrt \gamma \lambda^{1/3}<|t|<\sqrt \gamma \lambda^{2}$ and $|t|>\sqrt \gamma \lambda^{2}$.  Denote by $N^{*}=\lambda^{1/3}$, and start with $|t|<t^{*}=\sqrt \gamma N^{*}$: we reproduce the previous argument, except
now we evaluate the $L^{p}_{t}$ norm, for $p>5/2$ on a
time interval $(0,\sqrt \gamma N^{*})$.
\begin{align*}
  \int_{\sqrt\gamma}^{t^{*}} \sup_{X,Y}|\sum_{N} W_{N,a}(t/\sqrt a,X,Y)|^{p}\,dt
  \lesssim & \frac{h^{\frac p 3} \sqrt \gamma}{h^{2p}} \sum_{N\lesssim  N^{*}} \left(
 \sup(N^{-\frac p 4},\frac{\log N} N)+  \frac{\lambda^{\frac {p-6} {12}}}{N^{\frac {p-2}4}}  + \frac 1 N\right)\\
 \lesssim & \frac{h^{\frac p 3} \sqrt \gamma}{h^{2p}} \left(  {1}+\sup({({N^{*})}^{1-\frac p 4}},|\log N^{*} |^{2})\right)
\end{align*}
while for the transverse part, replacing $1/\sqrt{\gamma}$ by $N^{*}$ does not change the estimate. Hence, summing tangential and transverse estimates, and choosing $p=4$ again,
\begin{equation}
  \label{eq:122bis}
  \left(\int_{\sqrt\gamma}^{t^{*}} \sup_{X,Y}|G^{+}_{h,\gamma}(\cdots)|^{4}\,dt\right)^{\frac 1 4}  \lesssim \frac{h^{\frac 1 3}}{h^{2}}  \gamma^{1/8}|\log N^{*}|^{1/2}\,.
\end{equation}
Let us go back to \eqref{eq:10} for $w^{\pm}$: piling up $N^{*}$ intervals of length $\sqrt \gamma$, we may replace the corresponding inhomogeneous estimate, with $|J|=\sqrt \gamma N^{*}$,
\begin{equation}
  \label{eq:10bisbis}
\|w^{\pm}_{h,\gamma}\|_{L^{4}_{J}L^{\infty}} \lesssim \frac{h^{1/2}}{h^{2}} \sqrt{N^{*}}\|f\|_{L^{4/3}_{J}L^{1}}\,.
\end{equation}
 By the same interpolation we did before, between the $L^{8}_{J}L^{\infty}$ estimate resulting from \eqref{eq:122bis} and \eqref{eq:10bisbis}, and duality, we get, for any $q>5$, for $w^{\pm}_{h,\gamma}$ and then $u^{\pm}_{h,\gamma}$,
  \begin{equation}\label{eq:derdur}
    \| u^{\pm}_{h,\gamma}\|_{L^{q}_{(0,t^{*})}L^{\infty}_{x,y}} \lesssim \frac{h^{2/q}}h  \gamma^{\varepsilon(q)}\|f \|_{L^{q'}_{(0,t^{*})}L^{1}}\,,
  \end{equation}
  where $\varepsilon(q)>0$ (so we can later sum over $\gamma$). Note that $\gamma<h^{1/3}$ and therefore $\lambda^{1/3}=N^{*}< h^{-1/6}$, which explains the numerology.

We now proceed with $N^{*}<|t|/\sqrt \gamma <N_{1}=\inf(1/\sqrt \gamma, \lambda)$, and set $t_{1}=\sqrt \gamma N_{1}$. We compute again the $L^{p}_{t}$ norm, for $2< p<4$. After
rescaling in time, and with $\mu=1-p/4>0$, we have for tangential waves
\begin{align*}
  \nonumber  \int_{\sqrt\gamma N^{*}}^{t_{1}} \sup_{X,Y}|\sum_{N} W_{N,a}(\cdots)|^{p}\,dt    & \lesssim 
{h^{\frac p 3 -2p}\sqrt \gamma} \sum_{N^{*}<N<N_{1}} \int_{4N}^{4N+4} \frac {1} {((\frac N{N^{*}})^{1/2} +{N}^{1/4}|T-4N|^{1/4})^{p}}\, dT\\
& \lesssim 
{h^{\frac p 3-2p }\sqrt \gamma} \sum_{N^{*}<N<N_{1}} \int_{0}^{4} \frac {1} {((\frac N{N^{*}})^{2} +{N}\theta)^{p/4}}\,d\theta\\
& \lesssim 
{\frac {h^{\frac p 3}\sqrt \gamma }{h^{2p} } } \sum_{N^{*}<N<N_{1}}   \frac {1}{N} \left(\Bigl((\frac N{N^{*}})^{2}+4N)^{\mu}-(\frac N{N^{*}})^{2\mu}\Bigr)\right) \\
& \lesssim 
{\frac {h^{\frac p 3}\sqrt \gamma }{h^{2p} } } \sum_{N^{*}<N<N_{1}}   \frac {4} {\Bigl((\frac N{N^{*}})^{2}+4N)^{1-\mu}+(\frac N{N^{*}})^{2(1-\mu)}\Bigr)}\,.
\end{align*}
We compute the last sum: if $N_{1} \lesssim (N^{*})^{2}$, then $N_{1}=1/\sqrt \gamma$ and
$$
\sum_{N^{*}<N<\frac 1 {\sqrt\gamma}}\left( \cdots\right)\lesssim
\sum_{N^{*}<N<\frac 1 {\sqrt\gamma}} \frac {1} { N^{1-\mu}}\lesssim \left(\frac {1} { \sqrt \gamma} \right)^{\mu}\,,
$$
and if $N_{1} \gtrsim (N^{*})^{2}$, then $(\frac N{N^{*}})^{2}+4N)^{1-\mu}+(\frac N{N^{*}})^{2(1-\mu)}\gtrsim (\frac  N {N^{*}} ) ^{2(1-\mu)}$ and with $2(1-\mu)=p/2>1$, 
\begin{equation}
  \label{eq:9}
  \sum_{N^{*}<N<N_{1} }\left( \cdots\right)\lesssim
( N^{*})^{2\mu}+\sum_{(N^{*})^{2}<N<N_{1} } \frac {1} { (\frac N {N^{*}})^{2(1-\mu)}}\lesssim ( N^{*})^{2-p/2}=\lambda^{2/3-p/6}\,.
\end{equation}
For $t_{1}\leq |t|\lesssim 1$, which corresponds to $N>N_{1}= \lambda$, $\gamma\lesssim h^{1/2}$, we may use the improved bound \eqref{eq:1ff>},
 \begin{align}
\nonumber    \int^{1}_{t_{1}} \sup_{X,Y}|\sum_{N} W_{N,a}(\cdots)|^{p}\,dt    & \lesssim 
   {h^{\frac p 3 -2p}\sqrt \gamma} \sum_{N_{1}<N<1/\sqrt \gamma} \int_{4N}^{4N+4} \frac {\sqrt \lambda} {(\sqrt N (\frac N{N^{*}})^{1/2})^{p}}\, dT\\
       & \lesssim 
 {h^{\frac p 3 -2p}\sqrt \gamma}   {\lambda^{3/2-5p/6}}\,.\label{N>lambda}
 \end{align}
 For the transverse part, we get a straightforward estimate, directly on $\sqrt \gamma N^{*}<|t|\lesssim 1$,
\begin{align*}
 \int_{\sqrt\gamma N^{*}}^{1} \sup_{X,Y}|\sum_{N} W_{N,\gamma}(\cdots)|^{p}\,dt  &  \lesssim {h^{\frac p 3 -2p} \gamma^{p/4}} \int_{\sqrt\gamma N^{*}}^{1} \frac{dt}{t^{p/2}}\\
                                                                                   &  \lesssim {h^{\frac p 3 -2p} \gamma^{p/4}} (\sqrt \gamma N^{*})^{1-p/2} 
                   \lesssim {h^{\frac p 3 -2p} \sqrt \gamma} / (\lambda^{1/3})^{p/2-1}
\end{align*}
which is better than \eqref{N>lambda}.  
Therefore,
\begin{equation}
  \label{eq:122bister}
  \left(\int_{\sqrt\gamma N^{*}}^{1} \sup_{X,Y}|G^{+}_{h,\gamma}(\cdots)|^{p}\,dt\right)^{\frac 1 p}  \lesssim \frac{h^{\frac 1 p}}{h^{2}} {h^{1/3-1/p}}\sqrt \gamma^{1/p} \left(\frac{\inf(1/\sqrt\gamma, \lambda^{2/3}) } {\lambda^{1/3}}\right)^{(2-p/2)/p}\,,
\end{equation}
and for $p=18/7$ which is of particular interest to us, when $\gamma\geq h^{4/9}$ (e.g. $1/\sqrt \gamma\lesssim \lambda^{2/3}$)
\begin{equation}
  \label{eq:33bis}
  \left(\int_{\sqrt\gamma N^{*}}^{1} \sup_{X,Y}|G^{+}_{h,\gamma}(\cdots)|^{\frac{18}7}\,dt\right)^{\frac 7 {18}}  \lesssim \frac{h^{\frac 7{18}}}{h^{2}} \left(\frac{h^{\frac 4 9}}\gamma \right)^{\frac 1 {12}}\,.
\end{equation}
while for $h^{4/7}<\gamma<h^{4/9}$, 
\begin{equation}
  \label{eq:33}
  \left(\int_{\sqrt\gamma N^{*}}^{1} \sup_{X,Y}|G^{+}_{h,\gamma}(\cdots)|^{\frac{18}7}\,dt\right)^{\frac 7 {18}}  \lesssim \frac{h^{\frac 7{18}}}{h^{2}} \left(\frac{\gamma}{h^{\frac 4 9}}\right)^{\frac 1 3}\,.
\end{equation}
The worst case scenario is $\gamma\sim h^{4/9}$, and summing for $\gamma$ above or below, using duality we get an inhomogeneous $L^{36/7}_t L^{\infty}$ estimate for $w^{\pm}_{h,\gamma}$, but where the kernel is restricted to $|t-s|\gtrsim t^{*}$: summing with \eqref{eq:derdur}, we recover the same estimate for $u^{\pm}_{h,\gamma}$. The homogeneous estimate then writes
\begin{prop}
  \label{prop:p55}
  The half-wave operator $\sum_{h^{4/7}<\gamma<h^{1/3}}G^{\pm}_{h,\gamma}$ is such that, for any $q\geq 36/7$,
  \begin{equation}
   \| \sum_{h^{4/7}<\gamma<h^{1/3}}G^{\pm}_{h,\gamma} u_{0}\|_{L^{q}_{t}L^{\infty}_{x,y}} \lesssim
   h^{1/q-1} \|u_{0}\|_{2}\,.
  \end{equation}
  \end{prop}
\begin{rmq}
This is $q=5+1/7$ from Theorem \ref{thm+1}; the numerology relates to the bound \eqref{eq:1ff}, which saturates for all $T$'s where the corresponding wave $W_{N,a}$ is significant around $N\sim \lambda^{2/3}$ 
\end{rmq}
\subsection{Overlapping waves: $h^{2/3-\varepsilon}<\gamma \leq  h^{4/7}$}
  \label{sec:overl-waves:-h23}
For various reasons we explained earlier, this worst case scenario (in terms of overlap) occurs for $N_{max}\geq \lambda^{2}$, with $N_{max}\sim 1/\sqrt \gamma$ and $\lambda=\gamma^{3/2}/h$ (recall that this translates into $\gamma < h^{4/7}$.)

For $ |t|<\lambda^{2}\sqrt \gamma$, we may do the same argument as before (no overlap): in other words, \eqref{eq:122bis} holds and in our regime of $\gamma$, the infimum in the bound is $\lambda^{2/3}$. We actually compute the bound for $\gamma<h^{1/2}$, for $p=5/2$:
\begin{equation}
  \label{eq:122bisbis}
  \left(\int_{\sqrt\gamma N^{*}}^{\inf(1,\sqrt \gamma \lambda^{2})} \sup_{X,Y}|G^{+}_{h,\gamma}(\cdots)|^{\frac 5 2}\,dt\right)^{\frac 2 5}  \lesssim \frac{h^{\frac 2 5}}{h^{2}} \left( \frac{\gamma} {h^{20/42}}\right)^{\frac 7 {20}}\,.
\end{equation}
From Propositions \ref{cordispNgrand} and \ref{WgrandTN2}, we get a uniform bound for $\sqrt \gamma \lambda^{2}< t < 1$, and therefore
\begin{equation}
  \label{eq:14}
  \| G^{+,\sharp}_{h,\gamma}(t,\cdot)+G^{+,\flat}_{h,\gamma}(t,\cdot)\|_{L^{5/2}_{t>\sqrt \gamma \lambda^{2}}}\lesssim \frac{h^{1/3}}{h^{2}} (\lambda^{-4/3}+\lambda^{-3/2})\lesssim\frac {h^{\frac 2 5}}{h^{2}} \left(\frac{h^{\frac 2 3-\frac {1}{30}}}{\gamma}\right)^{2}\,.
\end{equation}
Here we pushed this direct computation as far as it goes: we proved
\begin{prop}
  \label{prop:p555}
  The half-wave operator $\sum_{h^{\frac 2 3-\frac 1{30}}<\gamma<h^{\frac 1 2}}G^{\pm}_{h,\gamma}$ is such that, for any $q> 5$,
  \begin{equation}
   \| \sum_{h^{\frac 2 3-\frac 1 {30} }<\gamma<h^{\frac 1 2 }}G^{\pm}_{h,\gamma} u_{0}\|_{L^{q}_{t}L^{\infty}_{x,y}} \lesssim
   h^{1/q-1} \|u_{0}\|_{2}\,.
  \end{equation}
  \end{prop}
Collecting all the Propositions in this section, we obtain Theorem \ref{thm+1}, but only for the operator $\sum_{h^{\frac 2 3-\varepsilon}<\gamma} G^{\pm}_{h,\gamma}$ with $\varepsilon=1/30$ at the moment. For lower values of $\gamma$, we will now turn to gallery modes.

\subsection{Strichartz estimates for gallery modes}
We refine a result of \cite{doi}: let
\begin{equation}
  \label{eq:1}
u_{k,\pm}(t,x,y)=  \frac 1h\int_{\mathbb{R}}e^{\pm it\sqrt{\lambda_k(\eta/h)}}
e^{iy\eta/h} \psi(\eta) e_k(x,\eta/h)\,d\eta\,.
\end{equation}
These solutions to the wave equation and are the so-called gallery modes. Using Lemma \ref{lemorthog}, $(u_{k,\pm})_{k}$is an  orthogonal family  in $L^2(\Omega_2)$.
We also set $u_{k,0}(x,y)=u_{k,\pm}(0,x,y)$.
\begin{thm}\label{thGM}
For any data of the form $u_{k,0}$, $k\geq 1$, Strichartz estimates hold, uniformly in the parameter $k$: there exists a universal constant $C$ such that for all $(q,r)$ satisfying $q\geq 4$, $\frac 1 q=\frac 1 2 (\frac 1 2-\frac 1 r)$, and with $\beta=2(1/2-1/r)-1/q$, we have
\begin{equation}
  \label{eq:2}
  \|u_{k,\pm}\|_{L^{q}_{t}(L^{r}(\Omega_2))} \leq C h^{-\beta}\|u_{k,0}\|_{L^2(\Omega_2)}\,,
\end{equation}
\end{thm}
\begin{proof}
Since $u_{k,-}(t,x,y)=u_{k,+}(-t,x,y)$, we prove for \eqref{eq:2} for $u_k:=u_{k,+}$. By definition, $\sqrt{\lambda_{k}(\eta/h)}=\frac{\eta}{h}\sqrt{1+\omega_{k}(h/\eta)^{2/3}}$, then $u_{k}(t,x,y)=\frac 1 h \int e^{i\frac t h \eta\sqrt{1+\omega_{k}(h/\eta)^{2/3}}} e_{k}(x,\eta/h) \psi(\eta) e^{i\frac {y\eta} h}\,d\eta$. At $t=0$, we get (recall $e_{k}$ is $L^{2}$-normalized), after Fourier transform in $y$, $\|u_k(0,\cdot)\|^{2}_{L^{2}(\Omega_2)}\sim h \|\psi\|^{2}_{L^{2}_{\eta}}$. We aim at controlling the Fourier multiplier (w.r.t. $y$, at fixed $x$) $e_{k}(x,\eta/h)$: if we can obtain good $L^{2}_{y}$ bounds on this multiplier, dispersion will reduce to an equation on 
\begin{equation}
  \label{eq:bis3}
  w_{k}(t,y):=\frac 1h \int  e^{i\frac t h \sqrt{\eta^2+\omega_{k}\eta^{4/3}h^{2/3}}} \psi(\eta) e^{i\frac {y\eta} h}\,d\eta\,.
\end{equation}
A simple computation yields
\begin{equation}\label{eq:simukwy0}
\|w_k(0,\cdot)\|^2_{L^2_y}\sim h\|\psi\|^2_{L^2_{\eta}}\sim  \|u_k(0,\cdot)\|^{2}_{L^{2}(\Omega_2)}.
\end{equation}
Since $e_k(x,\theta)=\frac{\theta^{1/3}}{\sqrt{L'(\omega_{k})}}Ai(\theta^{2/3}x-\omega_k)$, let $\theta=\eta/h$, we are left with a convolution with 
\begin{equation}
  \label{eq:bis6}
  \Gamma_{x,\omega_k}(y)=\frac{1}{\sqrt{L'(\omega_{k})}} \frac 1h \int e^{i\frac{y \eta}{h}} (\eta/h)^{1/3} Ai((\eta/h)^{2/3}x-\omega_{k}) \psi(\eta)\,d\eta\,,
\end{equation}
which easily maps $L^{2}_{y}$ to $L^{2}(\Omega_2)$. To map $L^{\infty}_{y}$ to $L^{\infty}(\Omega_{2})$, we compute $\sup_{x}\|\Gamma_{x,\omega_k}(y)\|_{L^{1}_{y}}$. We set $y=h Y$, $x=h^{2/3}z$ and $G_{z,\omega_k}(Y)=h \Gamma_{h^{2/3} z,\omega_k}(h Y)$, then, replacing $Ai$ by its integral formula \eqref{eq:bis47},
\begin{equation}
  \label{eq:bis7bis}
 2\pi G_{z,\omega_k}(Y)=\frac{h^{-1/3}}{\sqrt{L'(\omega_{k})}} \int e^{i(Y\eta+\frac {s^{3}}3 + s(\eta^{2/3}z-\omega_{k}))} \eta^{1/3}\psi(\eta) \,d\eta ds\,.
\end{equation}
Set $s=\omega_{k}^{1/2} \sigma$, $Y=\omega_{k}^{3/2} \Upsilon$, $ z=\omega_{k} Z$, and define $H_{Z,\omega_k}(\Upsilon):= \omega_{k}^{3/2}  G_{z,\omega_k}(\omega_k^{3/2}\Upsilon)$. We have
\begin{equation}\label{eq:L1}
\sup_x \|\Gamma_{x,\omega_k}(y)\|_{L^1_y}  =\sup_{z, x=h^{2/3}z}\int |G_{z,\omega_k}(Y)|dY =\sup_{Z,x=h^{2/3}\omega_k Z}\|H_{Z,\omega_k}(.)\|_{L^1_{\Upsilon}}\,.
\end{equation}
To estimate $\sup_{Z,x=h^{2/3}\omega_k Z}\|H_{Z,\omega_k}(.)\|_{L^1_{\Upsilon}}$ we compute $H_{Z,\omega_{k}}$. Thinking of $\omega_{k}^{3/2}$ as our large parameter, the phase function of $H_{Z,\omega_k}(\Upsilon)$ becomes
  $\Phi=\Upsilon \eta +\frac{\sigma^{3}}3+\sigma(\eta^{2/3}Z-1)$, and
\begin{equation}
  \label{eq:bis9}
  \partial_{\eta}\Phi=\Upsilon+\frac 2 3 \sigma Z \eta^{-1/3},\,\,\,\,\partial_{\sigma}\Phi=\sigma^{2}+\eta^{2/3}Z-1\,.
\end{equation}
If $Z$ large, then the phase is non-stationary in $\sigma$ and we can perform repeated integrations by parts in this variable to obtain a small contribution. In order $\Phi$ to be stationary, we need
\begin{equation}
  \label{eq:bis10}
  \sigma^{2}=1-\eta^{2/3}Z\,\text{ which gives }\Upsilon=\pm \frac 2 3 Z \eta^{-1/3}\sqrt{1-\eta^{2/3}Z}
\end{equation}
so both $Z$, $\sigma$ and $\Upsilon$ are bounded on the stationary set (indeed, from \eqref{eq:bis10} we must have $\sigma^2\lesssim 1$, $|Z|\lesssim \eta^{-2/3}$ and $\eta\in [\frac 1 2,\frac 3 2]$). The matrix of second order derivatives is
\[
\text{Hess } \Phi =
\left(
\begin{array}{cc}
 -\frac 29\sigma Z\eta^{-4/3} & \frac 23\eta^{-1/3}Z   \\
\frac 23\eta^{-1/3}Z    &  2\sigma   \\
 \end{array}
\right)
\]
and, using \eqref{eq:bis10}, at the stationary points $\Big|\det \text{Hess }(\Phi)\Big|=\frac 49 \eta^{-4/3}Z$.
Stationary phase applies provided that $\omega_{k}^{3/2}\sqrt{Z}\gg 1$. In this case,
\begin{equation}
  \label{eq:bis12}
  |H_{Z,\omega_k}(\Upsilon)|\sim  \frac{h^{-1/3}}{\sqrt{L'(\omega_{k})}}\omega_{k}^{2} \times \frac{\omega_k^{-3/2}}{\sqrt{Z}}.
\end{equation}
Using \eqref{eq:bis10}, it turns out that the main contribution in the integral defining $H_{Z,\omega_k}(\Upsilon)$ comes from values $|\Upsilon| \lesssim Z$, for larger values of $|\Upsilon|$ the contribution of $H_{Z,\omega_k}(\Upsilon)$ being $h^{-1/3}\times O(k^{-M})$ for any $M\geq 1$ by non-stationary phase. Therefore we can estimate
 \begin{equation}
  \label{eq:bis13}
 \sup_{\omega_k^{-3}\ll Z\lesssim 1} \|H_{Z,\omega_k}(\Upsilon)\|_{L^{1}_{\Upsilon}}\sim   \frac{h^{-1/3}}{\sqrt{L'(\omega_{k})}}\omega_{k}^{2}\int_{|\Upsilon|<CZ}\frac 1 {\omega_{k}^{3/2}\sqrt{Z}}\, d\Upsilon\sim \frac {\omega_{k}^{1/4}}{h^{1/3}}\sqrt{Z}.
\end{equation}
When $Z$ (hence $x$) is very small, we can no longer apply stationary phase. However, for values $|Z|\lesssim \omega_k^{-3}$ notice that non-stationary phase applies in $\sigma$, as, for $Z$ sufficiently small,
\[
|\partial^2_{\sigma}\Phi|\Big|_{\partial_{\sigma}\Phi=0}=
|2\sigma|\Big|_{\sigma=\pm\sqrt{1-\eta^{2/3}Z} } \geq \frac 14\,.
\]
The integral in $\sigma$, which is nothing but the Airy function, yields two main contributions, corresponding to $A_{\pm}$ defined in \eqref{eq:Apm}. In other words, when $x=h^{2/3}\omega_kZ$ is very small, we directly split the Airy function in \eqref{eq:bis6} as follows 
\[
Ai((\eta/h)^{2/3}x-\omega_{k}) =\sum_{\pm} A_{\pm}(\omega_k-(\eta/h)^{2/3}x)=\sum_{\pm}e^{\mp \frac 23 i (\omega_k-(\eta/h)^{2/3}x)^{3/2}}\Psi(\omega_k-(\eta/h)^{2/3}x),
\]
where, according to \eqref{eq:ApmAE}, $|\Psi(W)|\lesssim 1/(1+W^{1/4})$. We obtain 
\[
\Gamma_{x,\omega_k}(y)=\sum_{\pm}\frac{h^{-4/3}}{\sqrt{L'(\omega_k)}}\int e^{\frac ih (y\eta\mp \frac 23 (h^{2/3}\omega_k-x\eta^{2/3})^{3/2})}\Psi(\omega_k-(\eta/h)^{2/3}x) \eta^{1/3}\psi(\eta)d\eta,
\]
and we notice that $h^{2/3}\omega_k-x\eta^{2/3}\geq \frac 12 h^{2/3}\omega_k$ for small values of $x$. The phase is stationary when $y\pm \frac 23x\eta^{-1/3}(h^{2/3}\omega_k-x\eta^{2/3})^{1/2}=0$; for values $|y|> Cx (h^{2/3}\omega_k)^{1/2}$ for $C>0$ independent of $\omega_k$, the phase is non-stationary in $\eta$ and we obtain an $O(h^{\infty})$ contribution. We can now estimate
\begin{align*}
\int |\Gamma_{x,\omega_k}(y)| dy &\lesssim \int_{|y|\leq Cx h^{1/3}\sqrt{\omega_k}}\frac{h^{-4/3}}{\sqrt{L'(\omega_k)}}\times \frac{1}{\omega_k^{1/4}}dy=\frac{h^{-4/3}}{\sqrt{\omega_k}}\times Cxh^{1/3}\sqrt{\omega_k}\\
&\lesssim C h^{-1}\times \frac{h^{2/3}}{\omega_k^2} =C\frac{h^{-1/3}}{\omega_k^2},
\end{align*}
using that $x=h^{2/3}\omega_k Z$ and $Z\lesssim \omega_k^{-3}$. Since this last bound is smaller than the one in \eqref{eq:bis13} it follows that
\begin{equation}\label{eq:Gamest}
\sup_{x}\|\Gamma_{x,\omega_k}(.)\|_{L^1_{y}}\lesssim \frac{\omega_k^{1/4}}{h^{1/3}},
\end{equation}
and this bound saturates for $x\sim h^{2/3}\omega_k$. It remains to estimate $\|w_k\|_{L^{\infty}_y}$: going back to \eqref{eq:bis3}, stationary phase in $\eta$ does apply when $t\gg h^{1/3}$ since
\begin{equation}
  \label{eq:bis15}
|\partial^2_{\eta}(y\eta+t\sqrt{\eta^2+\omega_k \eta^{4/3}h^{2/3}})| \sim t \omega_{k} h^{2/3}\eta^{4/3}(1+O(h^{2/3})).
\end{equation}
Indeed, for $t\gg h^{1/3}$ we obtain a bound for $w_k(t,y)$ of the form
\begin{equation}
  \label{eq:bis16}
|w_k(t,y)|\sim  \frac 1 h \frac{\sqrt{h}}{\sqrt{t\omega_{k} h^{2/3}}}\sim \frac1h \Big(\frac ht\Big)^{1/2} \frac{1}{\omega_{k}^{1/2} h^{1/3}}.
\end{equation}
When $t\lesssim Mh^{1/3}$ for some constant $M$, the phase doesn't oscillate and we simply bound $w_k$ by $\frac 1h$. It follows that for every $t>h$, $\|w_k(t,\cdot)\|_{L^{\infty}}$ is bounded by 
\begin{equation}\label{eq:wkest}
\|w_k(t,\cdot)\|_{L^{\infty}}\lesssim  \frac1h \Big(\frac ht\Big)^{1/2} \frac{1}{\omega_{k}^{1/2} h^{1/3}}
\end{equation}
which then provides
\begin{equation}
  \label{eq:bis29}
  \|w_k\|_{L^{4}(L^{\infty}_{y})}\lesssim  \frac 1 {h^{1/4} h^{1/6} \omega_{k}^{1/4}} \|w_{k}(0,\cdot )\|_{L^{2}_y}.
\end{equation}
Returning to $u_k$ and using \eqref{eq:Gamest}, \eqref{eq:wkest} and \eqref{eq:simukwy0}, we obtain
\begin{equation}
  \label{eq:bis29bis}
  \|u_k\|_{L^{4}(L^{\infty}_{x,y})}\lesssim \frac{\omega_{k}^{1/4}}{h^{1/3}} \|w_k\|_{L^{4}(L^{\infty}_{y})}\lesssim  \frac 1 {h^{3/4}}\|u_k(0,\cdot)\|_{L^{2}(\Omega_2)}\,.
\end{equation}
This proves that gallery modes satisfy the usual Strichartz estimates in $\R^{2}$.
\end{proof}
Theorem \ref{thGM} has an interesting consequence: consider the spectral decomposition of a given function $u_{0}$, (where the $h^{-1/2}$ accounts for the $L^{1}$ normalization of $u_{k}$)
\begin{equation}
  \label{eq:3}
  u_{0}(x,y)=\sum_{k\geq 1} c_{k} h^{-1/2 }u_{k}(0,x,y)\,,
\end{equation}
where $u_k(0,x,y)$ is given in \eqref{eq:1} for $t=0$ and where $c_k:=<u_0, h^{-1/2}u_k(0,\cdot)>_{L^2(\Omega_2)}$. We may then evaluate the Strichartz norm of the solutions to the half-wave operators with data $u_{0}$,
\begin{equation}
  \label{eq:4}
u_{\pm}(t,x,y)= \sum_{k} c_{k} u_{k,\pm}(t,x,y)
\end{equation}
provided we restrict ourselves to $\phi_{\gamma}(Q_{x,y}) u_{\pm}=u_{\gamma,\pm}$, where we define the (tangential) pseudo-differential operator $Q_{x,y}=x+D^{-2}_{y}D^{2}_{x}=-D_{y}^{-2}\Delta_{F}-Id$ and where $\gamma\geq h^{2/3}$ and $\phi_{\gamma}(\cdot)=\phi(\cdot/\gamma)$ with $\phi\in C^{\infty}_0([-1,1])$. To do this, we let $u_{\gamma,\pm}$ be defined as follows
\begin{equation}
  \label{eq:1gamma}
u_{\gamma,\pm}(t,x,y)=  \sum_{k\geq 1} \frac 1h\int_{\mathbb{R}}e^{\pm it\sqrt{\lambda_k(\eta/h)}}
e^{iy\eta/h} \psi(\eta) \phi_{\gamma}(h^{2/3}\omega_{k}/\eta^{2/3})e_k(x,\eta/h)\,d\eta\,.
\end{equation}
As in Section \ref{sectspectralloc}, the cut-off $\phi_{\gamma}$ reduces the sum over $k$ to $k\lesssim \gamma^{3/2}/h=\lambda_{\gamma}$.
\begin{prop}\label{lemmaGMStrichartz} 
There exists a universal constant $C$ such that 
\begin{equation}
  \label{eq:5}
  \| u_{\gamma,\pm}\|_{L^{4}_{t}(L^{\infty}(\Omega_2))} \leq C h^{-1+1/4} \lambda_{\gamma}^{1/2} \| u_{\gamma, \pm}(0,\cdot )\|_{L^2(\Omega_2)}\,.
\end{equation}
\end{prop}
The Proposition is a simple consequence of the Cauchy-Schwarz inequality: as remarked earlier on, the localization w.r.t. $Q_{x,y}$ restricts the sum over $k$ to $(hk)^{2/3}\lesssim \gamma$, and therefore
\begin{align*}
  \| u_{\gamma}\|_{L^{4}_{t}(L^{\infty}(\Omega_{2}))} & \lesssim \sum_{k\lesssim \lambda_{\gamma} }  h^{-1+1/4} |c_{k}| \lesssim  h^{-1+1/4} \lambda_{\gamma}^{1/2} \left(\sum_{k\lesssim \lambda_{\gamma} }  |c_{k}|^{2}\right)^{1/2}\\
 & \lesssim  h^{-1+1/4} \lambda_{\gamma}^{1/2} \| u_{\gamma} (0,\cdot)\|_{2}\,.
\end{align*}
\subsection{The gallery mode regime: $\gamma\lesssim h^{2/3-\varepsilon}$}
\label{sec:gallery-mode-regime}
Using \eqref{eq:5} and recalling \eqref{eq:bornesup} we have
\begin{equation}\label{upmL4Linf}
  \|u^{\pm}_{h,\gamma}\|_{L^{4}_{t}L^{\infty}_{x,y}} \lesssim \frac {h^{1/2}}{h^{2}} \lambda_{\gamma} \|f\|_{L^{4/3}_{s} L^{1}_{a,b}}     \,\,\text{ and }\,\,\|u^{\pm}_{h,\gamma}\|_{L^{\infty}_{t,x,y}} \lesssim \frac {\sqrt \gamma}{h^{2}} \|f\|_{L^{1}_{s,a,b}}\,.
\end{equation}
Chose $q>4$: by interpolation between our two bounds, we can sum over all $h^{2/3}\lesssim \gamma \lesssim h^{2/3-\varepsilon}$: consider $q>4$ and $\theta=1-4/q$, then the summability condition on $\varepsilon$ and $q$ reads, with $\mu>0$
$$
\left(\frac{\gamma^{3/2}} h \right)^{1-\theta} \sqrt \gamma^{\theta} \lesssim h^{\mu}\,, \quad\text{ e.g.}\quad q> 4(1+\frac{9\varepsilon}{2-3\varepsilon})\,.
$$
For $\varepsilon=1/30$, we get $q> 4+9/57$, so that
\begin{prop}
  \label{prop:p5555}
  The half-wave operator $\sum_{\gamma\leq h^{\frac 2 3-\frac 1{30}}}G^{\pm}_{h,\gamma}$ is such that, for any $q\geq 4+3/19$,
  \begin{equation}
   \| \sum_{\gamma\leq h^{\frac 2 3-\frac 1 {30}}}G^{\pm}_{h,\gamma} u_{0}\|_{L^{q}_{t}L^{\infty}_{x,y}} \lesssim
   h^{1/q-1} \|u_{0}\|_{2}\,.
  \end{equation}
  \end{prop}
This completes the proof of Theorem \ref{thm+1}.\qed


\begin{thebibliography}{10}

\bibitem{blsmso08}
Matthew~D. Blair, Hart~F. Smith, and Christopher~D. Sogge.
\newblock Strichartz estimates for the wave equation on manifolds with
  boundary.
\newblock {\em Ann. Inst. H. Poincar\'e Anal. Non Lin\'eaire},
  26(5):1817--1829, 2009.

\bibitem{Iv2020Sch}
Oana Ivanovici.
\newblock Dispersive estimates for the semi-classical {S}chrödinger equation
  inside strictly convex domains.
\newblock {\tt arXiv:math/2009.13810}.

\bibitem{doi}
Oana Ivanovici.
\newblock Counterexamples to {S}trichartz estimates for the wave equation in
  domains.
\newblock {\em Math. Ann.}, 347(3):627--673, 2010.

\bibitem{doi2}
Oana Ivanovici.
\newblock Counterexamples to the {S}trichartz inequalities for the wave
  equation in general domains with boundary.
\newblock {\em J. Eur. Math. Soc. (JEMS)}, 14(5):1357--1388, 2012.

\bibitem{Iv2020KG}
Oana Ivanovici.
\newblock Dispersive estimates for the wave and the {Klein-Gordon equations in
  large time inside the Friedlander domain}.
\newblock {\em Discrete Contin. Dyn. Syst. Ser. B}, 2021.
\newblock {\tt doi:10.3934/dcds.2021093}.

\bibitem{ILLP}
Oana Ivanovici, Richard Lascar, Gilles Lebeau, and Fabrice Planchon.
\newblock Dispersion for the wave equation inside strictly convex domains {II}:
  the general case, 2020.
\newblock {\tt arXiv:math/1605.08800}.

\bibitem{ildispext}
Oana Ivanovici and Gilles Lebeau.
\newblock Dispersion for the wave and the {S}chr\"{o}dinger equations outside
  strictly convex obstacles and counterexamples.
\newblock {\em C. R. Math. Acad. Sci. Paris}, 355(7):774--779, 2017.

\bibitem{Annals}
Oana Ivanovici, Gilles Lebeau, and Fabrice Planchon.
\newblock Dispersion for the wave equation inside strictly convex domains {I}:
  the {F}riedlander model case.
\newblock {\em Ann. of Math. (2)}, 180(1):323--380, 2014.

\bibitem{ILPTunisie}
Oana Ivanovici, Gilles Lebeau, and Fabrice Planchon.
\newblock Estimations de {S}trichartz pour l'\'{e}quation des ondes dans un
  domaine strictement convexe.
\newblock In {\em P{DE}'s, dispersion, scattering theory and control theory},
  volume~30 of {\em S\'{e}min. Congr.}, pages 69--79. Soc. Math. France, Paris,
  2017.

\bibitem{ILP4}
Oana Ivanovici, Gilles Lebeau, and Fabrice Planchon.
\newblock New counterexamples to {Strichartz} estimates for the wave equation
  on a $2$d model convex domain.
\newblock {\em Journal de l{\textquoteright}\'Ecole polytechnique {\textemdash}
  Math\'ematiques}, 8:1133--1157, 2021.

\bibitem{smso95}
Hart~F. Smith and Christopher~D. Sogge.
\newblock On the critical semilinear wave equation outside convex obstacles.
\newblock {\em J. Amer. Math. Soc.}, 8(4):879--916, 1995.

\bibitem{smso06}
Hart~F. Smith and Christopher~D. Sogge.
\newblock On the {$L\sp p$} norm of spectral clusters for compact manifolds
  with boundary.
\newblock {\em Acta Math.}, 198(1):107--153, 2007.

\bibitem{tat02}
Daniel Tataru.
\newblock Strichartz estimates for second order hyperbolic operators with
  nonsmooth coefficients. {III}.
\newblock {\em J. Amer. Math. Soc.}, 15(2):419--442 (electronic), 2002.

\end{thebibliography}
\end{document}